\documentclass{tac}

\usepackage[english]{babel}
\usepackage{amsfonts, amssymb, amsmath}
\usepackage{mathtools} % For the barred arrow.
\usepackage{enumitem}
\usepackage{tensor} % For the pullback notation.

\usepackage{tikz} % Drawing pictures.
\usetikzlibrary{decorations.markings, matrix, arrows, external, fit, calc}

\usepackage{hyperref}
\hypersetup{pdfauthor={Seerp Roald Koudenburg}, pdftitle={On pointwise Kan extensions in double categories}}

\tikzstyle{map} = [->, font=\scriptsize]
\tikzstyle{linj} = [left hook->, font=\scriptsize]
\tikzstyle{rinj} = [right hook->, font=\scriptsize]
\tikzstyle{sur} = [->>, font=\scriptsize]
\tikzstyle{cell} = [double,double equal sign distance,-implies, shorten >= 3.5pt, shorten <= 3.5pt, font=\scriptsize]
\tikzstyle{cell125} = [cell, shorten >= 1pt, shorten <= 1pt]
\tikzstyle{eq} = [double,double equal sign distance]

\tikzstyle{iso} = [above, sloped, inner sep=1.5pt]
\tikzstyle{nat} = [above, sloped, inner sep=2pt]
\tikzstyle{desc} = [fill=white, inner sep=2pt]
\tikzstyle{small} = [font=\scriptsize]

\tikzstyle{textbaseline} = [baseline=-2.8pt]
\tikzstyle{barred} = [decoration={markings, mark=at position 0.5 with {\draw[-] (0,-1.5pt) -- (0,1.5pt);}}, postaction ={decorate}]

\tikzstyle{math} = [matrix of math nodes, row sep=1em, column sep=1em, text height=1.5ex, text depth=0.25ex, nodes in empty cells]
\tikzstyle{math125em} = [math, row sep=1.25em, column sep=1.25em]
\tikzstyle{math175em} = [math, row sep=1.75em, column sep=1.75em]
\tikzstyle{math2em} = [math, row sep=2em, column sep=2em]
\tikzstyle{tab} = [math, row sep=1.75em, column sep=1.25em]

\makeatletter
\def\slashedarrowfill@#1#2#3#4#5{%
  $\m@th\thickmuskip0mu\medmuskip\thickmuskip\thinmuskip\thickmuskip
   \relax#5#1\mkern-7mu%
   \cleaders\hbox{$#5\mkern-2mu#2\mkern-2mu$}\hfill
   \mathclap{#3}\mathclap{#2}%
   \cleaders\hbox{$#5\mkern-2mu#2\mkern-2mu$}\hfill
   \mkern-7mu#4$%
}
\def\rightslashedarrowfill@{%
  \slashedarrowfill@\relbar\relbar\mapstochar\rightarrow}
\newcommand\xslashedrightarrow[2][]{%
  \ext@arrow 0055{\rightslashedarrowfill@}{#1}{#2}}
\makeatother

\def\slashedrightarrow{\xslashedrightarrow{}}

\providecommand{\oc}{\strut\textup{opcart}} 
\providecommand{\cc}{\strut\textup{cart}}
\providecommand{\tc}{\strut\textup{tab}}

\newtheoremrm{definition}{Definition}

\providecommand{\corref}[1]{Corollary~\ref{#1}}
\providecommand{\defref}[1]{Definition~\ref{#1}}
\providecommand{\exref}[1]{Example~\ref{#1}}

\providecommand{\lemref}[1]{Lemma~\ref{#1}}
\providecommand{\propref}[1]{Proposition~\ref{#1}}

\providecommand{\of}{\circ}
\providecommand{\iso}{\cong}

\providecommand{\brar}{\slashedrightarrow}
\providecommand{\xrar}[1]{\xrightarrow{#1}}
\providecommand{\xlar}[1]{\xleftarrow{#1}}
\providecommand{\xsrar}[1]{\xslashedrightarrow{#1}}
\providecommand{\Rar}{\Rightarrow}
\providecommand{\xRar}[1]{\xRightarrow{#1}}

\providecommand{\eps}{\varepsilon}

\DeclareMathOperator{\dash}{--}
\providecommand{\ndash}{\nobreakdash-}

\providecommand{\tens}{\otimes}

\providecommand{\brks}[1]{\lbrack #1 \rbrack}
\providecommand{\bigbrks}[1]{\bigl\lbrack #1 \bigr\rbrack}

\providecommand{\pars}[1]{\left(#1\right)}
\providecommand{\bigpars}[1]{\bigl(#1\bigr)}

\providecommand{\angles}[1]{\langle#1\rangle}

\providecommand{\gen}[1]{\angles{#1}}

\providecommand{\ls}[2]{{}_{#1} #2} %left subscript

\providecommand{\act}[2]{#1 \cdot #2}

\providecommand{\pb}[2]{\tensor[_{#1}]{\times}{_{#2}}}

\providecommand{\natarrow}{\Rightarrow}

\providecommand{\map}[3]{#1\colon#2\to#3}

\providecommand{\nat}[3]{#1\colon#2\natarrow#3}
\providecommand{\cell}[3]{#1\colon#2\Rightarrow#3}

\providecommand{\hmap}[3]{#1\colon#2\slashedrightarrow#3}

\providecommand{\inv}[1]{{#1}^{-1}}

\DeclareMathOperator{\id}{id}

\providecommand{\ladj}{\dashv}

\providecommand{\op}[1]{#1^\textup{op}}
\providecommand{\co}[1]{#1^\textup{co}}

\providecommand{\catvar}[1]{\mathcal{#1}}

\providecommand{\2}{\mathsf 2}
\providecommand{\C}{\catvar C}
\providecommand{\D}{\catvar D}
\providecommand{\E}{\catvar E}

\providecommand{\K}{\catvar K}

\providecommand{\V}{\catvar V}

\providecommand{\Set}{\mathsf{Set}}
\providecommand{\Cat}{\mathsf{Cat}}
\providecommand{\enCat}[1]{#1\text-\Cat}

\providecommand{\inhom}[1]{\brks{#1}}

\providecommand{\Span}[1]{\mathsf{Span}(#1)}
\providecommand{\Prof}{\mathsf{Prof}}
\providecommand{\enProf}[1]{#1\text-\Prof}
\providecommand{\inProf}[1]{\Prof(#1)}

\providecommand{\Alg}[2]{\mathsf{Alg}_\textup{#1}\pars{#2}}

\providecommand{\hc}{\odot}

\providecommand{\rhom}{\triangleright}

\providecommand{\tab}[1]{\gen{#1}}

\title{On pointwise Kan extensions in double categories}
\author{Seerp Roald Koudenburg}
\address{Cortenhoeve 14\\2411 JM Bodegraven\\The Netherlands}
\eaddress{roaldkoudenburg@gmail.com}
\copyrightyear{2014}
\thanks{Many of the results in this paper first appeared as part of my PhD thesis ``Algebraic weighted colimits'' that was written under the guidance of Simon Willerton. I would like to thank Simon for his advice and encouragement. Also I thank the anonymous referee for helpful suggestions, and the University of Sheffield for its financial support of my PhD studies.}

\keywords{double category, equipment, pointwise Kan extension, exact cell, tabulation}
\amsclass{18D05, 18A40, 18D20}

\begin{document}
\maketitle
\begin{abstract}
	In this paper we consider a notion of pointwise Kan extension in double categories that naturally generalises Dubuc's notion of pointwise Kan extension along enriched functors. We show that, when considered in equipments that admit opcartesian tabulations, it generalises Street's notion of pointwise Kan extension in $2$-categories.
\end{abstract}
  
	\section*{Introduction}
	A useful construction in classical category theory is that of right Kan extension along functors and, dually, that of left Kan extension along functors. Many important notions, including that of limit and right adjoint functor, can be regarded as right Kan extensions. On the other hand right Kan extensions can often be constructed out of limits; such Kan extensions are called pointwise. It is this notion of Kan extension that was extended to more general settings, firstly by Dubuc in \cite{Dubuc70}, to a notion of pointwise Kan extension along $\V$-functors, between categories enriched in some suitable category $\V$, and later by Street in \cite{Street74}, to a notion of pointwise Kan extension along morphisms in any $2$-category.
	
	It is unfortunate that Street's notion, when considered in the $2$-category $\enCat\V$ of $\V$\ndash enriched categories, does not agree with Dubuc's notion of pointwise Kan extension, but is stronger in general. In this paper we show that by moving from $2$-categories to double categories it is possible to unify Dubuc's and Street's notion of pointwise Kan extension.
	
	In \S1 we recall the notion of double category, which generalises that of $2$-category by considering, instead of a single type, two types of morphism. For example one can consider both ring homomorphisms and bimodules between rings. One type of morphism is drawn vertically and the other horizontally so that cells in a double category, which have both a horizontal and vertical morphism as source and as target, are shaped like squares. Every double category $\K$ contains a $2$-category $V(\K)$ consisting of the objects and vertical morphisms of $\K$, as well as cells whose horizontal source and target are identities.
	
	In \S2 we recall the notion of equipment. An equipment is a double category in which horizontal morphisms can both be restricted and extended along vertical morphisms. Cells defining extensions will be especially useful in regard to Kan extensions; such cells are called opcartesian.
	
	In \S3 several notions of Kan extension in a double category are recalled. Among those is that of pointwise Kan extension in `closed' equipments which, in the closed equipment of enriched $\V$-categories, generalises Dubuc's notion. This notion we extend to arbitrary double categories.
	
	In \S4 we consider cells in double categories under whose precomposition the class of cells defining pointwise Kan extensions is closed. Because such cells generalise the notion of an exact square of functors, we call them exact cells.
	
	In \S5 we recall Street's definition of pointwise Kan extension in a $2$-category, which uses comma objects to define such Kan extensions in terms of ordinary Kan extensions. The notion of comma object in $2$-categories generalises to that of tabulation in equipments, and our main result shows that in an equipment that has all opcartesian tabulations, pointwise Kan extensions can be defined in terms of ordinary Kan extensions, analogous to Street's definition. From this it follows that, for such an equipment $\K$, the pointwise Kan extensions in its vertical $2$-category $V(\K)$, in the sense of Street, can be regarded as pointwise Kan extensions in $\K$, in our sense. We close this paper by showing that the equipment of categories internal to some suitable category $\E$ has all opcartesian tabulations, strengthening a result of Betti given in \cite{Betti96}.
	
	In the forthcoming \cite{Koudenburg14b} we consider conditions ensuring that algebraic Kan extensions can be lifted along the forgetful double functors $\Alg wT \to \K$. Here $T$ denotes a double monad on a double category $\K$, and $\Alg wT$ is the double category of $T$-algebras, weak vertical $T$-morphisms and horizontal $T$-morphisms, where `weak' means either `lax', `colax' or `pseudo'.
	
	\section{Double categories}
	We start by recalling the notion of double category. References for double categories include \cite{Grandis-Pare99} and \cite{Shulman08}.
  
	For many mathematical objects there is not one but two natural notions of morphism: besides functions of sets $X \to Y$ one can also consider relations $R \subseteq X \times Y$ as morphisms $X \brar Y$, besides ring homomorphisms $A \to B$ one can also consider $(A, B)$-bimodules as morphisms $A \brar B$, and besides functors of categories $\C \to \D$ also `profunctors' $\op\C \times \D \to \Set$ can be considered as morphisms $\C \brar \D$. In each case the second type of morphism can be composed: relations are composed as usual, bimodules are composed using tensor products and profunctors by using `coends'. It is worth noticing that in the latter two cases a choice of tensor product or coend has to be made for each composite, so that composition of bimodules and that of profunctors are not strictly associative.
	
	In each of the previous examples the interaction between the two types of morphism can be described by $2$\ndash dimensional cells that are shaped like squares, as drawn below. For example if $f$ and $g$ are ring homomorphisms, $J$ is an $(A, B)$-bimodule and $K$ is a $(C, D)$-bimodule, then the square $\phi$ below depicts an $(f, g)$-bilinear map $J \to K$---that is a homomorphism $\map\phi JK$ of abelian groups such that $\act{fa}{\phi x} = \phi(\act ax)$ and $\act{\phi x}{gb} = \phi(\act xb)$.
  \begin{displaymath}
  	\begin{tikzpicture}
  		\matrix(m)[math175em]{A & B \\ C & D. \\};
  		\path[map]  (m-1-1) edge[barred] node[above] {$J$} (m-1-2)
													edge node[left] {$f$} (m-2-1)
									(m-1-2) edge node[right] {$g$} (m-2-2)
									(m-2-1) edge[barred] node[below] {$K$} (m-2-2);
			\path[transform canvas={shift={($(m-1-2)!(0,0)!(m-2-2)$)}}] (m-1-1) edge[cell] node[right] {$\phi$} (m-2-1);
		\end{tikzpicture}
	\end{displaymath}
  
	The notion of double category describes situations as above: generalising the notion of a $2$-category, its data contains not one but two types of morphism while its cells are square like. Formally a double category is defined as a weakly internal category in the $2$-category $\Cat$ of categories, functors and natural transformations, as follows; we shall give an elementary description afterwards.
	\begin{definition} \label{double category}
		A \emph{double category} $\K$ is given as follows.
		\begin{displaymath}
			\begin{tikzpicture}[baseline]
    		\matrix(m)[math, column sep=2em]{\K_1 \pb RL \K_1 & \K_1 & \K_0 \\};
    		\path[map]  (m-1-1) edge node[above] {$\hc$} (m-1-2)
        		        (m-1-2) edge[transform canvas={yshift=5.5pt}] node[above] {$L$} (m-1-3)
        		                edge[transform canvas={yshift=-5.5pt}] node[below] {$R$} (m-1-3)
        		        (m-1-3) edge[transform canvas={yshift=-2.5pt}] node[above=-3pt] {$1$} (m-1-2);
  		\end{tikzpicture} \qquad\qquad\qquad \begin{tikzpicture}[baseline]
    		\matrix(m)[math2em]
    			{ \K_1 \pb RL \K_1 & \K_1 \\
      		\K_1 & \K_0 \\ };
    			\path[map]  (m-1-1) edge node[above] {$\pi_1$} (m-1-2)
          			              edge node[left] {$\pi_2$} (m-2-1)
          			      (m-1-2) edge node[right] {$R$} (m-2-2)
          			      (m-2-1) edge node[below] {$L$} (m-2-2);
  		\end{tikzpicture}
		\end{displaymath}
		It consists of a diagram of functors as on the left above, where $\K_1 \pb RL \K_1$ is the pullback on the right, such that
		\begin{displaymath}
  		L \of \hc = L \of \pi_1, \qquad R \of \hc = R \of \pi_2, \qquad \text{and} \qquad L \of 1 = \id = R \of 1,
  	\end{displaymath}
		together with natural isomorphisms
		\begin{displaymath}
			\mathfrak a\colon (J \hc H) \hc K \iso J \hc (H \hc K), \quad \mathfrak l\colon 1_A \hc M \iso M \quad \text{and} \quad \mathfrak r\colon M \hc 1_B \iso M,
		\end{displaymath}
		where $(J, H, K) \in \K_1 \pb RL \K_1 \pb RL \K_1$ and $M \in \K_1$ with $LM = A$ and $RM = B$. The isomorphisms $\mathfrak a$, $\mathfrak l$ and $\mathfrak r$ are required to satisfy the usual coherence axioms for a monoidal category or bicategory (see e.g.\ Section VII.1 of \cite{MacLane98}), while their images under both $R$ and $L$ must be identities.
	\end{definition}
	
	The objects of $\K_0$ are called \emph{objects} of $\K$ while the morphisms $\map fAC$ of $\K_0$ are called \emph{vertical morphisms} of $\K$. An object $J$ of $\K_1$ such that $LJ = A$ and $RJ = B$ is denoted by a barred arrow
\begin{displaymath}
  \hmap JAB
\end{displaymath}
  and called a \emph{horizontal morphism}. A morphism $\map\phi JK$ in $\K_1$, with \mbox{$L \phi = \map fAC$} and $R \phi = \map gBD$, is depicted
\begin{displaymath}
  \begin{tikzpicture}
    \matrix(m)[math175em]{A & B \\ C & D \\};
    \path[map]  (m-1-1) edge[barred] node[above] {$J$} (m-1-2)
                        edge node[left] {$f$} (m-2-1)
                (m-1-2) edge node[right] {$g$} (m-2-2)
                (m-2-1) edge[barred] node[below] {$K$} (m-2-2);
    \path[transform canvas={shift={($(m-1-2)!(0,0)!(m-2-2)$)}}] (m-1-1) edge[cell] node[right] {$\phi$} (m-2-1);
  \end{tikzpicture}
\end{displaymath}
  and called a \emph{cell}. We will call $J$ and $K$ the \emph{horizontal source} and \emph{target} of $\phi$, while we call $f$ and $g$ its \emph{vertical source} and \emph{target}. A cell whose vertical source and target are identities is called \emph{horizontal}.
  
  The composition of $\K_1$ allows us to vertically compose two cells $\phi$ and $\psi$, that share a common horizontal edge as on the left below, to form a new cell $\psi \of \phi$ whose sources and targets can be read off from the drawing of $\phi$ on top of $\psi$, as shown.
  \begin{equation} \label{interchange law}
  	\begin{tikzpicture}[textbaseline]
			\matrix(m)[math175em]{A & B \\ C & D \\ E & F \\};
			\path[map]	(m-1-1) edge[barred] node[above] {$J$} (m-1-2)
													edge node[left] {$f$} (m-2-1)
									(m-1-2) edge node[right] {$g$} (m-2-2)
									(m-2-1) edge[barred] node[below] {$K$} (m-2-2)
													edge node[left] {$h$} (m-3-1)
									(m-2-2) edge node[right] {$k$} (m-3-2)
									(m-3-1) edge[barred] node[below] {$L$} (m-3-2);
			\path[transform canvas={shift=(m-2-1)}]
									(m-1-2) edge[cell] node[right] {$\phi$} (m-2-2)
									(m-2-2) edge[cell, transform canvas={yshift=-3pt}] node[right] {$\psi$} (m-3-2);
		\end{tikzpicture} \qquad\qquad \begin{tikzpicture}[textbaseline]
			\matrix(m)[math175em]{A & B & E \\ C & D & F \\};
			\path[map]	(m-1-1) edge[barred] node[above] {$J$} (m-1-2)
													edge node[left] {$f$} (m-2-1)
									(m-1-2) edge[barred] node[above] {$H$} (m-1-3)
													edge node[right] {$g$} (m-2-2)
									(m-1-3) edge node[right] {$h$} (m-2-3)
									(m-2-1) edge[barred] node[below] {$K$} (m-2-2)
									(m-2-2) edge[barred] node[below] {$L$} (m-2-3);
			\path[transform canvas={shift=($(m-1-1)!0.5!(m-2-2)$)}]
									(m-1-2) edge[cell] node[right] {$\phi$} (m-2-2)
									(m-1-3) edge[cell] node[right] {$\chi$} (m-2-3);
		\end{tikzpicture} \qquad\qquad \begin{tikzpicture}[textbaseline]
			\matrix(m)[math175em]{\phantom{\text x} & \phantom{\text x} & \phantom{\text x} \\ \phantom{\text x} & \phantom{\text x} & \phantom{\text x} \\ \phantom{\text x} & \phantom{\text x} & \phantom{\text x} \\};
			\path[map]	(m-1-1) edge (m-1-2)
													edge (m-2-1)
									(m-1-2) edge (m-1-3)
													edge (m-2-2)
									(m-1-3) edge (m-2-3)
									(m-2-1) edge (m-2-2)
													edge (m-3-1)
									(m-2-2) edge (m-2-3)
													edge (m-3-2)
									(m-2-3) edge (m-3-3)
									(m-3-1) edge (m-3-2)
									(m-3-2) edge (m-3-3);
			\path[transform canvas={shift={($(m-2-1)!0.5!(m-2-2)$)}}] (m-1-2) edge[cell] node[right] {$\phi$} (m-2-2)
									(m-1-3) edge[cell] node[right] {$\chi$} (m-2-3)
									(m-2-2) edge[cell] node[right] {$\psi$} (m-3-2)
									(m-2-3) edge[cell] node[right] {$\xi$} (m-3-3);
		\end{tikzpicture}
  \end{equation}
  Two composable horizontal morphisms $A \xsrar J B \xsrar H E$ can be composed using the functor $\hc$, giving the horizontal composite $\hmap{J \hc H}AE$. Likewise a pair of cells $\phi$ and $\chi$, sharing a common vertical edge as in the middle above, can be horizontally composed, resulting in a cell $\phi \hc \chi$ whose sources and targets can be read off from the drawing of $\phi$ and $\chi$ side-by-side as shown. The functoriality of $\hc$ implies that, for a square of composable cells as on the right above, we have $(\psi \of \phi) \hc (\xi \of \chi) = (\psi \hc \xi) \of (\phi \hc \chi)$; this identity is called the \emph{interchange law}.
  
  Finally $\K$ is equipped with a horizontal unit $\hmap{1_A}AA$ for each object $A$, as well as a horizontal unit cell $1_f$, as on the left below, for each vertical morphism $\map fAC$. Unlike vertical composition of vertical morphisms and cells, which is strictly associative and unital, horizontal composition of horizontal morphisms and cells is only associative up to the invertible horizontal cells $\mathfrak a$ below, where $A \xsrar J B \xsrar H C \xsrar K D$, and unital up to the invertible horizontal cells $\mathfrak l$ and $\mathfrak r$ below, where $\hmap MAB$. The cells $\mathfrak a$ are called \emph{associators}, while the cells $\mathfrak l$ and $\mathfrak r$ are called \emph{unitors}. A double category with identities as associators and unitors is called \emph{strict}.
  \begin{displaymath}
  	\begin{tikzpicture}[baseline]
			\matrix(m)[math175em]{A & A \\ C & C \\};
			\path[map]  (m-1-1) edge[barred] node[above] {$1_A$} (m-1-2)
													edge node[left] {$f$} (m-2-1)
									(m-1-2) edge node[right] {$f$} (m-2-2)
									(m-2-1) edge[barred] node[below] {$1_C$} (m-2-2);
			\path[transform canvas={shift={($(m-1-2)!(0,0)!(m-2-2)$)}}] (m-1-1) edge[cell] node[right] {$1_f$} (m-2-1);
		\end{tikzpicture} \qquad \begin{tikzpicture}[baseline]
			\matrix(m)[math175em, column sep=4em]{A & D \\ A & D \\};
			\path[map]  (m-1-1) edge[barred] node[above] {$(J \hc H) \hc K$} (m-1-2)
													edge node[left] {$\id_A$} (m-2-1)
									(m-1-2) edge node[right] {$\id_D$} (m-2-2)
									(m-2-1) edge[barred] node[below] {$J \hc (H \hc K)$} (m-2-2);
			\path[transform canvas={shift={($(m-1-2)!(0,0)!(m-2-2)$)}}] (m-1-1) edge[cell] node[right] {$\mathfrak a$} (m-2-1);
		\end{tikzpicture} \qquad \begin{tikzpicture}[baseline]
			\matrix(m)[math175em]{A & B \\ A & B \\};
			\path[map]  (m-1-1) edge[barred] node[above] {$1_A \hc M$} (m-1-2)
													edge node[left] {$\id_A$} (m-2-1)
									(m-1-2) edge node[right] {$\id_B$} (m-2-2)
									(m-2-1) edge[barred] node[below] {$M$} (m-2-2);
			\path[transform canvas={shift={($(m-1-2)!(0,0)!(m-2-2)$)}}] (m-1-1) edge[cell] node[right] {$\mathfrak l$} (m-2-1);
		\end{tikzpicture} \qquad \begin{tikzpicture}[baseline]
			\matrix(m)[math175em]{A & B \\ A & B \\};
			\path[map]  (m-1-1) edge[barred] node[above] {$M \hc 1_B$} (m-1-2)
													edge node[left] {$\id_A$} (m-2-1)
									(m-1-2) edge node[right] {$\id_B$} (m-2-2)
									(m-2-1) edge[barred] node[below] {$M$} (m-2-2);
			\path[transform canvas={shift={($(m-1-2)!(0,0)!(m-2-2)$)}}] (m-1-1) edge[cell] node[right] {$\mathfrak r$} (m-2-1);
		\end{tikzpicture}
  \end{displaymath}
  
  Cells whose horizontal source and target are units are called \emph{vertical}. To make our drawings of cells more readable we will depict both vertical and horizontal identities simply as $\begin{tikzpicture}[baseline=- 4pt] \path (0.2,0) edge[eq] (0.7,0); \node at (0,0) {$A$}; \node at (0.9,0) {$A$};\end{tikzpicture}$. For example vertical and horizontal cells will be drawn as
	\begin{displaymath}
		\begin{tikzpicture}[textbaseline]
  		\matrix(m)[math175em]{A & A \\ C & C \\};
  		\path[map]  (m-1-1) edge node[left] {$f$} (m-2-1)
									(m-1-2) edge node[right] {$g$} (m-2-2);
			\path				(m-1-1) edge[eq] (m-1-2)
									(m-2-1) edge[eq] (m-2-2);
			\path[transform canvas={shift={($(m-1-2)!(0,0)!(m-2-2)$)}}] (m-1-1) edge[cell] node[right] {$\phi$} (m-2-1);
		\end{tikzpicture}
		\qquad \text{and} \qquad
		\begin{tikzpicture}[textbaseline]
  		\matrix(m)[math175em]{A & B \\ A & B \\};
  		\path[map]  (m-1-1) edge[barred] node[above] {$J$} (m-1-2)
									(m-2-1) edge[barred] node[below] {$K$} (m-2-2);
			\path				(m-1-1) edge[eq] (m-2-1)
									(m-1-2) edge[eq] (m-2-2);
			\path[transform canvas={shift={($(m-1-2)!(0,0)!(m-2-2)$)}}] (m-1-1) edge[cell] node[right] {$\phi$} (m-2-1);
		\end{tikzpicture}
	\end{displaymath}
	respectively; also vertical cells will be often denoted as $\cell\phi fg$ instead of the less informative $\cell\phi{1_A}{1_C}$.
	
	We will often describe compositions of cells by drawing `grids' like that on the right of \eqref{interchange law} above, in varying degrees of detail. In particular we shall often not denote identity cells in such grids but simply leave them empty. In horizontal compositions of more than two horizontal morphisms or cells we will leave out bracketings and assume a right to left bracketing, for example we will simply write $J_1 \hc \dotsb \hc J_n$ for $J_1 \hc \bigpars{\dotsb \hc (J_{n-1} \hc J_n) \dotsb }$. Moreover, when writing down compositions or drawing grids containing horizontal composites of more than two cells, we will leave out the (inverses of) associators and unitors of $\K$. This is possible because their coherence implies that any two ways of filling in the left out (inverses of) associators and unitors result in the same composite. In fact Grandis and Par\'e show in their Theorem 7.5 of \cite{Grandis-Pare99} that every double category is equivalent to a strict double category, whose associators and unitors are identity cells.
	
	\subsection{Examples.}
	In the remainder of this section we describe in detail the three double categories that we will use throughout: the archetypical double category $\Prof$ of categories, functors, profunctors and natural transformations, as well as two generalisations $\enProf \V$ and $\inProf\E$, the first obtained by enriching over a suitable symmetric monoidal category $\V$ and the second by internalising in a suitable category $\E$ with pullbacks. Other examples can be found in Section 3 of \cite{Grandis-Pare99} and Section 2 of \cite{Shulman08}.
	
	\begin{example} \label{example: unenriched profunctors}
		The double category $\Prof$ of profunctors is given as follows. Its objects and vertical morphisms, forming the category $\Prof_0$, are small categories and functors. Its horizontal morphisms $\hmap JAB$ are \emph{profunctors}, that is functors $\map J{\op A \times B}\Set$. We think of the elements of $J(a, b)$ as morphisms, and denote them $\map jab$. Likewise we think of the actions of $A$ and $B$ on $J$ as compositions; hence, for morphisms $\map u{a'}a$ and $\map vb{b'}$, we write $v \of j \of u = J(u, v)(j)$. A cell
		\begin{displaymath}
			\begin{tikzpicture}
				\matrix(m)[math175em]{A & B \\ C & D \\};
				\path[map]  (m-1-1) edge[barred] node[above] {$J$} (m-1-2)
														edge node[left] {$f$} (m-2-1)
										(m-1-2) edge node[right] {$g$} (m-2-2)
										(m-2-1) edge[barred] node[below] {$K$} (m-2-2);
				\path[transform canvas={shift={($(m-1-2)!(0,0)!(m-2-2)$)}}] (m-1-1) edge[cell] node[right] {$\phi$} (m-2-1);
			\end{tikzpicture}
		\end{displaymath}
		is a natural transformation $\nat\phi J{K(f, g)}$, where $\hmap{K(f,g) = K \of (\op f \times g)}AB$. Such transformations clearly vertically compose so that they form a category $\Prof_1$, with profunctors as objects.
	
		The horizontal composition $J \hc H$ of a pair of composable profunctors $\hmap JAB$ and \mbox{$\hmap HBE$} is defined on objects by the coends $(J \hc H)(a, e) = \int^B J(a, \dash) \times H(\dash, e)$. A coend is a type of colimit, see Section IX.6 of \cite{MacLane98}; in fact $(J \hc H)(a, e)$ coincides with the coequaliser of
		\begin{equation} \label{composition of unenriched profunctors}
  	  \coprod\limits_{b_1, b_2 \in B} J(a, b_1) \times B(b_1, b_2) \times H(b_2, e) \rightrightarrows \coprod\limits_{b \in B} J(a, b) \times H(b, e),
  	\end{equation}
  	where the maps let $B(b_1, b_2)$ act on $J(a, b_1)$ and $H(b_2, e)$ respectively. This can be thought of as a multi-object (and cartesian) variant of the tensor product of bimodules. In detail, $(J \hc H)(a, e)$ forms the set of equivalence classes of pairs $(\map jab, \map hbe)$, where $j \in J$ and $h \in H$, under the equivalence relation generated by $(a \xrar{j'} b_1 \xrar v b_2, b_2 \xrar{h'} e) \sim (a \xrar{j'} b_1, b_1 \xrar v b_2 \xrar{h'} e)$, for all $j' \in J$, $v \in B$ and $h' \in H$.
  	
  	The definition of $J \hc H$ on objects is extended to a functor $\op A \times E \to \Set$ by using the universal property of coends: the image of $\map{(u,w)}{(a,e)}{(a',e')}$ is taken to be the unique factorisation of
  	\begin{displaymath}
  		\coprod_{b \in B} J(a', b) \times H(b, e) \xrar{\coprod J(u, \id) \times H(\id, w)} \coprod_{b \in B} J(a, b) \times H(b, e') \to (J \hc H)(a, e')
  	\end{displaymath}
  	through the universal map defining $(J \hc H)(a', e)$.
	  \begin{displaymath}
	  	\begin{tikzpicture}
				\matrix(m)[math175em]{A & B & E \\ C & D & F \\};
				\path[map]	(m-1-1) edge[barred] node[above] {$J$} (m-1-2)
														edge node[left] {$f$} (m-2-1)
										(m-1-2) edge[barred] node[above] {$H$} (m-1-3)
														edge node[right] {$g$} (m-2-2)
										(m-1-3) edge node[right] {$h$} (m-2-3)
										(m-2-1) edge[barred] node[below] {$K$} (m-2-2)
										(m-2-2) edge[barred] node[below] {$L$} (m-2-3);
				\path[transform canvas={shift=($(m-1-1)!0.5!(m-2-2)$)}]
										(m-1-2) edge[cell] node[right] {$\phi$} (m-2-2)
										(m-1-3) edge[cell] node[right] {$\chi$} (m-2-3);
			\end{tikzpicture}
	  \end{displaymath}
	   Similarly the horizontal composite $\phi \hc \chi$ of the two cells above is given by the unique factorisation of the maps
	   \begin{multline*}
	   	\coprod_{b \in B} J(a, b) \times H(b,e) \xrar{\coprod \phi_{(a, b)} \times \chi_{(b,e)}} \coprod_{b \in B} K(fa, gb) \times L(gb, he) \\
	   		\to \coprod_{d \in D} K(fa, d) \times L(d, he) \to (K \hc L)(fa, he)
	   \end{multline*}
	   through the universal maps defining $(J \hc H)(a, e)$.
	   
	   The associators $\mathfrak a \colon (J \hc H) \hc K \iso J \hc (H \hc K)$ are obtained by using the fact that the cartesian product of $\Set$ preserves colimits in both variables, together with `Fubini's theorem for coends': for any functor $\map S{\op A \times A \times \op B \times B}\Set$ there exist canonical isomorphisms $\int^B \int^A S \iso \int^{A \times B} S \iso \int^A \int^B S$; see Section IX.8 of \cite{MacLane98} for the dual result for ends. Finally the unit profunctor $1_B$ on a category $B$ is given by its hom-objects $1_B(b_1, b_2) = B(b_1, b_2)$: one checks that the map $\coprod_b J(a, b) \times B(b, b') \to J(a, b')$ given by the action of $B$ on $J$ defines $J(a, b')$ as the coequaliser of \eqref{composition of unenriched profunctors}, in case $H = 1_B$, which induces the unitor $\mathfrak r\colon J \hc 1_B \iso J$; the other unitor $\mathfrak l\colon 1_A \hc J \iso J$ is obtained likewise. The horizontal unit $\cell{1_f}{1_A}{1_C}$ of a functor $\map fAC$ is simply given by the actions $A(a_1, a_2) \to C(fa_1, fa_2)$ of $f$ on the hom-sets.
	\end{example}
	
	\begin{example} \label{example: enriched profunctors}
		For any closed symmetric monoidal category $\V$ that is cocomplete there exists a $\V$-enriched variant of the double category $\Prof$, which is denoted $\enProf\V$. It consists of small $\V$-categories, $\V$-functors and \emph{$\V$-profunctors} $\hmap JAB$, that is $\V$-functors $\map J{\op A \tens B}\V$, while a cell
		\begin{displaymath}
			\begin{tikzpicture}
				\matrix(m)[math175em]{A & B \\ C & D \\};
				\path[map]  (m-1-1) edge[barred] node[above] {$J$} (m-1-2)
														edge node[left] {$f$} (m-2-1)
										(m-1-2) edge node[right] {$g$} (m-2-2)
										(m-2-1) edge[barred] node[below] {$K$} (m-2-2);
				\path[transform canvas={shift={($(m-1-2)!(0,0)!(m-2-2)$)}}] (m-1-1) edge[cell] node[right] {$\phi$} (m-2-1);
			\end{tikzpicture}
		\end{displaymath}
		is a $\V$-natural transformation $\nat\phi J{K(f,g)}$, where $K(f, g) = K \of (\op f \tens g)$. The structure of a double category on $\enProf\V$ is given completely analogously to that on $\Prof$, by replacing in the definition of $\Prof$ every instance of $\Set$ by $\V$, and every cartesian product of sets by a tensor product of $\V$-objects. In particular the horizontal composite $J \hc H$ is given by the coends $(J \hc H)(a, e) = \int^B J(a, \dash) \tens H(\dash, e)$, which form coequalisers of the $\V$-maps
		\begin{equation} \label{composition of enriched profunctors}
	    \coprod\limits_{b_1, b_2 \in B} J(a, b_1) \tens B(b_1, b_2) \tens H(b_2, e) \rightrightarrows \coprod\limits_{b \in B} J(a, b) \tens H(b, e),
	  \end{equation}
	  that are induced by letting $B(b_1, b_2)$ act on $J(a, b_1)$ and $H(b_2, e)$ respectively.
	\end{example}
	
	Next we describe the double category of spans in a category $\E$ with pullbacks, which will be used in describing the double category $\inProf\E$ of categories, functors and profunctors internal in $\E$.
	\begin{example} \label{example:spans}
		For a category $\E$ that has pullbacks, the double category $\Span\E$ of \emph{spans in $\E$} is defined as follows. The objects and vertical morphisms of $\Span\E$ are those of $\E$, while a horizontal morphism $\hmap JAB$ is a span $A \xlar{d_0} J \xrar{d_1} B$ in $\E$. A cell $\phi$ as on the left below is a map $\map\phi JK$ in $\E$ such that the diagram in the middle commutes.
		\begin{displaymath}
  		\begin{tikzpicture}[baseline]
  		  \matrix(m)[math175em]{A & B \\ C & D \\};
  		  \path[map]  (m-1-1) edge[barred] node[above] {$J$} (m-1-2)
  		                      edge node[left] {$f$} (m-2-1)
  		              (m-1-2) edge node[right] {$g$} (m-2-2)
  		              (m-2-1) edge[barred] node[below] {$K$} (m-2-2);
  		  \path[transform canvas={shift={($(m-1-2)!(0,0)!(m-2-2)$)}}] (m-1-1) edge[cell] node[right] {$\phi$} (m-2-1);
  		\end{tikzpicture}	\qquad\qquad \begin{tikzpicture}[baseline]
				\matrix(m)[math175em, column sep=1.25em, row sep=0.3em]{& J & \\ A & & B \\ & K & \\ C & & D \\};
				\path[map]	(m-1-2)	edge node[above left] {$d_0$} (m-2-1)
														edge node[right] {$\phi$} (m-3-2)
														edge node[above right] {$d_1$} (m-2-3)
										(m-2-1) edge node[left] {$f$} (m-4-1)
										(m-2-3) edge node[right] {$g$} (m-4-3)
										(m-3-2) edge node[below=4pt, right=-2pt] {$d_0$} (m-4-1)
														edge node[below=4pt, left=-5pt] {$d_1$} (m-4-3);
			\end{tikzpicture} \qquad\qquad \begin{tikzpicture}[baseline]
				\matrix(m)[math175em, row sep=1.2em, column sep=1.2em]
				{	& & \phantom B & & \\
					& J & & H & \\
					A & & B & & E \\ };
				\path[map]	(m-1-3) edge (m-2-2)
														edge (m-2-4)
										(m-2-2) edge node[above left] {$d_0$} (m-3-1)
														edge node[below left] {$d_1$} (m-3-3)
										(m-2-4) edge node[below right] {$d_0$} (m-3-3)
														edge node[above right] {$d_1$} (m-3-5);
				\draw (m-1-3) node {$J \times_B H$};
			\end{tikzpicture}
		\end{displaymath}
		Given spans $\hmap JAB$ and $\hmap HBE$, their composition $J \hc H$ is given by the usual composition of spans: after choosing a pullback $J \times_B H$ of $d_1$ and $d_0$ as on the right above, it is taken to consist of the two sides in this diagram. That this composition is associative and unital up to coherent isomorphisms, with spans of the form $A \xlar{\id} A \xrar{\id} A$ as units $1_A$, follows from the universality of pullbacks; horizontal composition of cells is also given using this universality.
	\end{example}
	
	Having recalled the notion of span in $\E$ we can now describe $\inProf\E$.
	\begin{example} \label{double category of internal profunctors}
		Let $\E$ be a category that has pullbacks and coequalisers, such that the coequalisers are preserved by pullback. The double category $\inProf\E$ of categories, functors and profunctors internal in $\E$ is given as follows. An \emph{internal category} $A$ in $\E$ consists of a triple $A = (A, m, e)$ where $A = \bigbrks{A_0 \xlar{d_0} A \xrar{d_1} A_0}$ is a span in $\E$, while $\cell m{A \hc A}A$ and $\cell e{1_{A_0}}A$ are horizontal cells in $\Span\E$, the \emph{multiplication} and \emph{unit} of $A$, which satisfy evident associativity and unit laws; see Section XII.1 of \cite{MacLane98}. For example, a category internal in the category $\Cat$ of categories and functors is a strict double category; compare \defref{double category}.
		
		An \emph{internal functor} $\map fAC$ in $\E$ consists of a cell $f$ in $\Span\E$ as on the left below, that is compatible with the multiplication and unit of $A$ and $B$. An \emph{internal profunctor} $\hmap JAB$ in $\E$ is a span $J = \bigbrks{A_0 \xlar{d_0} J \xrar{d_1} B_0}$ equipped with actions of $A$ and $B$, given by horizontal cells $\cell l{A \hc J}J$ and $\cell r{J \hc B}J$ in $\Span\E$, that are compatible with the multiplication and unit of $A$ and $B$, and satisfy a mixed associativity law; see Section~3 of \cite{Betti96}.
		\begin{displaymath}
			\begin{tikzpicture}[baseline]
  		  \matrix(m)[math175em]{A_0 & A_0 \\ C_0 & C_0 \\};
  		  \path[map]  (m-1-1) edge[barred] node[above] {$A$} (m-1-2)
  		                      edge node[left] {$f_0$} (m-2-1)
  		              (m-1-2) edge node[right] {$f_0$} (m-2-2)
  		              (m-2-1) edge[barred] node[below] {$C$} (m-2-2);
  		  \path[transform canvas={shift={($(m-1-2)!(0,0)!(m-2-2)$)}}] (m-1-1) edge[cell] node[right] {$f$} (m-2-1);
  		\end{tikzpicture} \qquad\qquad \begin{tikzpicture}[baseline]
  		  \matrix(m)[math175em]{A & B \\ C & D \\};
  		  \path[map]  (m-1-1) edge[barred] node[above] {$J$} (m-1-2)
  		                      edge node[left] {$f$} (m-2-1)
  		              (m-1-2) edge node[right] {$g$} (m-2-2)
  		              (m-2-1) edge[barred] node[below] {$K$} (m-2-2);
  		  \path[transform canvas={shift={($(m-1-2)!(0,0)!(m-2-2)$)}}] (m-1-1) edge[cell] node[right] {$\phi$} (m-2-1);
  		\end{tikzpicture} \qquad \qquad \begin{tikzpicture}[baseline]
  		  \matrix(m)[math175em]{A_0 & B_0 \\ C_0 & D_0 \\};
  		  \path[map]  (m-1-1) edge[barred] node[above] {$J$} (m-1-2)
  		                      edge node[left] {$f_0$} (m-2-1)
  		              (m-1-2) edge node[right] {$g_0$} (m-2-2)
  		              (m-2-1) edge[barred] node[below] {$K$} (m-2-2);
  		  \path[transform canvas={shift={($(m-1-2)!(0,0)!(m-2-2)$)}}] (m-1-1) edge[cell] node[right] {$\phi$} (m-2-1);
  		\end{tikzpicture}
		\end{displaymath}
		Finally an \emph{internal transformation} $\phi$ of internal profunctors, as in the middle above, is given by a cell $\phi$ in $\Span\E$ as on the right, that is compatible with the actions in the sense that the following diagrams commute in $\E$.
		\begin{equation} \label{naturality of internal transformation}
			\begin{tikzpicture}[textbaseline]
				\matrix(m)[math2em, column sep=2.5em]{A \times_{A_0} J & C \times_{C_0} K \\ J & K \\};
				\path[map]	(m-1-1) edge node[above] {$f \times_{f_0} \phi$} (m-1-2)
														edge node[left] {$l$} (m-2-1)
										(m-1-2) edge node[right] {$l$} (m-2-2)
										(m-2-1) edge node[below] {$\phi$} (m-2-2);
			\end{tikzpicture} \qquad\qquad\quad \begin{tikzpicture}[textbaseline]
				\matrix(m)[math2em, column sep=2.5em]{J \times_{B_0} B & K \times_{D_0} D \\ J & K \\};
				\path[map]	(m-1-1) edge node[above] {$\phi \times_{g_0} g$} (m-1-2)
														edge node[left] {$r$} (m-2-1)
										(m-1-2) edge node[right] {$r$} (m-2-2)
										(m-2-1) edge node[below] {$\phi$} (m-2-2);
			\end{tikzpicture}
		\end{equation}
		
		Given internal profunctors $\hmap JAB$ and $\hmap HBE$, their horizontal composite $J \hc H$ is defined to be the coequaliser
		\begin{displaymath}
		J \times_{B_0} B \times_{B_0} H \rightrightarrows J \times_{B_0} H \to J \hc H
		\end{displaymath}
		of the $\E$-maps given by the actions of $B$ on $J$ and $H$. The internal profunctor structures on $J$ and $H$ make $J \hc H$ into an internal profunctor $A \brar E$, by using the universal property of coequalisers and the fact that the coequalisers are preserved by pullback. The same universal property allows us to define horizontal composites of internal transformations, and it is not hard to prove that the horizontal composition of internal profunctors above is associative up to invertible associators. Moreover, notice that any internal category $A$ can be regarded as an internal profunctor $\hmap{1_A}AA$, with both actions given by the multiplication of $A$; these form the units for horizontal composition, up to invertible unitors.
	\end{example}
	
	Having described several examples, we now consider some simple constructions on double categories. The first constructions are duals: like $2$-categories, double categories have both a vertical and horizontal dual as follows.
	\begin{definition} \label{definition: duals of double categories}
		Given a double category $\K$ we denote by $\op \K$ the double category that has the same objects and horizontal morphisms of $\K$, while it has a vertical morphism $\map{\op f}CA$ for each vertical morphism $\map fAC$ in $\K$, and a cell $\cell{\op\phi}KJ$, as on the left below, for each cell $\phi$ in $\K$ as in the middle. The structure making $\op \K$ into a double category is induced by that of $\K$.
		\begin{displaymath}
			\begin{tikzpicture}[baseline]
  		  \matrix(m)[math175em]{C & D \\ A & B \\};
  		  \path[map]  (m-1-1) edge[barred] node[above] {$K$} (m-1-2)
  		                      edge node[left] {$\op f$} (m-2-1)
  		              (m-1-2) edge node[right] {$\op g$} (m-2-2)
  		              (m-2-1) edge[barred] node[below] {$J$} (m-2-2);
  		  \path[transform canvas={shift={($(m-1-2)!(0,0)!(m-2-2)$)}, xshift=-4pt}] (m-1-1) edge[cell] node[right] {$\op\phi$} (m-2-1);
  		\end{tikzpicture} \qquad\qquad \begin{tikzpicture}[baseline]
  		  \matrix(m)[math175em]{A & B \\ C & D \\};
  		  \path[map]  (m-1-1) edge[barred] node[above] {$J$} (m-1-2)
  		                      edge node[left] {$f$} (m-2-1)
  		              (m-1-2) edge node[right] {$g$} (m-2-2)
  		              (m-2-1) edge[barred] node[below] {$K$} (m-2-2);
  		  \path[transform canvas={shift={($(m-1-2)!(0,0)!(m-2-2)$)}}] (m-1-1) edge[cell] node[right] {$\phi$} (m-2-1);
  		\end{tikzpicture} \qquad\qquad \begin{tikzpicture}[baseline]
  		  \matrix(m)[math175em]{B & A \\ D & C \\};
  		  \path[map]  (m-1-1) edge[barred] node[above] {$\co J$} (m-1-2)
  		                      edge node[left] {$g$} (m-2-1)
  		              (m-1-2) edge node[right] {$f$} (m-2-2)
  		              (m-2-1) edge[barred] node[below] {$\co K$} (m-2-2);
  		  \path[transform canvas={shift={($(m-1-2)!(0,0)!(m-2-2)$)}, xshift=-4pt}] (m-1-1) edge[cell] node[right] {$\co\phi$} (m-2-1);
  		\end{tikzpicture}
		\end{displaymath}
		
		Likewise we denote by $\co\K$ the double category whose objects and vertical morphisms are those of $\K$, that has a horizontal morphism $\hmap{\co J}BA$ for each $\hmap JAB$ in $\K$ and a cell $\cell{\co\phi}{\co J}{\co K}$, as on the right above, for each cell $\phi$ in $\K$ as in the middle. The structure on $\co \K$ is induced by that of $\K$. We call $\op\K$ the \emph{vertical dual} of $\K$, and $\co \K$ the \emph{horizontal dual}.
	\end{definition}
	
	Secondly each double category contains a `vertical $2$-category' and a `horizontal bicategory' as follows.
	\begin{definition} \label{2-category and bicategory in a double category}
		Let $\K$ be a double category. We denote by $V(\K)$ the $2$-category that consists of the objects, vertical morphisms and vertical cells of $\K$. The vertical composite $\psi\phi$ of $\cell\phi fg$ and $\cell\psi gh$ in $V(\K)$ is given by the horizontal composite $\phi \hc \psi$\footnote{As usual we suppress the unitors here. If $f$, $g$ and $h$ are vertical morphisms $A \to C$ then $\phi \hc \psi$ is short for the composite $1_A \xRar{\inv{\mathfrak l}} 1_A \hc 1_A \xRar{\phi \hc \psi} 1_C \hc 1_C \xRar{\mathfrak l} 1_C$.} in $\K$, while the horizontal composite $\chi \of \phi$ in $V(\K)$, of composable cells $\cell\phi fg$ and $\cell\chi kl$, is given by the vertical composite $\chi \of \phi$ in $\K$.
		
		Likewise the objects, horizontal morphisms and horizontal cells of $\K$ combine into a bicategory $H(\K)$ whose compositions, units and coherence cells are those of $\K$.
	\end{definition}
	In the case of $V(\K)$ we have to check that its vertical composition is strictly associative: this follows from the coherence axioms for $\K$.
	\begin{example}
		It is easy to construct an isomorphism $V(\Prof) \iso \Cat$. Indeed, given functors $f$ and $\map gAC$, a vertical cell $\cell\phi fg$ in $\Prof$, that is given by natural maps $\map{\phi_{(a_1, a_2)}}{A(a_1, a_2)}{C(fa_1, ga_2)}$ where $a_1, a_2 \in A$, corresponds under this isomorphism to the natural transformation $f \Rar g$ in $\Cat$ that has components $\map{\phi_{(a, a)}(\id_a)}{fa}{ga}$. In the same way $V(\enProf\V) \iso \enCat\V$, the $2$-category of small $\V$-categories, $\V$\ndash functors and $\V$\ndash natural transformations. The bicategory $H(\enProf \V)$ is that of small $\V$-categories, $\V$\ndash profunctors and $\V$-natural transformations.
	\end{example}
	\begin{example}
		Let $\E$ be a category with pullbacks. Given internal functors $f$ and \mbox{$\map gAC$}, an \emph{internal transformation} \mbox{$\nat\phi fg$} is given by a cell $\phi$ in $\Span\E$, as on the left below, that makes the diagram of $\E$-maps on the right commute; see Section 1 of \cite{Street74}.
		\begin{displaymath}
			\begin{tikzpicture}[baseline]
  		  \matrix(m)[math175em]{A_0 & A_0 \\ C_0 & C_0 \\};
  		  \path[map]  (m-1-1) edge node[left] {$f_0$} (m-2-1)
  		              (m-1-2) edge node[right] {$g_0$} (m-2-2)
  		              (m-2-1) edge[barred] node[below] {$C$} (m-2-2);
  		  \path				(m-1-1) edge[eq] (m-1-2);
  		  \path[transform canvas={shift={($(m-1-2)!(0,0)!(m-2-2)$)}}] (m-1-1) edge[cell] node[right] {$\phi$} (m-2-1);
  		\end{tikzpicture} \qquad\qquad\qquad \begin{tikzpicture}[baseline]
  			\matrix(m)[math2em]{A & C \times_{C_0} C \\ C \times_{C_0} C & C \\};
  			\path[map]	(m-1-1) edge node[above] {$(f, \phi \of d_1)$} (m-1-2)
  													edge node[left] {$(\phi \of d_0, g)$} (m-2-1)
  									(m-1-2) edge node[right] {$m_C$} (m-2-2)
  									(m-2-1) edge node[below] {$m_C$} (m-2-2);
  		\end{tikzpicture}
		\end{displaymath}
		Categories, functors and transformations internal in $\E$ form a $2$-category $\Cat(\E)$. If $\E$ has coequalisers preserved by pullback, so that the double category $\inProf\E$ of profunctors internal in $\E$ exists, then it is not hard to show that $V(\inProf\E) \iso \Cat(\E)$ by using the following lemma.
	\end{example}
	
	\begin{lemma} \label{vertical internal transformations}
		Let $\map fAC$ and $\map gAD$ be internal functors in a category $\E$ with pullbacks, and let $\hmap KCD$ be an internal profunctor.
		\begin{displaymath}
			\begin{tikzpicture}[baseline]
  		  \matrix(m)[math175em]{A & A \\ C & D \\};
  		  \path[map]  (m-1-1) edge node[left] {$f$} (m-2-1)
  		              (m-1-2) edge node[right] {$g$} (m-2-2)
  		              (m-2-1) edge[barred] node[below] {$K$} (m-2-2);
  		  \path				(m-1-1) edge[eq] (m-1-2);
  		  \path[transform canvas={shift={($(m-1-2)!(0,0)!(m-2-2)$)}}] (m-1-1) edge[cell] node[right] {$\phi$} (m-2-1);
  		\end{tikzpicture} \qquad\qquad \begin{tikzpicture}[baseline]
  		  \matrix(m)[math175em]{A_0 & A_0 \\ C_0 & D_0 \\};
  		  \path[map]  (m-1-1) edge node[left] {$f_0$} (m-2-1)
  		              (m-1-2) edge node[right] {$g_0$} (m-2-2)
  		              (m-2-1) edge[barred] node[below] {$K$} (m-2-2);
  		  \path				(m-1-1) edge[eq] (m-1-2);
  		  \path[transform canvas={shift={($(m-1-2)!(0,0)!(m-2-2)$)}}] (m-1-1) edge[cell] node[right] {$\phi_0$} (m-2-1);
  		\end{tikzpicture} \qquad\quad \begin{tikzpicture}[baseline]
  			\matrix(m)[math2em]{A & C \times_{C_0} K \\ K \times_{D_0} D & K \\};
  			\path[map]	(m-1-1) edge node[above] {$(f, \phi_0 \of d_1)$} (m-1-2)
  													edge node[left] {$(\phi_0 \of d_0, g)$} (m-2-1)
  									(m-1-2) edge node[right] {$l$} (m-2-2)
  									(m-2-1) edge node[below] {$r$} (m-2-2);
  		\end{tikzpicture}
		\end{displaymath}
		Transformations of internal profunctors $\phi$, of the form as on the left above and natural in the sense of \eqref{naturality of internal transformation}, correspond bijectively to cells $\phi_0$ of spans in $\E$, as in the middle, that make the diagram on the right commute. This correspondence is given by the assignment $\phi \mapsto \brks{A_0 \xrar{e_A} A \xrar \phi K}$.
	\end{lemma}
	\begin{proof}
		First we show that for an internal transformation $\phi$ as on the left above, given by an $\E$-map $\map\phi AK$, the composite $\phi_0 = \phi \of e_A$ makes the naturality diagram on the right commute. To see this we compute its top leg:
		\begin{align} \label{naturality of phi_0}
			\bigbrks{A \xrar{(f, \phi_0 \of d_1)} C \times_{C_0} K \xrar l K} &= \bigbrks{A \xrar{(\id, e_A \of d_1)} A \times_{A_0} A \xrar{f \times_{f_0} \phi} C \times_{C_0} K \xrar l K} \notag\\
			&= \bigbrks{A \xrar{(\id, e_A \of d_1)} A \times_{A_0} A \xrar{m_A} A \xrar\phi K} = \phi,
		\end{align}
		where the second identity is the naturality of $\phi$, see \eqref{naturality of internal transformation}, while the third is the unit axiom for $A$. A similar argument shows that the bottom leg also equals $\phi$ so that commutativity follows.
		
		Secondly, for the reverse assignment we take $\phi_0 \mapsto l \of (f, \phi_0 \of d_1)$, which is the diagonal $A \to K$ of the commuting square above. To see that this forms a transformation of profunctors we have to show that the naturality diagrams \eqref{naturality of internal transformation} commute. That the first commutes is shown by
		\begin{align*}
			\bigbrks{A \times_{A_0} A \xrar{m_A} &A \xrar{(f, \phi_0 \of d_1)} C \times_{C_0} K \xrar l K} \\
			= \bigbrks{&A \times_{A_0} A \xrar{(f \of m_A, \phi_0 \of d_1)} C \times_{C_0} K \xrar l K} \\
			= \bigbrks{&A \times_{A_0} A \xrar{f \times_{f_0} (f, \phi_0 \of d_1)} C \times_{C_0} C \times_{C_0} K \xrar{m_C \times \id} C \times_{C_0} K \xrar l K} \\
			= \bigbrks{&A \times_{A_0} A \xrar{f \times_{f_0} l \of (f, \phi_0 \of d_1)} C \times_{C_0} K \xrar l K},
		\end{align*}
		where the second identity follows from the functoriality of $f$, and the third from the associativity of $l$. The second naturality diagram of \eqref{naturality of internal transformation} commutes likewise. That the assignments thus given are mutually inverse follows easily from \eqref{naturality of phi_0} above, and from the unit laws for $f$ and $\map l{C \times_{C_0} K}K$.
	\end{proof}
	
	\section{Equipments}
	Two important ways in which horizontal morphisms can be related to vertical morphisms are as `companions' and `conjoints'. Informally we think of the companion of $\map fAC$ as being a horizontal morphism $A \brar C$ that is `isomorphic' to $f$, while its conjoint is a horizontal morphism $C \brar A$ that is `adjoint' to $f$.
	
	\begin{definition} \label{definition: companions and conjoints}
		Let $\map fAC$ be a vertical morphism in a double category $\K$. A \emph{companion} of $f$ is a horizontal morphism $\hmap{f_*}AC$ equipped with cells $\ls f \eps$ and $\ls f \eta$ as on the left below, such that $\ls f \eps \of \ls f \eta = 1_f$ and $\ls f \eta \hc \ls f \eps = \id_{f_*}$\footnote{Again the unitors $\mathfrak l\colon 1_A \hc f_* \iso f_*$ and $\mathfrak r\colon f_* \hc 1_C \iso f_*$ are suppressed.}.
		\begin{displaymath}
			\begin{tikzpicture}[baseline]
   			\matrix(m)[math175em]{A & C \\ C & C \\};
    		\path[map]  (m-1-1) edge[barred] node[above] {$f_*$} (m-1-2)
        		                edge node[left] {$f$} (m-2-1);
    		\path       (m-1-2) edge[eq] (m-2-2)
        		        (m-2-1) edge[eq] (m-2-2);
    		\path[transform canvas={shift={($(m-1-1)!0.5!(m-2-1)$)}}] (m-1-2) edge[cell] node[right] {$\ls f \eps$} (m-2-2);
  		\end{tikzpicture} \qquad\qquad \begin{tikzpicture}[baseline]
    		\matrix(m)[math175em]{A & A \\ A & C \\};
    		\path       (m-1-1) edge[eq] (m-1-2)
        		                edge[eq] (m-2-1);
    		\path[map]  (m-1-2) edge node[right] {$f$} (m-2-2)
        		        (m-2-1) edge[barred] node[below] {$f_*$} (m-2-2);
    		\path[transform canvas={shift={($(m-1-1)!0.5!(m-2-1)$)}}] (m-1-2) edge[cell] node[right] {$\ls f \eta$} (m-2-2);
  		\end{tikzpicture} \qquad\qquad \begin{tikzpicture}[baseline]
    		\matrix(m)[math175em]{C & A \\ C & C \\};
    		\path[map]  (m-1-1) edge[barred] node[above] {$f^*$} (m-1-2)
										(m-1-2) edge node[right] {$f$} (m-2-2);
    		\path       (m-1-1) edge[eq] (m-2-1)
    		            (m-2-1) edge[eq] (m-2-2);
    		\path[transform canvas={shift={($(m-1-1)!0.5!(m-2-1)$)}}] (m-1-2) edge[cell] node[right] {$\eps_f$} (m-2-2);
  		\end{tikzpicture} \qquad\qquad \begin{tikzpicture}[baseline]
    		\matrix(m)[math175em]{A & A \\ C & A \\};
    		\path       (m-1-1) edge[eq] (m-1-2)
        		        (m-1-2) edge[eq] (m-2-2);
    		\path[map]  (m-1-1) edge node[left] {$f$} (m-2-1)
        		        (m-2-1) edge[barred] node[below] {$f^*$} (m-2-2);
    		\path[transform canvas={shift={($(m-1-1)!0.5!(m-2-1)$)}}] (m-1-2) edge[cell] node[right] {$\eta_f$} (m-2-2);
  		\end{tikzpicture}
		\end{displaymath}
		Dually a \emph{conjoint} of $f$ is a horizontal morphism $\hmap{f^*}CA$ equipped with cells $\eps_f$ and $\eta_f$ as on the right above, such that $\eps_f \of \eta_f = 1_f$ and $\eps_f \hc \eta_f = \id_{f^*}$.
	\end{definition}
	We call the cells $\ls f\eps$ and $\ls f\eta$ above \emph{companion cells}, and the identities that they satisfy \emph{companion identities}; likewise the cells $\eps_f$ and $\eta_f$ are called \emph{conjoint cells}, and their identities \emph{conjoint identities}. Notice that the notions of companion and conjoint are swapped when moving from $\K$ to $\co\K$.
	\begin{example} \label{example: enriched companions}
		In the double category $\enProf\V$ of $\V$-profunctors the companion of a $\V$\ndash functor $\map fAC$ can be given by $f_* = 1_C(f, \id)$, that is $f_*(a, c) = C(fa, c)$. Taking $\ls f\eps = \id_{1_C(f, \id)}$ and $\cell{\ls f\eta = 1_f}{1_A}{1_C(f,f)}$, given by the actions $A(a_1, a_2) \to C(fa_1, fa_2)$ of $f$ on hom-objects, we clearly have $\ls f\eps \of \ls f\eta = 1_f$. On the other hand $\ls f\eta \hc \ls f\eps$ is the factorisation of
		\begin{displaymath}
			\coprod_{a' \in A} A(a, a') \tens C(fa', c) \xrar{\coprod f \tens \id} \coprod_{a' \in A} C(fa, fa') \tens C(fa', c) \xrar{\of} C(fa, c),
		\end{displaymath}
		through the map $\coprod_{a'} A(a, a') \tens C(fa', c) \to C(fa, c)$ inducing the unitor \mbox{$\mathfrak l\colon 1_A \hc f_* \iso f_*$}. Since this factorisation is obtained by precomposing the above with the map $C(fa, c) \to \coprod_{a'} A(a, a') \tens C(fa', c)$, that is induced by the identity $\map{\id_a}1{A(a,a)}$, we see that the second companion identity $\ls f\eta \hc \ls f\eps = \id_{f_*}$ holds as well.
		
		The conjoint $\hmap{f^*}CA$ of $f$ is defined dually. 
	\end{example}
	
	The notions of companion and conjoint are closely related to that of `cartesian cell' and `opcartesian cell' which we shall now recall, using largely the same notation as used in Section 4 of \cite{Shulman08}.
	\begin{definition} \label{definition: cartesian and opcartesian cells}
		A cell $\phi$ on the left below is called \emph{cartesian} if any cell $\psi$, as in the middle, factors uniquely through $\phi$ as shown. Dually $\phi$ is called \emph{opcartesian} if any cell $\chi$ as on the right factors uniquely through $\phi$ as shown.
		\begin{displaymath}
  		\begin{tikzpicture}[textbaseline]
    		\matrix(m)[math175em]{A & B \\ C & D \\};
    		\path[map]  (m-1-1) edge[barred] node[above] {$J$} (m-1-2)
                        		edge node[left] {$f$} (m-2-1)
		                (m-1-2) edge node[right] {$g$} (m-2-2)
    		            (m-2-1) edge[barred] node[below] {$K$} (m-2-2);
    		\path[transform canvas={shift={($(m-1-1)!0.5!(m-2-1)$)}}] (m-1-2) edge[cell] node[right] {$\phi$} (m-2-2);
  		\end{tikzpicture} \quad\quad \begin{tikzpicture}[textbaseline]
    		\matrix(m)[math175em]{X & Y \\ A & B \\ C & D \\};
    		\path[map]  (m-1-1) edge[barred] node[above] {$H$} (m-1-2)
        		                edge node[left] {$h$} (m-2-1)
        		        (m-1-2) edge node[right] {$k$} (m-2-2)
        		        (m-2-1) edge node[left] {$f$} (m-3-1)
        		        (m-2-2) edge node[right] {$g$} (m-3-2)
        		        (m-3-1) edge[barred] node[below] {$K$} (m-3-2);
    		\path[transform canvas={shift={($(m-2-1)!0.5!(m-1-1)$)}}] (m-2-2) edge[cell] node[right] {$\psi$} (m-3-2);
  		\end{tikzpicture} = \begin{tikzpicture}[textbaseline]
    		\matrix(m)[math175em]{X & Y \\ A & B \\ C & D \\};
    		\path[map]  (m-1-1) edge[barred] node[above] {$H$} (m-1-2)
        		                edge node[left] {$h$} (m-2-1)
            		    (m-1-2) edge node[right] {$k$} (m-2-2)
            		    (m-2-1) edge[barred] node[below] {$J$} (m-2-2)
            		            edge node[left] {$f$} (m-3-1)
            		    (m-2-2) edge node[right] {$g$} (m-3-2)
            		    (m-3-1) edge[barred] node[below] {$K$} (m-3-2);
    		\path[transform canvas={shift=(m-2-1))}]
        		        (m-1-2) edge[cell] node[right] {$\psi'$} (m-2-2)
        		        (m-2-2) edge[transform canvas={yshift=-3pt}, cell] node[right] {$\phi$} (m-3-2);
  		\end{tikzpicture} \quad\quad \begin{tikzpicture}[textbaseline]
    		\matrix(m)[math175em]{A & B \\ C & D \\ X & Y \\};
    		\path[map]  (m-1-1) edge[barred] node[above] {$J$} (m-1-2)
        		                edge node[left] {$f$} (m-2-1)
        		        (m-1-2) edge node[right] {$g$} (m-2-2)
        		        (m-2-1) edge node[left] {$h$} (m-3-1)
        		        (m-2-2) edge node[right] {$k$} (m-3-2)
        		        (m-3-1) edge[barred] node[below] {$L$} (m-3-2);
    		\path[transform canvas={shift={($(m-2-1)!0.5!(m-1-1)$)}}] (m-2-2) edge[cell] node[right] {$\chi$} (m-3-2);
  		\end{tikzpicture} = \begin{tikzpicture}[textbaseline]
    		\matrix(m)[math175em]{A & B \\ C & D \\ X & Y \\};
    		\path[map]  (m-1-1) edge[barred] node[above] {$J$} (m-1-2)
        		                edge node[left] {$f$} (m-2-1)
            		    (m-1-2) edge node[right] {$g$} (m-2-2)
            		    (m-2-1) edge[barred] node[below] {$K$} (m-2-2)
            		            edge node[left] {$h$} (m-3-1)
            		    (m-2-2) edge node[right] {$k$} (m-3-2)
            		    (m-3-1) edge[barred] node[below] {$L$} (m-3-2);
    		\path[transform canvas={shift=(m-2-1))}]
        		        (m-1-2) edge[cell] node[right] {$\phi$} (m-2-2)
        		        (m-2-2) edge[transform canvas={yshift=-3pt}, cell] node[right] {$\chi'$} (m-3-2);
  		\end{tikzpicture}
		\end{displaymath}
	\end{definition}
	
	If a cartesian cell $\phi$ exists then we call $J$ a \emph{restriction of $K$ along $f$ and $g$}, and write $K(f, g) = J$; if $K = 1_C$ then we write $C(f, g) = 1_C(f,g)$. By their universal property, any two cartesian cells defining the same restriction factor through each other as invertible horizontal cells. Moreover, since the vertical composite of two cartesian cells is again cartesian, and since vertical units $\id_J$ are cartesian, it follows that restrictions are pseudofunctorial, in the sense that $K(f, g)(h, k) \iso K(f \of h, g \of k)$ and $K(\id, \id) \iso K$. Dually, if an opcartesian cell $\phi$ exists then we call $K$ an \emph{extension of $J$ along $f$ and $g$}; like restrictions, extensions are unique up to isomorphism and pseudofunctorial. In fact the notions of restriction and extension are vertically dual, that is they reverse when moving from $\K$ to $\op\K$. We shall usually not name cartesian and opcartesian cells, but simply depict them as below.
	\begin{displaymath}
		\begin{tikzpicture}[baseline]
    		\matrix(m)[math175em]{A & B \\ C & D \\};
    		\path[map]  (m-1-1) edge[barred] node[above] {$J$} (m-1-2)
                        		edge node[left] {$f$} (m-2-1)
		                (m-1-2) edge node[right] {$g$} (m-2-2)
    		            (m-2-1) edge[barred] node[below] {$K$} (m-2-2);
    		\draw ($(m-1-1)!0.5!(m-2-2)$) node {\cc};
  		\end{tikzpicture} \qquad\qquad\qquad\qquad\qquad \begin{tikzpicture}[baseline]
    		\matrix(m)[math175em]{A & B \\ C & D \\};
    		\path[map]  (m-1-1) edge[barred] node[above] {$J$} (m-1-2)
                        		edge node[left] {$f$} (m-2-1)
		                (m-1-2) edge node[right] {$g$} (m-2-2)
    		            (m-2-1) edge[barred] node[below] {$K$} (m-2-2);
    		\draw ($(m-1-1)!0.5!(m-2-2)$) node {\oc};
  		\end{tikzpicture}
	\end{displaymath}
	
	\begin{example}
		For morphisms $\map fAC$, $\hmap KCD$ and $\map gBD$ in $\enProf\V$ we have, in \exref{example: enriched profunctors}, already used the notation $K(f, g)$ to denote the composite \mbox{$\hmap{K \of (\op f \tens g)}AB$}. It is readily seen that the cell $\cell\eps{K(f,g)}K$ given by the identity transformation on $K\of (\op f \tens g)$ is indeed cartesian, so that $K \of (\op f \tens g)$ is the restriction of $K$ along $f$ and $g$.
	\end{example}
	
	The following lemmas record some standard properties of cartesian and opcartesian cells. The proof of the first is easy and omitted.
	\begin{lemma}[Pasting lemma]
		In a double category consider the following vertical composite.
		\begin{displaymath}
			\begin{tikzpicture}
				\matrix(m)[math175em]{A & B \\ C & D \\ E & F \\};
				\path[map]	(m-1-1) edge[barred] node[above] {$J$} (m-1-2)
														edge node[left] {$f$} (m-2-1)
										(m-1-2) edge node[right] {$g$} (m-2-2)
										(m-2-1) edge[barred] node[below] {$K$} (m-2-2)
														edge node[left] {$h$} (m-3-1)
										(m-2-2) edge node[right] {$k$} (m-3-2)
										(m-3-1) edge[barred] node[below] {$L$} (m-3-2);
				\path[transform canvas={shift=(m-2-1)}]
										(m-1-2) edge[cell] node[right] {$\phi$} (m-2-2)
										(m-2-2) edge[cell, transform canvas={yshift=-3pt}] node[right] {$\psi$} (m-3-2);
			\end{tikzpicture}
		\end{displaymath}
		The following hold:
		\begin{enumerate}[label=(\alph*)]
			\item if $\psi$ is cartesian then $\psi \of \phi$ is cartesian if and only if $\phi$ is;
			\item if $\phi$ is opcartesian then $\psi \of \phi$ is opcartesian if and only if $\psi$ is.
		\end{enumerate}
	\end{lemma}
	\begin{lemma} \label{invertible cartesian and opcartesian cells}
		Any cartesian cell has a vertical inverse if and only if its vertical source and target are invertible. The same holds for opcartesian cells.
	\end{lemma}
	\begin{proof}[(sketch).]
		The `if' part, for a cartesian cell $\phi$ with invertible vertical boundaries and horizontal target $\hmap KCD$, is proved by factorising the vertical unit cell $\id_K$ through $\phi$.
	\end{proof}
	
	The following three results, which are Theorem 4.1 of \cite{Shulman08}, show how companions, conjoints, restrictions and extensions are related.
	\begin{lemma}[Shulman]
		In a double category consider cells of the form below.
		\begin{displaymath}
			\begin{tikzpicture}[baseline]
   			\matrix(m)[math175em]{A & C \\ C & C \\};
    		\path[map]  (m-1-1) edge[barred] node[above] {$J$} (m-1-2)
        		                edge node[left] {$f$} (m-2-1);
    		\path       (m-1-2) edge[eq] (m-2-2)
        		        (m-2-1) edge[eq] (m-2-2);
    		\path[transform canvas={shift={($(m-1-1)!0.5!(m-2-1)$)}}] (m-1-2) edge[cell] node[right] {$\phi$} (m-2-2);
  		\end{tikzpicture} \quad\qquad \begin{tikzpicture}[baseline]
    		\matrix(m)[math175em]{A & A \\ A & C \\};
    		\path       (m-1-1) edge[eq] (m-1-2)
        		                edge[eq] (m-2-1);
    		\path[map]  (m-1-2) edge node[right] {$f$} (m-2-2)
        		        (m-2-1) edge[barred] node[below] {$J$} (m-2-2);
    		\path[transform canvas={shift={($(m-1-1)!0.5!(m-2-1)$)}}] (m-1-2) edge[cell] node[right] {$\psi$} (m-2-2);
  		\end{tikzpicture} \quad\qquad \begin{tikzpicture}[baseline]
    		\matrix(m)[math175em]{D & C \\ D & D \\};
    		\path[map]  (m-1-1) edge[barred] node[above] {$K$} (m-1-2)
										(m-1-2) edge node[right] {$g$} (m-2-2);
    		\path       (m-1-1) edge[eq] (m-2-1)
    		            (m-2-1) edge[eq] (m-2-2);
    		\path[transform canvas={shift={($(m-1-1)!0.5!(m-2-1)$)}}] (m-1-2) edge[cell] node[right] {$\chi$} (m-2-2);
  		\end{tikzpicture} \quad\qquad \begin{tikzpicture}[baseline]
    		\matrix(m)[math175em]{C & C \\ D & C \\};
    		\path       (m-1-1) edge[eq] (m-1-2)
        		        (m-1-2) edge[eq] (m-2-2);
    		\path[map]  (m-1-1) edge node[left] {$g$} (m-2-1)
        		        (m-2-1) edge[barred] node[below] {$K$} (m-2-2);
    		\path[transform canvas={shift={($(m-1-1)!0.5!(m-2-1)$)}}] (m-1-2) edge[cell] node[right] {$\xi$} (m-2-2);
  		\end{tikzpicture}
  	\end{displaymath}
		The following hold:
  		\begin{enumerate}[label=\textup{(\alph*)}]
  			\item $\phi$ is cartesian if and only if there exists a cell $\psi$ such that $(\phi, \psi)$ defines $J$ as the companion of $f$;
  			\item $\psi$ is opcartesian if and only if there exists a cell $\phi$ such that $(\phi, \psi)$ defines $J$ as the companion of $f$.
  		\end{enumerate}
		Horizontally dual analogues hold for the pair of cells on the right above, obtained by replacing `$\phi$' by `$\chi$', `$\psi$' by `$\xi$', and `$J$ as companion of $f$' by `$K$ as the conjoint of $g$'.
	\end{lemma}
	\begin{proof}[(sketch).]
		For the `if' part of (a): the unique factorisation of a cell $\rho$ through $\phi$, as in \defref{definition: cartesian and opcartesian cells}, is obtained by composing $\rho$ on the left with $\psi$. For the `only if' part: $\psi$ is obtained by factorising $1_f$ through $\phi$. The other assertions are duals of (a).
	\end{proof}
	\begin{lemma}[Shulman] \label{cartesian and opcartesian cells in terms of companions and conjoints}
		In a double category suppose that the pairs of cells $(\ls f \eps, \ls f\eta)$ and $(\eps_g, \eta_g)$, as in the composites below, define the companion of $\map fAC$ and the conjoint of $\map gBD$.
		\begin{displaymath}
			\begin{tikzpicture}[baseline]
				\matrix(m)[math175em]{A & C & D & B \\ C & C & D & D \\};
				\path[map]	(m-1-1) edge[barred] node[above] {$f_*$} (m-1-2)
														edge node[left] {$f$} (m-2-1)
										(m-1-2) edge[barred] node[above] {$K$} (m-1-3)
										(m-1-3) edge[barred] node[above] {$g^*$} (m-1-4)
										(m-1-4) edge node[right] {$g$} (m-2-4)
										(m-2-2) edge[barred] node[below] {$K$} (m-2-3);
				\path				(m-1-2) edge[eq] (m-2-2)
										(m-1-3) edge[eq] (m-2-3)
										(m-2-1) edge[eq] (m-2-2)
										(m-2-3) edge[eq] (m-2-4);
				\path[transform canvas={shift={($(m-1-3)!0.5!(m-2-3)$)}}]	(m-1-1) edge[cell] node[right] {$\ls f\eps$} (m-2-1)
										(m-1-3) edge[cell] node[right] {$\eps_g$} (m-2-3);				
			\end{tikzpicture} \qquad\qquad \begin{tikzpicture}[baseline]
				\matrix(m)[math175em]{B & B & A & A \\ D & B & A & C \\};
				\path[map]	(m-1-1) edge node[left] {$g$} (m-2-1)
										(m-1-2) edge[barred] node[above] {$J$} (m-1-3)
										(m-2-3) edge[barred] node[below] {$f_*$} (m-2-4)
										(m-1-4) edge node[right] {$f$} (m-2-4)
										(m-2-1) edge[barred] node[below] {$g^*$} (m-2-2)
										(m-2-2) edge[barred] node[below] {$J$} (m-2-3);
				\path				(m-1-2) edge[eq] (m-2-2)
										(m-1-3) edge[eq] (m-2-3)
										(m-1-1) edge[eq] (m-1-2)
										(m-1-3) edge[eq] (m-1-4);
				\path[transform canvas={shift={($(m-1-3)!0.5!(m-2-3)$)}}]	(m-1-1) edge[cell] node[right] {$\eta_g$} (m-2-1)
										(m-1-3) edge[cell] node[right] {$\ls f\eta$} (m-2-3);				
			\end{tikzpicture}
		\end{displaymath}
		For any horizontal morphism $\hmap KCD$ the composite on the left above is cartesian, while for any $\hmap JBA$ the composite on the right is opcartesian.
	\end{lemma}
	\begin{proof}[(sketch).] The unique factorisation of any cell $\psi$ through the composite on the left above, as in \defref{definition: cartesian and opcartesian cells}, is obtained by composing $\psi$ on the left with $\ls f\eta$ and on the right with $\eta_g$. Likewise unique factorisations through the composite on the right are obtained by composition on the left with $\eps_g$ and on the right with $\ls f\eps$.
	\end{proof}
	Together the pair of lemmas above implies the following.
	\begin{theorem}[Shulman]
		The following conditions on a double category $\K$ are equivalent:
		\begin{enumerate}[label=(\alph*)]
			\item $\K$ has all companions and conjoints;
			\item $\K$ has all restrictions;
			\item $\K$ has all extensions.
		\end{enumerate}
	\end{theorem}
		Next is the definition of equipment.
	\begin{definition}
		An \emph{equipment} is a double category together with, for each vertical morphism $\map fAC$, two specified pairs of cells $(\ls f\eta, \ls f\eps)$ and $(\eta_f, \eps_f)$ that define the companion $\hmap{f_*}AC$ and the conjoint $\hmap{f^*}CA$ respectively, as in \defref{definition: companions and conjoints}. 
	\end{definition}
	By \lemref{cartesian and opcartesian cells in terms of companions and conjoints} above, a specification of companions and conjoints induces a specification of restrictions and extensions, by taking the restriction of $\hmap KCD$ along $\map fAC$ and $\map gBD$ to be $K(f, g) = f_* \hc K \hc g^*$, and likewise by taking the extension of $\hmap JAB$ along $f$ and $g$ to be $f^* \hc J \hc g_*$. Unless specified otherwise we will always mean these particular restrictions and extensions, along with their defining cells as given by \lemref{cartesian and opcartesian cells in terms of companions and conjoints}, when working in an equipment.
	\begin{example}
		The double category $\enProf\V$ of $\V$-enriched profunctors is an equipment with companions and conjoints as given in \exref{example: enriched companions}.
	\end{example}
	
	\begin{example} \label{example: equipment of internal profunctors}
		The double category $\inProf\E$ can be made into an equipment as well. Briefly, the companion $\hmap{f_*}AC$ of an internal functor $\map fAC$ can be taken to be the span
		\begin{displaymath}
			f_* = \bigbrks{A_0 \xlar{\pi_1} A_0 \pb{f_0}{d_0} C \xrar{\pi_2} C \xrar{d_1} C_0},
		\end{displaymath}
		where $A_0 \pb{f_0}{d_0} C$ denotes the pullback of $A_0 \xrar{f_0} C_0 \xlar{d_0} C$, with projections $\pi_1$ and $\pi_2$. If $\E = \Set$ then $f_*$ is the set consisting of morphisms $\map u{fa}c$ in $C$. The actions that make $f_*$ into an internal profunctor are induced by the multiplication of $C$, and it is straightforward to show that the projection $\map{\pi_2}{f_*} C$ defines a cartesian internal transformation $\cell{\ls f\eps}{f_*}{1_C}$. A conjoint $\hmap{f^*}CA$ for $f$ can be given similarly.
	\end{example}
	
	Consider a cell $\phi$ in a double category as in the middle below, and assume that the companions of $f$ and $g$ exist. Then the opcartesian and cartesian cell in the composite on the left below exist by \lemref{cartesian and opcartesian cells in terms of companions and conjoints}, and we write $\phi_*$ for the horizontal cell that is obtained by factorising $\phi$ as in the identity on the left below. Dually, if the conjoints of $f$ and $g$ exist then we write $\phi^*$ for the factorisation of $\phi$ as on the right. Especially the factorisations $\phi_*$ will be important later on.
	\begin{equation} \label{companion and conjoint factorisations}
		\begin{tikzpicture}[textbaseline]
			\matrix(m)[math175em]{A & & B \\ A & B & D \\ A & C & D \\ C & & D \\};
			\path[map]	(m-1-1) edge[barred] node[above] {$J$} (m-1-3)
									(m-1-3) edge node[right] {$g$} (m-2-3)
									(m-2-1) edge[barred] node[below] {$J$} (m-2-2)
									(m-2-2) edge[barred] node[below] {$g_*$} (m-2-3)
									(m-3-1) edge[barred] node[below] {$f_*$} (m-3-2)
													edge node[left] {$f$} (m-4-1)
									(m-3-2) edge[barred] node[below] {$K$} (m-3-3)
									(m-4-1) edge[barred] node[below] {$K$} (m-4-3)
									(m-2-2) edge[cell] node[right] {$\phi_*$} (m-3-2);
			\path				(m-1-1) edge[eq] (m-2-1)
									(m-2-1) edge[eq] (m-3-1)
									(m-2-3) edge[eq] (m-3-3)
									(m-3-3) edge[eq] (m-4-3);
			\draw				($(m-1-1)!0.5!(m-2-3)$) node {\oc}
									($(m-3-1)!0.5!(m-4-3)$) node {\cc};
		\end{tikzpicture} = \begin{tikzpicture}[textbaseline]
    	\matrix(m)[math175em]{A & B \\ C & D \\};
    	\path[map]  (m-1-1) edge[barred] node[above] {$J$} (m-1-2)
                       		edge node[left] {$f$} (m-2-1)
		               (m-1-2) edge node[right] {$g$} (m-2-2)
    	            (m-2-1) edge[barred] node[below] {$K$} (m-2-2);
    	\path[transform canvas={shift={($(m-1-1)!0.5!(m-2-1)$)}}] (m-1-2) edge[cell] node[right] {$\phi$} (m-2-2);
  	\end{tikzpicture} = \begin{tikzpicture}[textbaseline]
			\matrix(m)[math175em]{A & & B \\ C & A & B \\ C & D & B \\ C & & D \\};
			\path[map]	(m-1-1) edge[barred] node[above] {$J$} (m-1-3)
													edge node[left] {$f$} (m-2-1)
									(m-2-1) edge[barred] node[below] {$f^*$} (m-2-2)
									(m-2-2) edge[barred] node[below] {$J$} (m-2-3)
									(m-3-1) edge[barred] node[below] {$K$} (m-3-2)
									(m-3-2) edge[barred] node[below] {$g^*$} (m-3-3)
									(m-3-3)	edge node[right] {$g$} (m-4-3)
									(m-4-1) edge[barred] node[below] {$K$} (m-4-3)
									(m-2-2) edge[cell] node[right] {$\phi^*$} (m-3-2);
			\path				(m-1-3) edge[eq] (m-2-3)
									(m-2-1) edge[eq] (m-3-1)
									(m-2-3) edge[eq] (m-3-3)
									(m-3-1) edge[eq] (m-4-1);
			\draw				($(m-1-1)!0.5!(m-2-3)$) node {\oc}
									($(m-3-1)!0.5!(m-4-3)$) node {\cc};
		\end{tikzpicture}
	\end{equation}
	
	\begin{lemma} \label{companion and conjoint factorisations functoriality}
		In a double category the correspondence $\phi \leftrightarrow \phi^*$ given above is functorial in the sense that, for horizontally composable cells as on the left below such that the conjoints of $f$, $g$ and $h$ exist, the identity on the right holds.
		\begin{displaymath}
			\begin{tikzpicture}[textbaseline]
				\matrix(m)[math175em]{A & B & E \\ C & D & F \\};
				\path[map]	(m-1-1) edge[barred] node[above] {$J$} (m-1-2)
														edge node[left] {$f$} (m-2-1)
										(m-1-2) edge[barred] node[above] {$H$} (m-1-3)
														edge node[right] {$g$} (m-2-2)
										(m-1-3) edge node[right] {$h$} (m-2-3)
										(m-2-1) edge[barred] node[below] {$K$} (m-2-2)
										(m-2-2) edge[barred] node[below] {$L$} (m-2-3);
				\path[transform canvas={shift=($(m-1-1)!0.5!(m-2-2)$)}]
										(m-1-2) edge[cell] node[right] {$\phi$} (m-2-2)
										(m-1-3) edge[cell] node[right] {$\chi$} (m-2-3);
			\end{tikzpicture} \quad \begin{tikzpicture}[textbaseline]
				\matrix(m)[math175em]{C & A & B & E \\ C & D & B & E \\ C & D & F & E \\};
				\path[map]	(m-1-1) edge[barred] node[above] {$f^*$} (m-1-2)
										(m-1-2) edge[barred] node[above] {$J$} (m-1-3)
										(m-1-3) edge[barred] node[above] {$H$} (m-1-4)
										(m-2-1) edge[barred] node[below] {$K$} (m-2-2)
										(m-2-2) edge[barred] node[below] {$g^*$} (m-2-3)
										(m-2-3) edge[barred] node[below] {$H$} (m-2-4)
										(m-3-1) edge[barred] node[below] {$K$} (m-3-2)
										(m-3-2) edge[barred] node[below] {$L$} (m-3-3)
										(m-3-3) edge[barred] node[below] {$h^*$} (m-3-4)
										(m-1-2) edge[cell] node[right] {$\phi^*$} (m-2-2)
										(m-2-3) edge[cell] node[right] {$\chi^*$} (m-3-3);
				\path				(m-1-1) edge[eq] (m-2-1)
										(m-1-3) edge[eq] (m-2-3)
										(m-1-4) edge[eq] (m-2-4)
										(m-2-1) edge[eq] (m-3-1)
										(m-2-2) edge[eq] (m-3-2)
										(m-2-4) edge[eq] (m-3-4);
			\end{tikzpicture} = \begin{tikzpicture}[textbaseline]
				\matrix(m)[math175em]{C & A & B & E \\ C & D & F & E \\};
				\path[map]	(m-1-1) edge[barred] node[above] {$f^*$} (m-1-2)
										(m-1-2) edge[barred] node[above] {$J$} (m-1-3)
										(m-1-3) edge[barred] node[above] {$H$} (m-1-4)
										(m-2-1) edge[barred] node[below] {$K$} (m-2-2)
										(m-2-2) edge[barred] node[below] {$L$} (m-2-3)
										(m-2-3) edge[barred] node[below] {$h^*$} (m-2-4);
				\path				(m-1-1) edge[eq] (m-2-1)
										(m-1-4) edge[eq] (m-2-4);
				\path[map, transform canvas={shift={($(m-1-3)!0.5!(m-2-3)$)}}]	(m-1-2) edge[cell] node[right] {$(\phi \hc \chi)^*$} (m-2-2);
			\end{tikzpicture} 
		\end{displaymath}
		By horizontal duality the correspondence $\phi \leftrightarrow \phi_*$ is similarly functorial.
	\end{lemma}
	\begin{proof}
		The opcartesian and cartesian cell in the composite on the right of \eqref{companion and conjoint factorisations} are defined by the horizontal composites $\eta_f \hc \id_J$ and $\id_K \hc \eps_g$ respectively, so that the factorisation of $\phi$ is given by the composite $\phi^* = \eps_f \hc \phi \hc \eta_g$; compare the proof of \lemref{cartesian and opcartesian cells in terms of companions and conjoints}. Likewise the factorisations of $\chi$ and $\phi \hc \chi$ are given by $\chi^* = \eps_g \hc \chi \hc \eta_h$ and $(\phi \hc \chi)^* = \eps_f \hc \phi \hc \chi \hc \eta_h$, so that the identity follows from the conjoint identity $\eps_g \of \eta_g = 1_g$.
	\end{proof}
	
	We end this section with a useful consequence of \lemref{cartesian and opcartesian cells in terms of companions and conjoints}.
	\begin{lemma} \label{composing cartesian and opcartesian cells}
		In an equipment consider the following horizontal composites.
		\begin{displaymath}
			\begin{tikzpicture}[baseline]
				\matrix(m)[math175em]{A & B & E \\ C & D & E \\};
				\path[map]	(m-1-1) edge[barred] node[above] {$J$} (m-1-2)
														edge node[left] {$f$} (m-2-1)
										(m-1-2) edge[barred] node[above] {$H$} (m-1-3)
														edge node[right] {$g$} (m-2-2)
										(m-2-1) edge[barred] node[below] {$K$} (m-2-2)
										(m-2-2) edge[barred] node[below] {$L$} (m-2-3);
				\path				(m-1-3) edge[eq] (m-2-3);
				\path[transform canvas={shift=($(m-1-1)!0.5!(m-2-2)$)}]
										(m-1-2) edge[cell] node[right] {$\phi$} (m-2-2)
										(m-1-3) edge[cell] node[right] {$\chi$} (m-2-3);
			\end{tikzpicture} \qquad\qquad\qquad\qquad \begin{tikzpicture}[baseline]
				\matrix(m)[math175em]{A & B & E \\ A & D & F \\};
				\path[map]	(m-1-1) edge[barred] node[above] {$J$} (m-1-2)
										(m-1-2) edge[barred] node[above] {$H$} (m-1-3)
														edge node[right] {$g$} (m-2-2)
										(m-1-3) edge node[right] {$h$} (m-2-3)
										(m-2-1) edge[barred] node[below] {$K$} (m-2-2)
										(m-2-2) edge[barred] node[below] {$L$} (m-2-3);
				\path				(m-1-1) edge[eq] (m-2-1);
				\path[transform canvas={shift=($(m-1-1)!0.5!(m-2-2)$)}]
										(m-1-2) edge[cell] node[right] {$\psi$} (m-2-2)
										(m-1-3) edge[cell] node[right] {$\xi$} (m-2-3);
			\end{tikzpicture}
		\end{displaymath}
		The following hold for the composite on the left.
		\begin{enumerate}[label=(\alph*)]
			\item If $\phi$ is cartesian and $\chi$ is opcartesian then $\phi \hc \chi$ is cartesian.
			\item If $\phi$ is opcartesian and $\chi$ is cartesian then $\phi \hc \chi$ is opcartesian.
		\end{enumerate}
		For the composite on the right horizontally dual analogues hold, obtained by replacing `$\,\phi$' by `$\,\xi$' and `$\,\chi$' by `$\,\psi$'.
	\end{lemma}
	\begin{proof}
		Proving (b) for the composite on the left suffices, since proving (a) is vertically dual while proving (a) and (b) for the composite on the right is horizontally dual to proving them for the composite on the left. So we assume $\phi$ to be opcartesian and $\chi$ to be cartesian. By \lemref{cartesian and opcartesian cells in terms of companions and conjoints} the horizontal composite $\eta_f \hc \id_J \hc \ls g\eta$ forms, like $\phi$, an opcartesian cell $\cell\eta J{f^* \hc J \hc g_*}$ that defines the extension of $J$ along $f$ and $g$, which thus factors through $\phi$ as an invertible horizontal cell $\eta' \colon K \iso f^* \hc J \hc g_*$. Likewise $\cell{\eps = \ls g\eps \hc \id_L}{g_* \hc L}L$ forms a cartesian cell that factors through $\chi$ as $\eps'\colon g_* \hc L \iso H$. Composing $\phi \hc \chi$ with the isomorphisms $\eta'$ and $\eps'$ gives
		\begin{displaymath}
			(\eta' \of \phi) \hc (\chi \of \eps') = \eta_f \hc \id_J \hc \ls g\eta \hc \ls g\eps \hc \id_L = \eta_f \hc \id_J \hc \id_{g_*} \hc \id_L,
		\end{displaymath}
		where the second identity follows from the companion identity $\ls g\eta \hc \ls g\eps = \id_{g_*}$. By \lemref{cartesian and opcartesian cells in terms of companions and conjoints} the composite on the right-hand side is opcartesian, so that $\phi \hc \chi$ is too.
	\end{proof}

	\section{Kan extensions in double categories}
	We are now ready to describe various notions of Kan extension in double categories. In particular we shall extend a variation of Wood's notion of `indexed (co-)limit' in `bicategories equipped with abstract proarrows', that was given in \cite{Wood82}, to a notion of pointwise Kan extension in double categories.
	
	Let $\K$ be an equipment and consider the $2$-category $V(\K)$ consisting of its objects, vertical morphisms and vertical cells. Recall that a vertical cell $\eps$ as on the left below defines $r$ as the right Kan extension of $d$ along $j$ in $V(\K)$ if every cell $\psi$ as on the right factors uniquely through $\eps$ as shown. See for example Section X.3 of \cite{MacLane98}, where right Kan extensions are defined in the $2$-category $\Cat$ of categories, functors and natural transformations.
	\begin{displaymath}
		\begin{tikzpicture}[textbaseline]
    	\matrix(m)[math175em]{A & A \\ B & \phantom B \\ M & M \\};
    	\path[map]  (m-1-1) edge node[left] {$j$} (m-2-1)
       		        (m-1-2) edge node[right] {$d$} (m-3-2)
       		        (m-2-1) edge node[left] {$r$} (m-3-1);
      \path				(m-1-1) edge[eq] (m-1-2)
      						(m-3-1) edge[eq] (m-3-2);
    	\path[transform canvas={shift={($(m-2-1)!0.5!(m-1-1)$)}}] (m-2-2) edge[cell] node[right] {$\eps$} (m-3-2);
  	\end{tikzpicture} \qquad\qquad\qquad \begin{tikzpicture}[textbaseline]
    	\matrix(m)[math175em]{A & A \\ B & \phantom B \\ M & M \\};
    	\path[map]  (m-1-1) edge node[left] {$j$} (m-2-1)
       		        (m-1-2) edge node[right] {$d$} (m-3-2)
       		        (m-2-1) edge node[left] {$s$} (m-3-1);
      \path				(m-1-1) edge[eq] (m-1-2)
      						(m-3-1) edge[eq] (m-3-2);
    	\path[transform canvas={shift={($(m-2-1)!0.5!(m-1-1)$)}}] (m-2-2) edge[cell] node[right] {$\psi$} (m-3-2);
  	\end{tikzpicture} = \begin{tikzpicture}[textbaseline]
  		\matrix(m)[math175em]{A & A & A \\ B & B & \phantom B \\ M & M & M \\};
  		\path[map]	(m-1-1) edge node[left] {$j$} (m-2-1)
  								(m-1-2) edge node[right] {$j$} (m-2-2)
  								(m-1-3) edge node[right] {$d$} (m-3-3)
  								(m-2-1) edge node[left] {$s$} (m-3-1)
  								(m-2-2) edge node[right] {$r$} (m-3-2);
  		\path				(m-1-1) edge[eq] (m-1-2)
  								(m-1-2) edge[eq] (m-1-3)
  								(m-2-1) edge[eq] (m-2-2)
  								(m-3-1) edge[eq] (m-3-2)
  								(m-3-2) edge[eq] (m-3-3);
  		\path[transform canvas={shift={($(m-2-2)!0.5!(m-2-3)$)}}] (m-1-1) edge[cell] node[right] {$1_j$} (m-2-1)
  								(m-2-1) edge[cell] node[right] {$\psi'$} (m-3-1);
  		\path[transform canvas={shift={($(m-2-2)!0.5!(m-3-3)$)}}] (m-1-2) edge[cell] node[right] {$\eps$} (m-2-2);
  	\end{tikzpicture}
	\end{displaymath}
	Factorising through the opcartesian cell $\eta_j$ defining the conjoint $\hmap{j^*}BA$, we see that the unique factorisations above correspond exactly to unique factorisations
	\begin{displaymath}
		\begin{tikzpicture}[textbaseline]
  		\matrix(m)[math175em]{B & A \\ M & M \\};
  		\path[map]  (m-1-1) edge[barred] node[above] {$j^*$} (m-1-2)
													edge node[left] {$s$} (m-2-1)
									(m-1-2) edge node[right] {$d$} (m-2-2);
			\path				(m-2-1) edge[eq] (m-2-2);
			\path[transform canvas={shift={($(m-1-2)!(0,0)!(m-2-2)$)}}] (m-1-1) edge[cell] node[right] {$\phi$} (m-2-1);
		\end{tikzpicture} = \begin{tikzpicture}[textbaseline]
			\matrix(m)[math175em]{B & B & A \\ M & M & M \\};
			\path[map]	(m-1-1) edge node[left] {$s$} (m-2-1)
									(m-1-2) edge[barred] node[above] {$j^*$} (m-1-3)
													edge node[right] {$r$} (m-2-2)
									(m-1-3) edge node[right] {$d$} (m-2-3);
			\path				(m-1-1) edge[eq] (m-1-2)
									(m-2-1) edge[eq] (m-2-2)
									(m-2-2) edge[eq] (m-2-3);
			\path[transform canvas={shift=($(m-1-1)!0.5!(m-2-2)$)}]
									(m-1-2) edge[cell] node[right] {$\phi'$} (m-2-2)
									(m-1-3) edge[cell] node[right] {$\eps'$} (m-2-3);
		\end{tikzpicture}
	\end{displaymath}
	in $\K$, where $\eps'$ is the unique factorisation of $\eps$ through $\eta_j$. This observation leads us to the following definition, which was given by Grandis and Par\'e in \cite{Grandis-Pare08}.
	\begin{definition}[Grandis and Par\'e] \label{definition: right Kan extension}
		Let $\map dBM$ and $\hmap JAB$ be morphisms in a double category. The cell $\eps$ in the right-hand side below is said to define $r$ as the \emph{right Kan extension of $d$ along $J$} if every cell $\phi$ below factors uniquely through $\eps$ as shown.
		\begin{displaymath}
			\begin{tikzpicture}[textbaseline]
  			\matrix(m)[math175em]{A & B \\ M & M \\};
  			\path[map]  (m-1-1) edge[barred] node[above] {$J$} (m-1-2)
														edge node[left] {$s$} (m-2-1)
										(m-1-2) edge node[right] {$d$} (m-2-2);
				\path				(m-2-1) edge[eq] (m-2-2);
				\path[transform canvas={shift={($(m-1-2)!(0,0)!(m-2-2)$)}}] (m-1-1) edge[cell] node[right] {$\phi$} (m-2-1);
			\end{tikzpicture} = \begin{tikzpicture}[textbaseline]
				\matrix(m)[math175em]{A & A & B \\ M & M & M \\};
				\path[map]	(m-1-1) edge node[left] {$s$} (m-2-1)
										(m-1-2) edge[barred] node[above] {$J$} (m-1-3)
														edge node[right] {$r$} (m-2-2)
										(m-1-3) edge node[right] {$d$} (m-2-3);
				\path				(m-1-1) edge[eq] (m-1-2)
										(m-2-1) edge[eq] (m-2-2)
										(m-2-2) edge[eq] (m-2-3);
				\path[transform canvas={shift=($(m-1-1)!0.5!(m-2-2)$)}]
										(m-1-2) edge[cell] node[right] {$\phi'$} (m-2-2)
										(m-1-3) edge[cell] node[right] {$\eps$} (m-2-3);
			\end{tikzpicture}
		\end{displaymath}
	\end{definition}
	As usual any two cells defining the same right Kan extension factor uniquely through each other as invertible vertical cells. We remark that the definition of Grandis and Par\'e is in fact more general, as it allows cells $\eps$ whose horizontal target is arbitrary where we assume it to be a horizontal unit.
	
	As is shown by the discussion above, the preceding definition generalises the notion of right Kan extension in $2$-categories as follows.
	\begin{proposition} \label{right Kan extensions along conjoints as right Kan extensions in V(K)}
		In a double category $\K$ consider a vertical cell $\eps$, as on the left below, and assume that the conjoint $\hmap{j^*}BA$ of $j$ exists.
		\begin{displaymath}
			\begin{tikzpicture}[textbaseline]
    		\matrix(m)[math175em]{A & A \\ B & \phantom A \\ M & M \\};
    		\path[map]  (m-1-1) edge node[left] {$j$} (m-2-1)
       			        (m-1-2) edge node[right] {$d$} (m-3-2)
       			        (m-2-1) edge node[left] {$r$} (m-3-1);
      	\path				(m-1-1) edge[eq] (m-1-2)
      							(m-3-1) edge[eq] (m-3-2);
    		\path[transform canvas={shift={($(m-2-1)!0.5!(m-1-1)$)}}] (m-2-2) edge[cell] node[right] {$\eps$} (m-3-2);
  		\end{tikzpicture} = \begin{tikzpicture}[textbaseline]
				\matrix(m)[math175em]{A & A \\ B & A \\ M & M \\};
				\path[map]	(m-1-1) edge node[left] {$j$} (m-2-1)
										(m-2-1) edge[barred] node[below] {$j^*$} (m-2-2)
														edge node[left] {$r$} (m-3-1)
										(m-2-2) edge node[right] {$d$} (m-3-2);
				\path				(m-1-1) edge[eq] (m-1-2)
										(m-1-2) edge[eq] (m-2-2)
										(m-3-1) edge[eq] (m-3-2);
				\path[transform canvas={shift=(m-2-1), yshift=-2pt}]	(m-2-2) edge[cell] node[right] {$\eps'$} (m-3-2);
				\draw				($(m-1-1)!0.5!(m-2-2)$)	node {\oc};
			\end{tikzpicture}
		\end{displaymath}
		The vertical cell $\eps$ defines $r$ as the right Kan extension of $d$ along $j$ in $V(\K)$ precisely if its factorisation $\eps'$, as shown, defines $r$ as the right Kan extension of $d$ along $j^*$ in $\K$.
	\end{proposition}
	
	\subsection{Pointwise Kan extensions along enriched functors.} \label{subsection: enriched Kan extensions.}
	Having recalled the notion of ordinary right Kan extension, our aim is now to give a notion of pointwise Kan extension in double categories that extends the notion of pointwise Kan extension for enriched functors. In this section we recall the latter notion, which is extended to double categories in the next section.
	
	To begin we recall a notion of pointwise Kan extension for unenriched functors. Given functors $\map jAB$, $\map dAM$ and $\map rBM$ between small categories, a natural transformation $\nat\eps{r \of j}d$ defines $r$ as the pointwise right Kan extension of $d$ along $j$ if the maps
	\begin{equation} \label{unenriched pointwise Kan extension}
		M(m, rb) \to \inhom{A, \Set}\bigpars{B(b, j\dash), M(m, d\dash)},
	\end{equation}
	which assign to $\map um{rb}$ the natural transformation with components
	\begin{equation} \label{unenriched pointwise Kan extension underlying map}
		B(b, ja) \xrar r M(rb, rja) \xrar{M(u, \eps_a)} M(m, da),
	\end{equation}
	are bijections. That this coincides with the usual notion follows easily from the fact that the target of \eqref{unenriched pointwise Kan extension} is the set of cones from $m$ to the functor $b \slash j \to A \xrar d M$, where $b \slash j$ denotes the comma category. For details see Section X.5 of \cite{MacLane98}.
	
	Given a closed symmetric monoidal category $\V = (\V, \tens, I, \mathfrak s, \inhom{\dash, \dash})$ that is complete and cocomplete, we now assume that the functors $j$, $d$ and $r$ are $\V$-functors, and that $\nat\eps{r \of j}d$ is a $\V$-natural transformation. Then we can consider enriched variants of \eqref{unenriched pointwise Kan extension}, which are $\V$-maps
	\begin{equation} \label{enriched pointwise Kan extension}
		M(m, rb) \to \inhom{A, \V}\bigpars{B(b, j\dash), M(m, d\dash)},
	\end{equation}
	as we shall explain, and $\eps$ defines the $\V$-functor $r$ as the pointwise right Kan extension of $d$ along $j$ if these are isomorphisms; see the last condition of Theorem 4.6 of \cite{Kelly82}, which lists equivalent conditions on $\eps$ for it to define $r$ as the pointwise right Kan extension. Pointwise Kan extensions along enriched functors were first introduced by Dubuc in Section I.4 of \cite{Dubuc70}.
	
	To describe \eqref{enriched pointwise Kan extension} we first have to recall the definition of their targets, which form the `$\V$-objects of $\V$-natural transformations' $B(b, j\dash) \Rar M(m, d\dash)$. We shall not do this directly, but instead recall the definition of the `right hom' $\hmap{K \rhom H}AB$ for any pair of $\V$-profunctors $\hmap HBE$ and $\hmap KAE$; the targets of \eqref{enriched pointwise Kan extension} are then given by the right hom $\hmap{d^* \rhom j^*}MB$ of the conjoints $\hmap{j^*}BA$ and $\hmap{d^*}MA$.
	
	On objects we define $K \rhom H$ by the ends $(K \rhom H)(a, b) = \int_E \inhom{H(b, \dash), K(a, \dash)}$, where $\inhom{\dash, \dash}$ is the inner hom of $\V$. Dually to \eqref{composition of enriched profunctors}, these ends can be taken to be the equalisers of the $\V$-maps
	\begin{displaymath}
    \prod\limits_{e \in E} \inhom{H(b, e), K(a,e)} \rightrightarrows \prod\limits_{e_1, e_2 \in E} \inhom{H(b, e_1) \tens E(e_1, e_2), K(a,e_2)}
  \end{displaymath}
  that are induced by letting $E(e_1, e_2)$ act on $H(b, e_1)$ and $K(a, e_1)$ respectively. The universal properties of these ends ensure that they combine to form a $\V$-profunctor $\hmap{K \rhom H}AB$.
  
  Using the fact that ends are dual to coends, as well as the tensor-hom adjunction $\dash \tens X \ladj \inhom{X, \dash}$ of $\V$, it is straightforward to obtain a correspondence of $\V$-natural transformations
  \begin{displaymath}
  	\cell\phi{J\hc H}K \qquad \leftrightarrow \qquad \cell\psi J{K \rhom H},
  \end{displaymath}
  for any $\hmap JAB$. To be precise we obtain, for all $\hmap HBE$, an adjunction
  \begin{equation}\label{left hom adjunction}
  	\begin{tikzpicture}[textbaseline]
			\matrix(m)[math, column sep=1.75em]{\dash \hc H \colon H(\enProf\V)(A, B) & H(\enProf\V)(A, E) \colon \dash \rhom H, \\ };
			\path[transform canvas={yshift=3.5pt}, map] (m-1-1) edge (m-1-2);
			\path[transform canvas={yshift=-3.5pt}, map]	(m-1-2) edge (m-1-1);
			\path[color=white] (m-1-1) edge node[color=black, font=\scriptsize] {$\bot$} (m-1-2);
		\end{tikzpicture}
  \end{equation}
  where $H(\enProf\V)$ denotes the horizontal bicategory contained in the double category $\enProf \V$, that is the bicategory of $\V$-categories, $\V$-profunctors and their transformations. 
  
  We now return to describing the $\V$-maps \eqref{enriched pointwise Kan extension}. First recall that
  \begin{displaymath}
  	\inhom{A, \V}\bigpars{B(b, j\dash), M(m, d\dash)} = \int_A \inhom{B(b, j\dash), M(m, d\dash)} = (d^* \rhom j^*)(m, b);
  \end{displaymath}
  see Section 2.2 of \cite{Kelly82}. Next consider the $\V$-natural transformation $\nat\eps{r\of j}d$ as a vertical cell in $\enProf\V$, as on the left below.
  \begin{equation} \label{corresponding cells in a right-closed equipment}
  	\begin{tikzpicture}[textbaseline]
    	\matrix(m)[math175em]{A & A \\ B & \phantom B \\ M & M \\};
    	\path[map]  (m-1-1) edge node[left] {$j$} (m-2-1)
       		        (m-1-2) edge node[right] {$d$} (m-3-2)
       		        (m-2-1) edge node[left] {$r$} (m-3-1);
      \path				(m-1-1) edge[eq] (m-1-2)
      						(m-3-1) edge[eq] (m-3-2);
    	\path[transform canvas={shift={($(m-2-1)!0.5!(m-1-1)$)}}] (m-2-2) edge[cell] node[right] {$\eps$} (m-3-2);
  	\end{tikzpicture} \qquad\qquad \begin{tikzpicture}[textbaseline]
			\matrix(m)[math175em]{M & B & A \\ M & & A \\};
			\path[map]	(m-1-1) edge[barred] node[above] {$r^*$} (m-1-2)
									(m-1-2) edge[barred] node[above] {$j^*$} (m-1-3)
									(m-2-1) edge[barred] node[below] {$d^*$} (m-2-3)
									(m-1-2) edge[cell] node[right] {$\eps^*$} (m-2-2);
			\path				(m-1-1) edge[eq] (m-2-1)
									(m-1-3) edge[eq] (m-2-3);
		\end{tikzpicture} \qquad\qquad \begin{tikzpicture}[textbaseline]
  		\matrix(m)[math175em]{M & B \\ M & B \\};
  		\path[map]  (m-1-1) edge[barred] node[above] {$r^*$} (m-1-2)
									(m-2-1) edge[barred] node[below] {$d^* \rhom j^*$} (m-2-2);
			\path				(m-1-1) edge[eq] (m-2-1)
									(m-1-2) edge[eq] (m-2-2);
			\path[transform canvas={shift={($(m-1-2)!(0,0)!(m-2-2)$)}, xshift=-0.9em}] (m-1-1) edge[cell] node[right] {$(\eps^*)^\flat$} (m-2-1);
		\end{tikzpicture}
  \end{equation}
  It factors uniquely, through the opcartesian cell defining $r^* \hc j^*$ as the conjoint of $r \of j$ and the cartesian cell defining $d^*$ as the conjoint of $d$, as a horizontal cell $\eps^*$ as in the middle; compare the definition of $\phi^*$ in \eqref{companion and conjoint factorisations}. Remembering from the proof of \lemref{cartesian and opcartesian cells in terms of companions and conjoints} that this factorisation is obtained by composing $\eps$ on the left with the cartesian cells $\eps_j$ and $\eps_r$ defining $j^*$ and $r^*$, and on the right with the opcartesian cell $\eta_d$ defining $d^*$, we find that $\eps^*$ is induced by the $\V$-maps
  \begin{multline*}
  	\coprod_{b \in B} M(m, rb) \tens B(b, ja) \xrar{\coprod\id \tens r} \coprod_{b \in B} M(m, rb) \tens M(rb, rja) \\ \xrar\of M(m, rja) \xrar{M(\id, \eps_a)} M(m, da);
  \end{multline*}
  compare this with the components \eqref{unenriched pointwise Kan extension underlying map} that are used in the unenriched case. The cell $\eps^*$ in turn uniquely corresponds to a horizontal cell $(\eps^*)^\flat$ as on the right of \eqref{corresponding cells in a right-closed equipment}, under the adjunction \eqref{left hom adjunction}, and the components of $(\eps^*)^\flat$ form the $\V$-maps \eqref{enriched pointwise Kan extension}; see the last condition of Theorem~4.6 of \cite{Kelly82}.
  
  The following definition summarises this section.
  \begin{definition} \label{definition: pointwise Kan extension along enriched functors}
  	Let $\V$ be a closed symmetric monoidal category that is complete and cocomplete. A $\V$-natural transformation $\nat\eps{r \of j}d$ of $\V$-functors, as on the left of \eqref{corresponding cells in a right-closed equipment}, defines the $\V$-functor $r$ as the \emph{pointwise right Kan extension of $d$ along $j$} if the corresponding $\V$-natural transformation of $\V$-profunctors $\cell{(\eps^*)^\flat}{r^*}{d^* \rhom j^*}$, as on the right of \eqref{corresponding cells in a right-closed equipment}, is invertible.
  \end{definition}
  
  \subsection{Pointwise Kan extensions in double categories.}
  In this section we, by mimicking the ideas of the previous section, obtain a notion of pointwise right Kan extension in `right closed' equipments, which we extend to general double categories afterwards. The first notion is closely related to Wood's definition of `indexed limit' in bicategories that are, in his sense, `equipped with abstract proarrows', as introduced in \cite{Wood82}.
  
  More precisely, by restricting Wood's notion of bicategories equipped with abstract proarrows to $2$-categories, one obtains a notion that is equivalent to that of equipments in our sense which, additionally, are both right closed (see below) and `left closed'; for details see Appendix~C of \cite{Shulman08}. In such `$2$-categories equipped with abstract proarrows' Wood's notion of indexed limit coincides with \defref{definition: pointwise right Kan extension in right closed equipments} given below.
  \begin{definition}
  	A double category $\K$ is called \emph{right closed} if, for each horizontal morphism $\hmap HBE$, the functor
  	\begin{displaymath}
  		\map{\dash \hc H}{H(\K)(A, B)}{H(\K)(A, E)}
  	\end{displaymath}
  	has a right adjoint $\dash \rhom H$.
  \end{definition}
  
  As in \eqref{corresponding cells in a right-closed equipment}, in a right closed equipment $\K$ every cell $\eps$ on the left below corresponds to the horizontal cell $\eps^*$ in the middle, which in turn corresponds to a horizontal cell $(\eps^*)^\flat$ on the right, under the adjunction $\dash \hc J \ladj \dash \rhom J$.
  \begin{displaymath}
  	\begin{tikzpicture}[baseline]
  			\matrix(m)[math175em]{A & B \\ M & M \\};
  			\path[map]  (m-1-1) edge[barred] node[above] {$J$} (m-1-2)
														edge node[left] {$r$} (m-2-1)
										(m-1-2) edge node[right] {$d$} (m-2-2);
				\path				(m-2-1) edge[eq] (m-2-2);
				\path[transform canvas={shift={($(m-1-2)!(0,0)!(m-2-2)$)}}] (m-1-1) edge[cell] node[right] {$\eps$} (m-2-1);
			\end{tikzpicture} \qquad\qquad \begin{tikzpicture}[baseline]
			\matrix(m)[math175em]{M & A & B \\ M & & B \\};
			\path[map]	(m-1-1) edge[barred] node[above] {$r^*$} (m-1-2)
									(m-1-2) edge[barred] node[above] {$J$} (m-1-3)
									(m-2-1) edge[barred] node[below] {$d^*$} (m-2-3)
									(m-1-2) edge[cell] node[right] {$\eps^*$} (m-2-2);
			\path				(m-1-1) edge[eq] (m-2-1)
									(m-1-3) edge[eq] (m-2-3);
		\end{tikzpicture} \qquad\qquad \begin{tikzpicture}[baseline]
  		\matrix(m)[math175em]{M & A \\ M & A \\};
  		\path[map]  (m-1-1) edge[barred] node[above] {$r^*$} (m-1-2)
									(m-2-1) edge[barred] node[below] {$d^* \rhom J$} (m-2-2);
			\path				(m-1-1) edge[eq] (m-2-1)
									(m-1-2) edge[eq] (m-2-2);
			\path[transform canvas={shift={($(m-1-2)!(0,0)!(m-2-2)$)}, xshift=-0.9em}] (m-1-1) edge[cell] node[right] {$(\eps^*)^\flat$} (m-2-1);
		\end{tikzpicture}
  \end{displaymath}
  
	Mimicking \defref{definition: pointwise Kan extension along enriched functors} we make the following definition. It is closely related to Wood's definition of indexed limit, as previously discussed.
  \begin{definition} \label{definition: pointwise right Kan extension in right closed equipments}
		Let $\map dBM$ and $\hmap JAB$ be morphisms in a right closed equipment. The cell $\eps$ on the left above is said to define $r$ as the \emph{pointwise right Kan extension of $d$ along $J$} if the corresponding horizontal cell $(\eps^*)^\flat$ on the right is invertible.
  \end{definition}
  
  \begin{example}
  	As we have seen, in the case that $\K = \enProf\V$ and $J = j^*$, for a $\V$-functor \mbox{$\map jBA$}, the $\V$-natural transformation $\eps$ above defines $r$ as the pointwise right Kan extension of $J$ along $d$ precisely if the corresponding $\V$-natural transformation \mbox{$\nat{\eps \of \eta_j}r{d \of j}$} defines $r$ as the pointwise right Kan extension of $d$ along $j$ in $\enCat\V = V(\enProf\V)$.
  	
  	Similarly $\V$-weighted limits can be defined in $\enProf\V$, as follows. Given a $\V$-weight $\map jB\V$, notice that it can be regarded as a $\V$-profunctor $\hmap J1B$, where $1$ denotes the `unit' $\V$-category, consisting of one object $*$ and $1(*, *) = I$, the unit of $\V$. As with enriched Kan extensions, one can show that the $j$-weighted limit of a $\V$-functor $\map dBM$ exists if and only if the pointwise right Kan extension of $d$ along $J$ exists in $\enProf\V$.
  \end{example}
  
  In order to extend \defref{definition: pointwise right Kan extension in right closed equipments} to general double categories we restate its condition `$(\eps^*)^\flat$ is invertible' in terms that do not use conjoints or the right adjoint $\dash \rhom J$, as follows. It is easy to see that $(\eps^*)^\flat$ is invertible if and only if $\eps^*$ forms a `universal arrow $\dash \hc J \Rar d^*$', in the sense of Section III.1 of \cite{MacLane98}. This means that every cell $\psi$ as on the left below factors uniquely through $\eps^*$ as shown.
  \begin{displaymath}
  	\begin{tikzpicture}[textbaseline]
			\matrix(m)[math175em]{M & A & B \\ M & & B \\};
			\path[map]	(m-1-1) edge[barred] node[above] {$K$} (m-1-2)
									(m-1-2) edge[barred] node[above] {$J$} (m-1-3)
									(m-2-1) edge[barred] node[below] {$d^*$} (m-2-3)
									(m-1-2) edge[cell] node[right] {$\psi$} (m-2-2);
			\path				(m-1-1) edge[eq] (m-2-1)
									(m-1-3) edge[eq] (m-2-3);
		\end{tikzpicture} = \begin{tikzpicture}[textbaseline]
  		\matrix(m)[math175em]{M & A & B \\ M & A & B \\ M & & B \\};
  		\path[map]	(m-1-1) edge[barred] node[above] {$K$} (m-1-2)
  								(m-1-2) edge[barred] node[above] {$J$} (m-1-3)
  								(m-2-1) edge[barred] node[below] {$r^*$} (m-2-2)
  								(m-2-2) edge[barred] node[below] {$J$} (m-2-3)
  								(m-3-1) edge[barred] node[below] {$d^*$} (m-3-3)
  								(m-2-2) edge[cell] node[right] {$\eps^*$} (m-3-2);
  		\path				(m-1-1) edge[eq] (m-2-1)
  								(m-1-2) edge[eq] (m-2-2)
  								(m-1-3) edge[eq] (m-2-3)
  								(m-2-1) edge[eq] (m-3-1)
  								(m-2-3) edge[eq] (m-3-3);
  		\path[transform canvas={shift={($(m-2-2)!0.5!(m-2-3)$)}}] (m-1-1) edge[cell] node[right] {$\psi'$} (m-2-1);
  	\end{tikzpicture}
  \end{displaymath}
  But the unique factorisations above correspond precisely to the unique factorisations on the left below, under the correspondence $\psi \leftrightarrow \psi^*$ given in \eqref{companion and conjoint factorisations}. Here we use the functoriality of the latter, see \lemref{companion and conjoint factorisations functoriality}.
  \begin{displaymath}
  	\begin{tikzpicture}[textbaseline]
			\matrix(m)[math175em]{M & A & B \\ M & & M \\};
			\path[map]	(m-1-1) edge[barred] node[above] {$K$} (m-1-2)
									(m-1-2) edge[barred] node[above] {$J$} (m-1-3)
									(m-1-3) edge node[right] {$d$} (m-2-3)
									(m-1-2) edge[cell] node[right] {$\xi$} (m-2-2);
			\path				(m-1-1) edge[eq] (m-2-1)
									(m-2-1) edge[eq] (m-2-3);
		\end{tikzpicture} = \begin{tikzpicture}[textbaseline]
			\matrix(m)[math175em]{M & A & B \\ M & M & M \\};
			\path[map]	(m-1-1) edge[barred] node[above] {$K$} (m-1-2)
									(m-1-2) edge[barred] node[above] {$J$} (m-1-3)
													edge node[right] {$r$} (m-2-2)
									(m-1-3) edge node[right] {$d$} (m-2-3);
			\path				(m-1-1) edge[eq] (m-2-1)
									(m-2-1) edge[eq] (m-2-2)
									(m-2-2) edge[eq] (m-2-3);
			\path[transform canvas={shift={($(m-1-2)!0.5!(m-2-3)$)}}] (m-1-1) edge[cell] node[right] {$\xi'$} (m-2-1)
									(m-1-2) edge[cell] node[right] {$\eps$} (m-2-2);
		\end{tikzpicture} \qquad\qquad\quad \begin{tikzpicture}[textbaseline]
			\matrix(m)[math175em]{C & A & B \\ M & A & B \\};
			\path[map]	(m-1-1) edge[barred] node[above] {$H$} (m-1-2)
													edge node[left] {$s$} (m-2-1)
									(m-1-2) edge[barred] node[above] {$J$} (m-1-3)
									(m-2-1) edge[barred] node[below] {$K$} (m-2-2)
									(m-2-2) edge[barred] node[below] {$J$} (m-2-3);
			\path				(m-1-2) edge[eq] (m-2-2)
									(m-1-3) edge[eq] (m-2-3);
			\draw				($(m-1-1)!0.5!(m-2-2)$) node {\oc};		
		\end{tikzpicture}
  \end{displaymath}
  
  Furthermore, the unique factorisation of the cells $\xi$ through $\eps$, on the left above, is equivalent to the unique factorisation of the more general cells $\phi$ through $\eps$, as shown below. Indeed, the one below reduces to the one above by taking $H = K$ and $s = \id_M$. On the other hand, unique factorisations of the form below correspond precisely to ones of the form on the left above, under the unique factorisation through the composite on the right, where $K$ is any choice of extension of $H$ along $s$; this composite is opcartesian by \lemref{composing cartesian and opcartesian cells}.
  \begin{equation} \label{pointwise right Kan extension factorisation}
  	\begin{tikzpicture}[textbaseline]
			\matrix(m)[math175em]{C & A & B \\ M & & M \\};
			\path[map]	(m-1-1) edge[barred] node[above] {$H$} (m-1-2)
													edge node[left] {$s$} (m-2-1)
									(m-1-2) edge[barred] node[above] {$J$} (m-1-3)
									(m-1-3) edge node[right] {$d$} (m-2-3)
									(m-1-2) edge[cell] node[right] {$\phi$} (m-2-2);
			\path				(m-2-1) edge[eq] (m-2-3);
		\end{tikzpicture} = \begin{tikzpicture}[textbaseline]
			\matrix(m)[math175em]{C & A & B \\ M & M & M \\};
			\path[map]	(m-1-1) edge[barred] node[above] {$H$} (m-1-2)
													edge node[left] {$s$} (m-2-1)
									(m-1-2) edge[barred] node[above] {$J$} (m-1-3)
													edge node[right] {$r$} (m-2-2)
									(m-1-3) edge node[right] {$d$} (m-2-3);
			\path				(m-2-1) edge[eq] (m-2-2)
									(m-2-2) edge[eq] (m-2-3);
			\path[transform canvas={shift={($(m-1-2)!0.5!(m-2-3)$)}}] (m-1-1) edge[cell] node[right] {$\phi'$} (m-2-1)
									(m-1-2) edge[cell] node[right] {$\eps$} (m-2-2);
		\end{tikzpicture}
  \end{equation}
  We conclude that, in the case that $\K$ is a right closed equipment, the definition below coincides with \defref{definition: pointwise right Kan extension in right closed equipments}.
  \begin{definition} \label{definition: pointwise right Kan extension}
  	Let $\map dBM$ and $\hmap JAB$ be morphisms in a double category. The cell $\eps$ in the right-hand side above is said to define $r$ as the \emph{pointwise right Kan extension of $d$ along $J$} if every cell $\phi$ above factors uniquely through $\eps$ as shown.
  \end{definition}
  
  Like ordinary Kan extensions, two cells defining the same pointwise Kan extension factor uniquely through each other as invertible vertical cells. We remark that Grandis and Par\'e introduced in \cite{Grandis-Pare08} a notion similar to the above, that of `absolute' right Kan extension, which is stronger in that it allows the cells $\phi$ and $\phi'$ in \eqref{pointwise right Kan extension factorisation}  to have any, not necessarily unital, horizontal target.
  
  Pointwise left Kan extensions in $\K$ are simply pointwise right Kan extensions in $\co\K$ as follows, where $\co\K$ is the horizontal dual of $\K$; see \defref{definition: duals of double categories}.
  \begin{definition}
  	Let $\map dAM$ and $\hmap JAB$ be morphisms in a double category. The cell $\eta$ in the right-hand side below defines $l$ as the \emph{pointwise left Kan extension of $d$ along $J$} if every cell $\phi$ below factors uniquely through $\eta$ as shown.
  	\begin{displaymath}
  		\begin{tikzpicture}[textbaseline]
				\matrix(m)[math175em]{A & B & C \\ M & & M \\};
				\path[map]	(m-1-1) edge[barred] node[above] {$J$} (m-1-2)
														edge node[left] {$d$} (m-2-1)
										(m-1-2) edge[barred] node[above] {$H$} (m-1-3)
										(m-1-3) edge node[right] {$k$} (m-2-3)
										(m-1-2) edge[cell] node[right] {$\phi$} (m-2-2);
				\path				(m-2-1) edge[eq] (m-2-3);
			\end{tikzpicture} = \begin{tikzpicture}[textbaseline]
				\matrix(m)[math175em]{A & B & C \\ M & M & M \\};
				\path[map]	(m-1-1) edge[barred] node[above] {$J$} (m-1-2)
														edge node[left] {$d$} (m-2-1)
										(m-1-2) edge[barred] node[above] {$H$} (m-1-3)
														edge node[right] {$l$} (m-2-2)
										(m-1-3) edge node[right] {$k$} (m-2-3);
				\path				(m-2-1) edge[eq] (m-2-2)
										(m-2-2) edge[eq] (m-2-3);
				\path[transform canvas={shift={($(m-1-2)!0.5!(m-2-3)$)}}] (m-1-1) edge[cell] node[right] {$\eta$} (m-2-1)
										(m-1-2) edge[cell] node[right] {$\phi'$} (m-2-2);
			\end{tikzpicture}
  	\end{displaymath}
  \end{definition}
  
  The following two lemmas record some simple properties of pointwise Kan extensions. The second one, on iterated pointwise Kan extensions, generalises the analogous result for pointwise Kan extensions along enriched functors; see Theorem 4.47 of \cite{Kelly82}. Its proof is straightforward and omitted.
	\begin{lemma} \label{pointwise right Kan extensions are right Kan extensions}
	The cell $\eps$ below defines $r$ as a right Kan extension whenever it defines $r$ as a pointwise right Kan extension. The converse holds once the pointwise right Kan extension of $d$ along $J$ is known to exist.
		\begin{displaymath}
			\begin{tikzpicture}
  			\matrix(m)[math175em]{A & B \\ M & M \\};
  			\path[map]  (m-1-1) edge[barred] node[above] {$J$} (m-1-2)
														edge node[left] {$r$} (m-2-1)
										(m-1-2) edge node[right] {$d$} (m-2-2);
				\path				(m-2-1) edge[eq] (m-2-2);
				\path[transform canvas={shift={($(m-1-2)!(0,0)!(m-2-2)$)}}] (m-1-1) edge[cell] node[right] {$\eps$} (m-2-1);
		\end{tikzpicture}
		\end{displaymath}
	\end{lemma}
	\begin{proof}
		For the converse remember that any two cells defining the same right Kan extension factor through each other as invertible vertical cells, so that if one of them defines a pointwise right Kan extension then so does the other.
	\end{proof}
	
	\begin{lemma}[Pasting lemma]
		Consider horizontally composable cells
		\begin{displaymath}
			\begin{tikzpicture}
				\matrix(m)[math175em]{A & B & C \\ M & M & M \\};
				\path[map]	(m-1-1) edge[barred] node[above] {$J$} (m-1-2)
														edge node[left] {$s$} (m-2-1)
										(m-1-2) edge[barred] node[above] {$H$} (m-1-3)
														edge node[right] {$r$} (m-2-2)
										(m-1-3) edge node[right] {$d$} (m-2-3);
				\path	(m-2-1) edge[eq] (m-2-2)
							(m-2-2) edge[eq] (m-2-3);
				\path[transform canvas={shift={($(m-1-2)!0.5!(m-2-3)$)}}] (m-1-1) edge[cell] node[right] {$\gamma$} (m-2-1)
				(m-1-2) edge[cell] node[right] {$\eps$} (m-2-2);
			\end{tikzpicture}
		\end{displaymath}
		in a double category, and suppose that $\eps$ defines $r$ as the pointwise right Kan extension of $d$ along $H$. Then $\gamma$ defines $s$ as the (pointwise) right Kan extension of $r$ along $J$ precisely if $\gamma \hc \eps$ defines $s$ as the (pointwise) right Kan extension of $d$ along $J \hc H$.
	\end{lemma}
	
	We close this section with a lemma showing that `precomposition is a form of pointwise right Kan extension', which will be very useful in the next section.
	\begin{lemma} \label{precomposition}
  	Consider a composable pair $\map fAC$ and $\map rCM$ in a double category, and assume that the companion $\hmap {f_*}AC$ of $f$ exists. The composite below defines $r \of f$ as the pointwise right Kan extension of $r$ along $f_*$.
  	\begin{displaymath}
  		\begin{tikzpicture}
  			\matrix(m)[math175em]{A & C \\ C & C \\ M & M \\};
  			\path[map]	(m-1-1) edge[barred] node[above] {$f_*$} (m-1-2)
  													edge node[left] {$f$} (m-2-1)
  									(m-2-1) edge node[left] {$r$} (m-3-1)
  									(m-2-2) edge node[right] {$r$} (m-3-2);
  			\path				(m-1-2) edge[eq] (m-2-2)
  									(m-2-1) edge[eq] (m-2-2)
  									(m-3-1) edge[eq] (m-3-2);
  			\draw				($(m-1-1)!0.5!(m-2-2)$) node {\textup{cart}};
  		\end{tikzpicture}
  	\end{displaymath}
  \end{lemma}
  \begin{proof}
  	We have to show that every cell $\phi$ as on the left-hand side below factors uniquely through the composite above. That such a factorisation exists is shown by the following identity, which is obtained by using one of the companion identities for $f_*$.
  	\begin{displaymath}
  		\begin{tikzpicture}[textbaseline]
  			\matrix(m)[math175em]{D & A & C \\ M & \phantom M & M \\};
  			\path[map]	(m-1-1) edge[barred] node[above] {$H$} (m-1-2)
  													edge node[left] {$s$} (m-2-1)
  									(m-1-2) edge[barred] node[above] {$f_*$} (m-1-3)
  									(m-1-3) edge node[right] {$r$} (m-2-3);
  			\path				(m-2-1) edge[eq] (m-2-3)
  									(m-1-2) edge[cell] node[right] {$\phi$} (m-2-2);
  		\end{tikzpicture} = \begin{tikzpicture}[textbaseline]
  			\matrix(m)[math175em]{D & A & A & C \\ D & A & C & C \\ M & \phantom M & M & M \\};
  			\path[map]	(m-1-1) edge[barred] node[above] {$H$} (m-1-2)
  									(m-1-3) edge[barred] node[above] {$f_*$} (m-1-4)
  													edge node[right, inner sep=1pt] {$f$} (m-2-3)
  									(m-2-1) edge[barred] node[above] {$H$} (m-2-2)
  													edge node[left] {$s$} (m-3-1)
  									(m-2-2) edge[barred] node[below] {$f_*$} (m-2-3)
  									(m-2-3) edge node[right] {$r$} (m-3-3)
  									(m-2-4) edge node[right] {$r$} (m-3-4);
  			\path				(m-1-1) edge[eq] (m-2-1)
  									(m-1-2) edge[eq] (m-1-3)
  													edge[eq] (m-2-2)
  									(m-1-4) edge[eq] (m-2-4)
  									(m-2-3) edge[eq] (m-2-4)
  									(m-3-1) edge[eq] (m-3-3)
  									(m-3-3) edge[eq] (m-3-4)
  									(m-2-2) edge[cell] node[right] {$\phi$} (m-3-2);
  			\draw				($(m-1-2)!0.5!(m-2-3)$) node {\oc}
  									($(m-1-3)!0.5!(m-2-4)$) node {\cc};
  		\end{tikzpicture}
  	\end{displaymath}
  	To show that this factorisation is unique assume that there exists another cell $\phi'$, as on the left below, such that $\phi = \phi' \hc (1_r \of \textup{cart})$. Then the following equation, which uses the other companion identity for $f_*$, shows that $\phi'$ coincides with the factorisation that was obtained above.
  	\begin{displaymath}
  		\begin{tikzpicture}[textbaseline]
  			\matrix(m)[math175em]{D & A \\ & C \\ M & M \\};
  			\path[map]	(m-1-1) edge[barred] node[above] {$H$} (m-1-2)
  													edge node[left] {$s$} (m-3-1)
  									(m-1-2) edge node[right] {$f$} (m-2-2)
  									(m-2-2) edge node[right] {$r$} (m-3-2);
  			\path				(m-3-1) edge[eq] (m-3-2);
  			\path[transform canvas={shift={($(m-2-2)!0.5!(m-3-2)$)}}]	(m-1-1) edge[cell] node[right] {$\phi'$} (m-2-1);
  		\end{tikzpicture} = \begin{tikzpicture}[textbaseline]
  			\matrix(m)[math175em]{D & A & A \\ D & A & C \\ & C & C \\ M & M & M \\};
  			\path[map]	(m-1-1) edge[barred] node[above] {$H$} (m-1-2)
  									(m-1-3) edge node[right] {$f$} (m-2-3)
  									(m-2-1) edge[barred] node[above] {$H$} (m-2-2)
  													edge node[left] {$s$} (m-4-1)
  									(m-2-2) edge[barred] node[below] {$f_*$} (m-2-3)
  													edge node[left, inner sep=1pt] {$f$} (m-3-2)
  									(m-3-2) edge node[right] {$r$} (m-4-2)
  									(m-3-3) edge node[right] {$r$} (m-4-3);
  			\path				(m-1-1) edge[eq] (m-2-1)
  									(m-1-2) edge[eq] (m-1-3)
  													edge[eq] (m-2-2)
  									(m-2-3) edge[eq] (m-3-3)
  									(m-3-2) edge[eq] (m-3-3)
  									(m-4-1) edge[eq] (m-4-2)
  									(m-4-2) edge[eq] (m-4-3);
  			\path[transform canvas={shift={($(m-3-2)!0.5!(m-3-3)$)}}]	(m-2-1) edge[cell] node[right] {$\phi'$} (m-3-1);
  			\draw				($(m-1-2)!0.5!(m-2-3)$) node {\oc}
  									($(m-2-2)!0.5!(m-3-3)$) node {\cc};
  		\end{tikzpicture} = \begin{tikzpicture}[textbaseline]
  			\matrix(m)[math175em]{D & A & A \\ D & A & C \\ M & \phantom M & M \\};
  			\path[map]	(m-1-1) edge[barred] node[above] {$H$} (m-1-2)
  									(m-1-3) edge node[right] {$f$} (m-2-3)
  									(m-2-1) edge[barred] node[above] {$H$} (m-2-2)
  													edge node[left] {$s$} (m-3-1)
  									(m-2-2) edge[barred] node[below] {$f_*$} (m-2-3)
  									(m-2-3) edge node[right] {$r$} (m-3-3);
  			\path				(m-1-1) edge[eq] (m-2-1)
  									(m-1-2) edge[eq] (m-1-3)
  													edge[eq] (m-2-2)
  									(m-3-1) edge[eq] (m-3-3)
  									(m-2-2) edge[cell] node[right] {$\phi$} (m-3-2);
  			\draw				($(m-1-2)!0.5!(m-2-3)$) node {\oc};
  		\end{tikzpicture}
  	\end{displaymath}
  	This completes the proof.
  \end{proof}
	
	\section{Exact cells}
	In this section we introduce the notion of `exact cell' in double categories. In the double category $\Prof$ of unenriched profunctors exact cells generalise the `carr\'es exacts' that were studied by Guitart in \cite{Guitart80}, as we shall see in \exref{example: carres exacts}.
	
	\begin{definition} \label{definition: exact cell}
		In a double category consider a cell $\phi$ as on the left below. We call $\phi$ \emph{(pointwise) right exact} if for any cell $\eps$, as on the right, that defines $r$ as the (pointwise) right Kan extension of $d$ along $K$, the composite $\eps \of \phi$ defines $r \of f$ as the (pointwise) right Kan extension of $d \of g$ along $J$. If the converse holds as well then we call $\phi$ \emph{(pointwise) initial}.
		\begin{displaymath}
			\begin{tikzpicture}[baseline]
				\matrix(m)[math175em]{A & B \\ C & D \\};
				\path[map]  (m-1-1) edge[barred] node[above] {$J$} (m-1-2)
														edge node[left] {$f$} (m-2-1)
										(m-1-2) edge node[right] {$g$} (m-2-2)
										(m-2-1) edge[barred] node[below] {$K$} (m-2-2);
				\path[transform canvas={shift={($(m-1-2)!(0,0)!(m-2-2)$)}}] (m-1-1) edge[cell] node[right] {$\phi$} (m-2-1);
			\end{tikzpicture} \qquad\qquad\qquad\qquad \begin{tikzpicture}[baseline]
  			\matrix(m)[math175em]{C & D \\ M & M \\};
  			\path[map]  (m-1-1) edge[barred] node[above] {$K$} (m-1-2)
														edge node[left] {$r$} (m-2-1)
										(m-1-2) edge node[right] {$d$} (m-2-2);
				\path				(m-2-1) edge[eq] (m-2-2);
				\path[transform canvas={shift={($(m-1-2)!(0,0)!(m-2-2)$)}}] (m-1-1) edge[cell] node[right] {$\eps$} (m-2-1);
			\end{tikzpicture}
		\end{displaymath}
	\end{definition}
	Notice that both the class of (pointwise) right exact cells, as well as that of (pointwise) initial cells, are closed under vertical composition. 
	
	The remainder of this section describes right exact cells and initial cells.
	\begin{proposition} \label{initial cells}
		In a double category consider a morphism $\map gBD$. If the companion $\hmap{g_*}BD$ of $g$ exists then any opcartesian cell of the form below is both initial and pointwise initial.
		\begin{displaymath}
			\begin{tikzpicture}
				\matrix(m)[math175em]{A & B \\ A & D \\};
				\path[map]	(m-1-1) edge[barred] node[above] {$J$} (m-1-2)
										(m-1-2) edge node[right] {$g$} (m-2-2)
										(m-2-1) edge[barred] node[below] {$K$} (m-2-2);
				\path				(m-1-1) edge[eq] (m-2-1);
				\draw				($(m-1-1)!0.5!(m-2-2)$)	node {\oc};
			\end{tikzpicture}
		\end{displaymath}
	\end{proposition}
	\begin{proof}
		First notice that by \lemref{cartesian and opcartesian cells in terms of companions and conjoints} we may, without loss of generality, assume that the opcartesian cell above is given by the horizontal composite of $\id_J$ and the opcartesian cell defining $g_*$; in particular we thus have $K = J \hc g_*$.
		
		For any cell $\eps$ as on the left below consider the following identity, which follows from one of the companion identities for $g_*$. We have to show that $\eps$ defines $r$ as a (pointwise) right Kan extension if and only if the composite of the first two columns on the right-hand side does so.
		\begin{displaymath}
			\begin{tikzpicture}[textbaseline]
				\matrix(m)[math175em]{A & B & D \\ M & \phantom M & M \\};
				\path[map]	(m-1-1) edge[barred] node[above] {$J$} (m-1-2)
														edge node[left] {$r$} (m-2-1)
										(m-1-2) edge[barred] node[above] {$g_*$} (m-1-3)
										(m-1-3) edge node[right] {$d$} (m-2-3);
				\path				(m-2-1) edge[eq] (m-2-3)
										(m-1-2) edge[cell] node[right] {$\eps$} (m-2-2);
			\end{tikzpicture} = \begin{tikzpicture}[textbaseline]
				\matrix(m)[math175em]{A & B & B & D \\ A & B & D & D \\ M & \phantom M & M & M \\};
				\path[map]	(m-1-1) edge[barred] node[above] {$J$} (m-1-2)
										(m-1-3) edge[barred] node[above] {$g_*$} (m-1-4)
														edge node[right, inner sep=1pt] {$g$} (m-2-3)
										(m-2-1) edge[barred] node[above] {$J$} (m-2-2)
														edge node[left] {$r$} (m-3-1)
										(m-2-2) edge[barred] node[below] {$g_*$} (m-2-3)
										(m-2-3) edge node[right] {$d$} (m-3-3)
										(m-2-4) edge node[right] {$d$} (m-3-4);
				\path				(m-1-1) edge[eq] (m-2-1)
										(m-1-2) edge[eq] (m-1-3)
														edge[eq] (m-2-2)
										(m-1-4) edge[eq] (m-2-4)
										(m-2-3) edge[eq] (m-2-4)
										(m-3-1) edge[eq] (m-3-3)
										(m-3-3) edge[eq] (m-3-4)
										(m-2-2) edge[cell] node[right] {$\eps$} (m-3-2);
				\draw				($(m-1-2)!0.5!(m-2-3)$) node {\oc}
										($(m-1-3)!0.5!(m-2-4)$) node {\cc};
			\end{tikzpicture}
		\end{displaymath}
		The last column of the right-hand side above defines $d \of g$ as a pointwise right Kan extension by \lemref{precomposition}, so that the proof follows from the pasting lemma.
	\end{proof}
	
	The following example shows that the notion of initial cell extends the classical notion of `initial functor'.
	\begin{example}
		Recall that a functor $\map gBD$ is called \emph{initial} if, for each $x \in D$, the comma category $g \slash x$ is connected; see for example Section 2.5 of \cite{Kashiwara-Schapira06}, where such functors are called `co-cofinal'. Writing $\map !D1$ for the unique functor into the terminal category $1$, we claim that $\map gBD$ is initial precisely if the obvious cartesian cell
		\begin{displaymath}
			\begin{tikzpicture}
					\matrix(m)[math175em]{1 & B \\ 1 & D \\};
					\path[map]	(m-1-1) edge[barred] node[above] {$!^*$} (m-1-2)
											(m-1-2) edge node[right] {$g$} (m-2-2)
											(m-2-1) edge[barred] node[below] {$!^*$} (m-2-2);
					\path				(m-1-1) edge[eq] (m-2-1);
					\draw				($(m-1-1)!0.5!(m-2-2)$) node {\cc};
			\end{tikzpicture}
		\end{displaymath}
		is opcartesian in $\Prof$. To see this, notice that the latter is equivalent to the horizontal cell $\cell{\textup{cart}_*}{!^* \hc g_*}{!^*}$, that is the factorisation of the cell above defined by \eqref{companion and conjoint factorisations}, being invertible. Remembering the coequalisers \eqref{composition of unenriched profunctors} that define horizontal composition in $\Prof$, the latter is readily seen to be equivalent to the comma categories $g \slash x$ being connected for each $x$.
			
		Initial functors are useful because diagrams $\map dDM$ can be restricted along $g$ without changing their limits: formally, if either $\lim d$ or $\lim (d \of g)$ exists then so does the other, and in that case the canonical morphism
		\begin{displaymath}
			\lim d \to \lim (d \of g)
		\end{displaymath}
		is an isomorphism. This result can be recovered by applying the previous proposition to the cartesian cell above, by using \propref{right Kan extensions along conjoints as right Kan extensions in V(K)} and the well-known fact that the limits of $d$ and $d \of g$ coincide with the right Kan extensions of $d$ and $d \of g$ along terminal functors (see Theorem X.7.1 of \cite{MacLane98} for the dual result on colimits).
	\end{example}
	
	\begin{proposition} \label{pointwise exact cartesian cells}
		In a double category consider a morphism $\map fAC$. If the companion $\hmap{f_*}AC$ of $f$ exists then any cartesian cell of the form below is pointwise right exact.
		\begin{displaymath}
			\begin{tikzpicture}
				\matrix(m)[math175em]{A & B \\ C & B \\};
					\path[map]	(m-1-1) edge[barred] node[above] {$J$} (m-1-2)
															edge node[left] {$f$} (m-2-1)
											(m-2-1) edge[barred] node[below] {$K$} (m-2-2);
					\path				(m-1-2) edge[eq] (m-2-2);
					\draw				($(m-1-1)!0.5!(m-2-2)$) node {\cc};
			\end{tikzpicture}
		\end{displaymath}
	\end{proposition}
	\begin{proof}
		By \lemref{cartesian and opcartesian cells in terms of companions and conjoints} we may, without loss of generality, assume that the cartesian cell above is given by the horizontal composite of the cartesian cell defining $f_*$ and the identity $\id_K$; in particular we thus have $J = f_* \hc K$.
		
		Consider a cell $\eps$ as in the composite below and assume that it defines $r$ as the pointwise right Kan extension of $d$ along $K$; we have to prove that the full composite defines $r \of f$ as a pointwise left Kan extension.
		\begin{displaymath}
			\begin{tikzpicture}
				\matrix(m)[math175em]{A & C & B \\ C & C & B \\ M & M & M \\};
				\path[map]	(m-1-1) edge[barred] node[above] {$f_*$} (m-1-2)
														edge node[left] {$f$} (m-2-1)
										(m-1-2) edge[barred] node[above] {$K$} (m-1-3)
										(m-2-1) edge node[left] {$r$} (m-3-1)
										(m-2-2) edge[barred] node[above] {$K$} (m-2-3)
														edge node[left] {$r$} (m-3-2)
										(m-2-3) edge node[right] {$d$} (m-3-3);
				\path				(m-1-2) edge[eq] (m-2-2)
										(m-1-3) edge[eq] (m-2-3)
										(m-2-1) edge[eq] (m-2-2)
										(m-3-1) edge[eq] (m-3-2)
										(m-3-2) edge[eq] (m-3-3);
				\path[transform canvas={shift={($(m-2-2)!0.5!(m-2-3)$)}}]	(m-2-2) edge[cell] node[right] {$\eps$} (m-3-2);
				\draw				($(m-1-1)!0.5!(m-2-2)$) node {\cc};
			\end{tikzpicture}
		\end{displaymath}
		That it does follows from the fact that the first column defines $r \of f$ as a pointwise right Kan extension by \lemref{precomposition}, together with the pasting lemma.
	\end{proof}
	
	The following corollary combines Propositions \ref{initial cells} and \ref{pointwise exact cartesian cells}.
	\begin{corollary} \label{pointwise exact cells}
		In a double category consider a cell $\phi$ below, and assume that the companions $\hmap{f_*}AC$ and $\hmap{g_*}BD$ of $f$ and $g$ exist.
		\begin{displaymath}
			\begin{tikzpicture}
				\matrix(m)[math175em]{A & B \\ C & D \\};
				\path[map]	(m-1-1) edge[barred] node[above] {$J$} (m-1-2)
														edge node[left] {$f$} (m-2-1)
										(m-1-2) edge node[right] {$g$} (m-2-2)
										(m-2-1) edge[barred] node[below] {$K$} (m-2-2);
				\path[transform canvas={shift={($(m-1-2)!(0,0)!(m-2-2)$)}}] (m-1-1) edge[cell] node[right] {$\phi$} (m-2-1);
			\end{tikzpicture}
		\end{displaymath}
		The cell $\phi$ is pointwise right exact as soon as the horizontal cell $\cell{\phi_*}{J \hc g_*}{f_* \hc K}$, that is defined by \eqref{companion and conjoint factorisations}, is invertible.
	\end{corollary}
	\begin{proof}
		If $\phi_*$ is invertible then the left identity of \eqref{companion and conjoint factorisations} factors $\phi$ into an opcartesian cell, which defines the extension of $J$ along $g$, followed by a cartesian cell which defines the restriction of $K$ along $f$. The first is pointwise right exact by \propref{initial cells} and the second by \propref{pointwise exact cartesian cells}, so that their vertical composite $\phi$ is pointwise right exact too.
	\end{proof}
	\begin{example} \label{example: carres exacts}
		In the double category $\Prof$ of unenriched profunctors consider a transformation $\phi$ of functors as on the left below, as well as its factorisation $\phi'$, as shown.
		\begin{displaymath}
			\begin{tikzpicture}[textbaseline]
				\matrix(m)[math175em]{B & B \\ A & D \\ C & C \\};
				\path[map]	(m-1-1) edge node[left] {$j$} (m-2-1)
										(m-1-2) edge node[right] {$g$} (m-2-2)
										(m-2-1) edge node[left] {$f$} (m-3-1)
										(m-2-2) edge node[right] {$k$} (m-3-2);
				\path				(m-1-1) edge[eq] (m-1-2)
										(m-3-1) edge[eq] (m-3-2);
				\path[transform canvas={shift={($(m-2-2)!0.5!(m-3-2)$)}}]	(m-1-1) edge[cell] node[right] {$\phi$} (m-2-1);
			\end{tikzpicture} = \begin{tikzpicture}[textbaseline]
				\matrix(m)[math175em]{B & B \\ A & B \\ C & D \\ C & C \\};
				\path[map]	(m-1-1) edge node[left] {$j$} (m-2-1)
										(m-2-1) edge[barred] node[below] {$j^*$} (m-2-2)
														edge node[left] {$f$} (m-3-1)
										(m-2-2) edge node[right] {$g$} (m-3-2)
										(m-3-1) edge[barred] node[below] {$k^*$} (m-3-2)
										(m-3-2) edge node[right] {$k$} (m-4-2);
				\path				(m-1-1) edge[eq] (m-1-2)
										(m-1-2) edge[eq] (m-2-2)
										(m-3-1) edge[eq] (m-4-1)
										(m-4-1) edge[eq] (m-4-2);
				\path[transform canvas={shift={($(m-2-2)!0.5!(m-3-2)$)}, yshift=-3pt}]	(m-2-1) edge[cell] node[right] {$\phi'$} (m-3-1);
				\draw				($(m-1-1)!0.5!(m-2-2)$) node {\oc}
										($(m-3-1)!0.5!(m-4-2)$) node {\cc};
			\end{tikzpicture}
		\end{displaymath}
		Using $f_* \hc k^* \iso C(f, k)$, the hypothesis of the previous corollary for the factorisation $\phi'$ means that the maps $\map{\phi'_*}{\int^{b \in B} A(a, jb) \times D(gb, d)}{C(fa, kd)}$, that are induced by the assignments
		\begin{displaymath}
			(a \xrar u jb, gb \xrar v d) \mapsto (fa \xrar{fu} fjb \xrar{\phi_b} kgb \xrar{kv} kd),
		\end{displaymath}
		are bijections. In Section 1 of \cite{Guitart80} a natural transformation $\phi$ as on the left above, such that $\phi'_*$ is invertible, is called a `carr\'e exact'; this explains our terminology of `right exact cell'. The condition of $\phi'_*$ being an isomorphism is known as the \emph{right Beck-Chevalley condition}; in $\Prof$ this condition is in fact equivalent to $\phi'$ being a right exact cell in the sense of \defref{definition: exact cell}, as follows from Theorem 1.10 of \cite{Guitart80}.
	\end{example}
	
	\section{Pointwise Kan extensions in terms of Kan extensions}
	In this section we show that if an equipment $\K$ has all `opcartesian tabulations' then pointwise Kan extensions in $\K$ (\defref{definition: pointwise right Kan extension}) can be defined in terms of Kan extensions (\defref{definition: right Kan extension}). This is analogous to Street's definition of pointwise Kan extension in $2$-categories, that was introduced in \cite{Street74}, which we recall first.
	
	\subsection{Pointwise Kan extensions in \texorpdfstring{$2$}{2}-categories.}
	Street's definition of pointwise Kan extension in $2$-categories uses the notion of `comma object' in a $2$-category, which is recalled below. It generalises the well-known notion of a `comma category' $f \slash g$ of two functors $\map fAC$ and $\map gBC$, in the $2$-category $\Cat$ of categories, functors and transformations; see Section II.6 of \cite{MacLane98}.
	\begin{definition} \label{definition: comma object}
		Given two morphisms $\map fAC$ and $\map gBC$ in a $2$-category $\C$, the \emph{comma object} $f \slash g$ of $f$ and $g$ consists of an object $f \slash g$ equipped with projections $\pi_A$ and $\pi_B$, as in the diagram on the left below, as well as a cell $\cell\pi{f \of \pi_A}{g \of \pi_B}$, that satisfies the following $1$-dimensional and $2$-dimensional universal properties.
		\begin{displaymath}
			\begin{tikzpicture}[textbaseline]
				\matrix(m)[math2em]{f \slash g &[-3pt] A \\ B & C \\};
				\path[map]	(m-1-1) edge node[above] {$\pi_A$} (m-1-2)
														edge node[left] {$\pi_B$} (m-2-1)
										(m-1-2) edge node[right] {$f$} (m-2-2)
										(m-2-1) edge node[below] {$g$} (m-2-2);
				\path				(m-1-2) edge[cell, shorten >= 7pt, shorten <= 7pt] node[below right] {$\pi$} (m-2-1);
			\end{tikzpicture} \qquad \begin{tikzpicture}[textbaseline]
				\matrix(m)[math2em]{X & A \\ B & C \\};
				\path[map]	(m-1-1) edge node[above] {$\phi_A$} (m-1-2)
														edge node[left] {$\phi_B$} (m-2-1)
										(m-1-2) edge node[right] {$f$} (m-2-2)
										(m-2-1) edge node[below] {$g$} (m-2-2);
				\path				(m-1-2) edge[cell, shorten >= 7pt, shorten <= 7pt] node[below right] {$\phi$} (m-2-1);
			\end{tikzpicture} \qquad \begin{tikzpicture}[textbaseline]
				\matrix(m)[math2em]{X & A \\ B & C \\};
				\path[map]	(m-1-1) edge[bend left=35] node[above] {$\phi_A$} (m-1-2)
														edge[bend right=35] node[below left, inner sep=2pt] {$\psi_A$} (m-1-2)
														edge node[left] {$\psi_B$} (m-2-1)
										(m-1-2) edge node[right] {$f$} (m-2-2)
										(m-2-1) edge node[below] {$g$} (m-2-2)
										($(m-1-1)!0.5!(m-1-2)+(-0.3em,0.75em)$) edge[cell] node[right] {$\xi_A$} ($(m-1-1)!0.5!(m-1-2)-(0.3em,0.75em)$)
										(m-1-2) edge[cell, shorten >= 8pt, shorten <= 8pt, transform canvas={xshift=4pt, yshift=-4pt}] node[below right, inner sep=1pt] {$\psi$} (m-2-1);
			\end{tikzpicture} = \begin{tikzpicture}[textbaseline]
				\matrix(m)[math2em]{X & A \\ B & C \\};
				\path[map]	(m-1-1) edge node[above] {$\phi_A$} (m-1-2)
														edge[bend left=35] node[above right, inner sep=2pt] {$\phi_B$} (m-2-1)
														edge[bend right=35] node[left] {$\psi_B$} (m-2-1)
										(m-1-2) edge node[right] {$f$} (m-2-2)
										(m-2-1) edge node[below] {$g$} (m-2-2)
										($(m-1-1)!0.5!(m-2-1)+(0.75em,-0.3em)$) edge[cell] node[above] {$\xi_B$} ($(m-1-1)!0.5!(m-2-1)-(0.75em,0.3em)$)
										(m-1-2) edge[cell, shorten >= 8pt, shorten <= 8pt, transform canvas={xshift=4pt, yshift=-4pt}] node[below right, inner sep=1pt] {$\phi$} (m-2-1);
			\end{tikzpicture}
		\end{displaymath}
		Given any other cell $\phi$ in $\C$ as in the middle above, the $1$-dimensional property states that there exists a unique morphism $\map{\phi'}X{f \slash g}$ such that $\pi_A \of \phi' = \phi_A$, \mbox{$\pi_B \of \phi' = \phi_B$} and $\pi \of \phi' = \phi$.
	
		The $2$-dimensional property is the following. Suppose we are given another cell $\psi$ as in the identity on the right above, which factors through $\pi$ as $\map{\psi'}X{f \slash g}$, like $\phi$ factors as $\phi'$. Then for any pair of cells $\cell{\xi_A}{\phi_A}{\psi_A}$ and $\cell{\xi_B}{\phi_B}{\psi_B}$, such that the identity above holds, there exists a unique cell $\cell{\xi'}{\phi'}{\psi'}$ such that $\pi_A \of \xi' = \xi_A$ and $\pi_B \of \xi' = \xi_B$.
	\end{definition}
	We can now give Street's definition of pointwise left Kan extension, that was introduced in Section 4 of \cite{Street74}.
	\begin{definition}[Street] \label{definition: pointwise left Kan extensions in 2-categories}
		Let $\map dAM$ and $\map jAB$ be morphisms in a $2$\ndash category that has all comma objects. The cell $\cell\eps {r \of j}d$ in the diagram below is said to define $r$ as the \emph{pointwise right Kan extension of $d$ along $j$} if, for each $\map fCB$, the composition of cells below defines $r \of f$ as the right Kan extension of $d \of \pi_A$ along $\pi_C$.
		\begin{displaymath}
			\begin{tikzpicture}
				\matrix(m)[math2em, column sep=0.25em, row sep=2.5em]{f \slash j & & C \\ A & & B \\[-0.6em] \phantom M & M & \phantom M \\};
				\path[map]	(m-1-1) edge node[above] {$\pi_C$} (m-1-3)
														edge node[left] {$\pi_A$} (m-2-1)
										(m-1-3) edge node[right] {$f$} (m-2-3)
										(m-2-1) edge node[above] {$j$} (m-2-3)
														edge node[below left] {$d$} (m-3-2)
										(m-2-3) edge node[below right] {$r$} (m-3-2);
				\path				(m-1-3) edge[cell, shorten >=12pt, shorten <=12pt] node[above left] {$\pi$} (m-2-1)
										(m-2-3) edge[cell, shorten >=8pt, shorten <=8pt] node[below right, xshift=-2pt] {$\eps$} ($(m-2-1)!0.5!(m-3-2)$);
			\end{tikzpicture}
		\end{displaymath}
	\end{definition}
	
	\begin{remark} We remark that Street's notion of pointwise right Kan extension is in general stronger than Dubuc's notion of pointwise right Kan extension along enriched functors, that was recalled in \defref{definition: pointwise Kan extension along enriched functors}. For instance, an example of an enriched left Kan extension of $2$\ndash functors that is not pointwise in the $2$-category of $2$-categories, $2$-functors and $2$-transformations, in the above sense, is given in Example 2.24 of \cite{Koudenburg13}. Thus pointwise right Kan extensions in $\enProf\V$, as defined in \defref{definition: pointwise right Kan extension in right closed equipments}, will in general not be pointwise in the $2$-category $\enCat\V = V(\enProf\V)$, in the above sense. In the next subsection we will consider a condition on equipments $\K$ ensuring that pointwise right Kan extensions in $\K$, as in \defref{definition: pointwise right Kan extension}, do correspond to pointwise right Kan extensions in $V(\K)$, as in the definition above; see \corref{pointwise Kan extensions as pointwise vertical Kan extensions}.
	\end{remark}
	
	\subsection{Pointwise Kan extensions in terms of Kan extensions.}
	Having recalled how Street's definition of pointwise Kan extension in a $2$-category is given in terms of ordinary Kan extensions, we now wonder whether pointwise Kan extensions in equipments $\K$ can be defined in terms of ordinary ones too. It turns out that we can as soon as $\K$ has `opcartesian tabulations'. We start by recalling the definition of tabulation from Section~5 of \cite{Grandis-Pare99}, where it is called `tabulator'. Tabulations generalise comma objects, as we will see in \propref{comma objects in terms of tabulations}.
	\begin{definition} \label{definition: tabulation}
		Given a horizontal morphism $\hmap JAB$ in a double category $\K$, the \emph{tabulation} $\tab J$ of $J$ consists of an object $\tab J$ equipped with a cell $\pi$ as on the left below, satisfying the following $1$-dimensional and $2$-dimensional universal properties.
		\begin{displaymath}
			\begin{tikzpicture}[baseline]
  			\matrix(m)[tab]{\tab J & \tab J \\ A & B \\};
  			\path[map]  (m-1-1) edge node[left] {$\pi_A$} (m-2-1)
										(m-1-2) edge node[right] {$\pi_B$} (m-2-2)
										(m-2-1) edge[barred] node[below] {$J$} (m-2-2);
				\path				(m-1-1) edge[eq] (m-1-2);
				\path[transform canvas={shift={($(m-1-2)!(0,0)!(m-2-2)$)}}] (m-1-1) edge[cell] node[right] {$\pi$} (m-2-1);
			\end{tikzpicture} \qquad\qquad\qquad\qquad\qquad \begin{tikzpicture}[baseline]
  			\matrix(m)[math175em]{X & X \\ A & B \\};
  			\path[map]  (m-1-1) edge node[left] {$\phi_A$} (m-2-1)
										(m-1-2) edge node[right] {$\phi_B$} (m-2-2)
										(m-2-1) edge[barred] node[below] {$J$} (m-2-2);
				\path				(m-1-1) edge[eq] (m-1-2);
				\path[transform canvas={shift={($(m-1-2)!(0,0)!(m-2-2)$)}}] (m-1-1) edge[cell] node[right] {$\phi$} (m-2-1);
			\end{tikzpicture}
		\end{displaymath}
		Given any other cell $\phi$ in $\K$ as on the right above, the $1$-dimensional property states that there exists a unique vertical morphism $\map{\phi'}X{\tab J}$ such that $\pi \of 1_{\phi'} = \phi$.
		
		The $2$-dimensional property is the following. Suppose we are given another cell $\psi$ as in the identity below, which factors through $\pi$ as $\map{\psi'}Y{\tab J}$, like $\phi$ factors as $\phi'$. Then for any pair of cells $\xi_A$ and $\xi_B$ as below, so that the identity holds, there exists a unique cell $\xi'$ as on the right below such that $1_{\pi_A} \of \xi' = \xi_A$ and $1_{\pi_B} \of \xi' = \xi_B$.
		\begin{displaymath}
			 \begin{tikzpicture}[textbaseline]
				\matrix(m)[math175em]{X & Y & Y \\ A & A & B \\};
				\path[map]	(m-1-1) edge[barred] node[above] {$H$} (m-1-2)
														edge node[left] {$\phi_A$} (m-2-1)
										(m-1-2) edge node[right, inner sep=1pt] {$\psi_A$} (m-2-2)
										(m-1-3) edge node[right] {$\psi_B$} (m-2-3)
										(m-2-2) edge[barred] node[below] {$J$} (m-2-3);
				\path				(m-1-2) edge[eq] (m-1-3)
										(m-2-1) edge[eq] (m-2-2);
				\path[transform canvas={shift={($(m-1-2)!0.5!(m-2-3)$)}}] (m-1-1) edge[cell] node[right] {$\xi_A$} (m-2-1)
										(m-1-2) edge[cell] node[right] {$\psi$} (m-2-2);
			\end{tikzpicture} = \begin{tikzpicture}[textbaseline]
				\matrix(m)[math175em]{X & X & Y \\ A & B & B \\};
				\path[map]	(m-1-1) edge node[left] {$\phi_A$} (m-2-1)
										(m-1-2) edge[barred] node[above] {$H$} (m-1-3)
														edge node[right, inner sep=1pt] {$\phi_B$} (m-2-2)
										(m-1-3) edge node[right] {$\psi_B$} (m-2-3)
										(m-2-1) edge[barred] node[below] {$J$} (m-2-2);
				\path				(m-1-1) edge[eq] (m-1-2)
										(m-2-2) edge[eq] (m-2-3);
				\path[transform canvas={shift={($(m-1-2)!0.5!(m-2-3)$)}}] (m-1-1) edge[cell] node[right] {$\phi$} (m-2-1)
										(m-1-2) edge[cell] node[right] {$\xi_B$} (m-2-2);
			\end{tikzpicture} \qquad \qquad \begin{tikzpicture}[textbaseline]
  			\matrix(m)[math175em]{X & Y \\ \tab J & \tab J \\};
  			\path[map]  (m-1-1) edge[barred] node[above] {$H$} (m-1-2)
														edge node[left] {$\phi'$} (m-2-1)
										(m-1-2) edge node[right] {$\psi'$} (m-2-2);
				\path				(m-2-1) edge[eq] (m-2-2);
				\path[transform canvas={shift={($(m-1-2)!(0,0)!(m-2-2)$)}}] (m-1-1) edge[cell] node[right] {$\xi'$} (m-2-1);
			\end{tikzpicture}
		\end{displaymath}
		A tabulation is called \emph{opcartesian} whenever its defining cell $\pi$ is opcartesian.
	\end{definition}
	
	As with cartesian and opcartesian cells, we often will not name tabulations but simply depict them as
	\begin{displaymath}
		\begin{tikzpicture}[baseline]
  		\matrix(m)[tab]{\tab J & \tab J \\ A & B. \\};
  		\path[map]  (m-1-1) edge node[left] {$\pi_A$} (m-2-1)
									(m-1-2) edge node[right] {$\pi_B$} (m-2-2)
									(m-2-1) edge[barred] node[below] {$J$} (m-2-2);
			\path				(m-1-1) edge[eq] (m-1-2);
			\draw				($(m-1-1)!0.5!(m-2-2)$) node {\tc};
		\end{tikzpicture}
	\end{displaymath}

	\begin{example} \label{example: tabulations of profunctors}
		The tabulation $\tab J$ of a profunctor $\map J{\op A \times B}\Set$ is the category that has triples $(a, j, b)$ as objects, where $(a, b) \in A \times B$ and $\map jab$ in $J$, while a map $(a, j, b) \to (a', j', b')$ is a pair $\map{(u, v)}{(a,b)}{(a', b')}$ in $A \times B$ making the diagram
	\begin{equation} \label{morphism in a double comma category}
		\begin{tikzpicture}[textbaseline]
			\matrix(m)[math175em]{a & b \\ a' & b' \\};
			\path[map]	(m-1-1) edge node[above] {$j$} (m-1-2)
													edge node[left] {$u$} (m-2-1)
									(m-1-2) edge node[right] {$v$} (m-2-2)
									(m-2-1) edge node[below] {$j'$} (m-2-2);
		\end{tikzpicture}
	\end{equation}
	commute in $J$. The functors $\pi_A$ and $\pi_B$ are projections and the natural transformation $\nat \pi{1_{\tab J}}J$ maps the pair $(u, v)$ to the diagonal $j' \of u = v \of j$. It is easy to check that $\pi$ satisfies both the $1$-dimensional and $2$-dimensional universal property above, and that it is opcartesian.
	\end{example}
	
	\begin{example}
		Generalising the previous example, we will show in \propref{opcartesian tabulations for internal profunctors} that all internal profunctors of the double category $\inProf\E$ admit opcartesian tabulations.
	\end{example}
	
	Comma objects and tabulations are related as follows.
	\begin{proposition} \label{comma objects in terms of tabulations}
		Let $\map fAC$ and $\map gBC$ be vertical morphisms in a double category $\K$. If both the cartesian cell and the tabulation below exist then their composite defines the object $\tab{C(f, g)}$ as the comma object of $f$ and $g$ in $V(\K)$.
		\begin{displaymath}
			\begin{tikzpicture}
				\matrix(m)[tab, column sep=0.75em]{\tab{C(f, g)} & \tab{C(f,g)} \\ A & B \\ C & C \\};
				\path[map]	(m-1-1) edge node[left] {$\pi_A$} (m-2-1)
										(m-1-2) edge node[right] {$\pi_B$} (m-2-2)
										(m-2-1) edge[barred] node[below] {$C(f, g)$} (m-2-2)
														edge node[left] {$f$} (m-3-1)
										(m-2-2) edge node[right] {$g$} (m-3-2);
				\path				(m-1-1) edge[eq] (m-1-2)
										(m-3-1) edge[eq] (m-3-2);
				\draw				($(m-1-1)!0.5!(m-2-2)$)	node {\tc}
										($(m-2-1)!0.5!(m-3-2)$) node {\cc};
			\end{tikzpicture}
		\end{displaymath}
	\end{proposition}
	\begin{proof}
		This follows immediately from the fact that the unique factorisations defining the comma object of $f$ and $g$ in $V(\K)$, as in \defref{definition: comma object}, correspond precisely to unique factorisations through the tabulation above in $\K$, as in \defref{definition: tabulation}, by factorising them through the cartesian cell.
	\end{proof}
	
	If an equipment has opcartesian tabulations then any (pointwise) right Kan extension can be regarded as a (pointwise) right Kan extension along a conjoint, as follows.
	\begin{proposition} \label{right Kan extensions along conjoints}
		In a double category consider a tabulation as on the left below and assume that the conjoint $\hmap{\pi_A^*}A{\gen J}$ of $\pi_A$ and the companion $\hmap{\pi_{B*}}{\gen J}B$ of $\pi_B$ exist. Consider the factorisation of the tabulation through the opcartesian cell defining $\pi_A^*$, as shown.
		\begin{displaymath}
			\begin{tikzpicture}[textbaseline]
  			\matrix(m)[tab]{\tab J & \tab J \\ A & B \\};
  			\path[map]  (m-1-1) edge node[left] {$\pi_A$} (m-2-1)
										(m-1-2) edge node[right] {$\pi_B$} (m-2-2)
										(m-2-1) edge[barred] node[below] {$J$} (m-2-2);
				\path				(m-1-1) edge[eq] (m-1-2);
				\draw				($(m-1-1)!0.5!(m-2-2)$) node {\tc};
			\end{tikzpicture} = \begin{tikzpicture}[textbaseline]
				\matrix(m)[tab]{\tab J & \tab J \\ A & \tab J \\ A & B \\};
				\path[map]	(m-1-1) edge node[left] {$\pi_A$} (m-2-1)
										(m-2-1) edge[barred] node[below] {$\pi_A^*$} (m-2-2)
										(m-2-2) edge node[right] {$\pi_B$} (m-3-2)
										(m-3-1) edge[barred] node[below] {$J$} (m-3-2);
				\path				(m-1-1) edge[eq] (m-1-2)
										(m-1-2) edge[eq] (m-2-2)
										(m-2-1) edge[eq] (m-3-1);
				\draw				($(m-1-1)!0.5!(m-2-2)$)	node {\oc}
										($(m-2-1)!0.5!(m-3-2)$) node[, yshift=-2pt, xshift=1.5pt] {$\strut\textup{tab}'$};
			\end{tikzpicture} \qquad\qquad\qquad \begin{tikzpicture}[textbaseline]
				\matrix(m)[math175em, column sep=1.5em]{A & \tab J \\ A & B \\ M & M \\};
				\path[map]	(m-1-1) edge[barred] node[above] {$\pi_A^*$} (m-1-2)
										(m-1-2) edge node[right] {$\pi_B$} (m-2-2)
										(m-2-1) edge[barred] node[above] {$J$} (m-2-2)
														edge node[left] {$r$} (m-3-1)
										(m-2-2) edge node[right] {$d$} (m-3-2);
				\path				(m-1-1) edge[eq] (m-2-1)
										(m-3-1) edge[eq] (m-3-2);
				\path[transform canvas={shift=(m-2-1), xshift=0.65pt}]	(m-2-2) edge[cell] node[right] {$\eps$} (m-3-2);
				\draw				($(m-1-1)!0.5!(m-2-2)$)	node[, yshift=1pt, xshift=1.5pt] {$\strut\textup{tab}'$};
			\end{tikzpicture}
		\end{displaymath}
		If the tabulation is opcartesian then a cell $\eps$, as in the composite on the right, defines $r$ as the (pointwise) right Kan extension of $d$ along $J$ if and only if the composite $\eps \of \textup{tab}'$ defines $r$ as the (pointwise) right Kan extension of $d \of \pi_B$ along $\pi_A^*$.
	\end{proposition}
	\begin{proof}
		Since $\textup{tab}$ and $\textup{opcart}$ are both opcartesian, the factorisation $\textup{tab}'$ is opcartesian by the pasting lemma, and the result follows from \propref{initial cells}.
	\end{proof}
	
	In an equipment with opcartesian tabulations, pointwise Kan extensions can be defined in terms of Kan extensions as follows.
	\begin{theorem} \label{pointwise Kan extensions in terms of Kan extensions}
		In an equipment $\K$ consider the cell $\eps$ on the left below.
		\begin{displaymath}
			\begin{tikzpicture}[baseline]
  			\matrix(m)[math175em]{A & B \\ M & M \\};
  			\path[map]  (m-1-1) edge[barred] node[above] {$J$} (m-1-2)
														edge node[left] {$r$} (m-2-1)
										(m-1-2) edge node[right] {$d$} (m-2-2);
				\path				(m-2-1) edge[eq] (m-2-2);
				\path[transform canvas={shift={($(m-1-2)!(0,0)!(m-2-2)$)}}] (m-1-1) edge[cell] node[right] {$\eps$} (m-2-1);
			\end{tikzpicture} \qquad\qquad\qquad\qquad\qquad \begin{tikzpicture}[baseline]
				\matrix(m)[math175em]{C & B \\ A & B \\ M & M \\};
				\path[map]	(m-1-1) edge[barred] node[above] {$J(f, \id)$} (m-1-2)
														edge node[left] {$f$} (m-2-1)
										(m-2-1) edge[barred] node[below] {$J$} (m-2-2)
														edge node[left] {$r$} (m-3-1)
										(m-2-2) edge node[right] {$d$} (m-3-2);
				\path				(m-1-2) edge[eq] (m-2-2)
										(m-3-1) edge[eq] (m-3-2);
				\path[transform canvas={shift=(m-2-1), yshift=-2pt}]	(m-2-2) edge[cell] node[right] {$\eps$} (m-3-2);
				\draw				($(m-1-1)!0.5!(m-2-2)$)	node {\cc};
			\end{tikzpicture} 
		\end{displaymath}
		For the following conditions the implications (a) $\Leftrightarrow$ (b) $\Rightarrow$ (c) hold, while (c) $\Rightarrow$ (a) holds as soon as $\K$ has opcartesian tabulations.
		\begin{enumerate}[label=(\alph*)]
			\item The cell $\eps$ defines $r$ as the pointwise right Kan extension of $d$ along $J$;
			\item for all $\map fCA$ the composite on the right above defines $r \of f$ as the pointwise right Kan extension of $d$ along $J(f, \id)$;
			\item for all $\map fCA$ the composite on the right above defines $r \of f$ as the right Kan extension of $d$ along $J(f, \id)$.
		\end{enumerate}
	\end{theorem}
	\begin{proof}
		(a) $\Rightarrow$ (b) follows from \propref{pointwise exact cartesian cells}.
		
		(b) $\Rightarrow$ (a). By taking $f = \id_A$ in (b) the cartesian cell defining $J(f, \id)$ is invertible by \lemref{invertible cartesian and opcartesian cells}, so in that case $\eps \of \textup{cart}$ defines $r$ as the pointwise right Kan extension precisely if $\eps$ does, and the implication follows.
		
		(b) $\Rightarrow$ (c) follows from \lemref{pointwise right Kan extensions are right Kan extensions}.
		
		(c) $\Rightarrow$ (a). To prove (a) we have to show that every cell $\phi$ as in the composite on the left below factors uniquely through $\eps$.
		\begin{displaymath}
			\begin{tikzpicture}[textbaseline]
  			\matrix(m)[math175em]{\tab H & \tab H &[0.4em] B \\ C & A & B \\ M & \phantom A & M \\};
  			\path[map]	(m-1-1) edge node[left] {$\pi_C$} (m-2-1)
  									(m-1-2) edge[barred] node[above] {$J(\pi_A, \id)$} (m-1-3)
  													edge node[right, inner sep=2pt] {$\pi_A$} (m-2-2)
  									(m-2-1) edge[barred] node[below] {$H$} (m-2-2)
  													edge node[left] {$s$} (m-3-1)
  									(m-2-2) edge[barred] node[below] {$J$} (m-2-3)
  									(m-2-3) edge node[right] {$d$} (m-3-3)
  									(m-2-2) edge[cell] node[right] {$\phi$} (m-3-2);
 		 		\path				(m-1-1) edge[eq] (m-1-2)
  									(m-1-3) edge[eq] (m-2-3)
  									(m-3-1) edge[eq] (m-3-3);
  			\draw				($(m-1-1)!0.5!(m-2-2)$) node {\tc}
  									($(m-1-2)!0.5!(m-2-3)$) node {\cc};
  		\end{tikzpicture} = \mspace{-10mu}\begin{tikzpicture}[textbaseline]
  			\matrix(m)[math175em]{\tab H & \tab H &[0.4em] B \\ C & A & B \\ M & M & M \\};
  			\path[map]	(m-1-1) edge node[left] {$\pi_C$} (m-2-1)
  									(m-1-2) edge[barred] node[above] {$J(\pi_A, \id)$} (m-1-3)
  													edge node[left] {$\pi_A$} (m-2-2)
  									(m-2-1) edge node[left] {$s$} (m-3-1)
  									(m-2-2) edge[barred] node[below] {$J$} (m-2-3)
  													edge node[left] {$r$} (m-3-2)
  									(m-2-3) edge node[right] {$d$} (m-3-3);
 		 		\path				(m-1-1) edge[eq] (m-1-2)
  									(m-1-3) edge[eq] (m-2-3)
  									(m-3-1) edge[eq] (m-3-2)
  									(m-3-2) edge[eq] (m-3-3);
  			\path[transform canvas={shift={($(m-2-3)!0.5!(m-3-2)$)}}]	(m-1-1) edge[cell] node[right] {$\phi'$} (m-2-1);
  			\path[transform canvas={shift={($(m-2-2)!0.5!(m-2-3)$)}, yshift=-2pt}] (m-2-2) edge[cell] node[right] {$\eps$} (m-3-2);
  			\draw				($(m-1-2)!0.5!(m-2-3)$) node {\cc};
  		\end{tikzpicture} = \mspace{-10mu}\begin{tikzpicture}[textbaseline]
  			\matrix(m)[math175em]{\tab H & \tab H &[0.4em] B \\ C & A & B \\ M & M & M \\};
  			\path[map]	(m-1-1) edge node[left] {$\pi_C$} (m-2-1)
  									(m-1-2) edge[barred] node[above] {$J(\pi_A, \id)$} (m-1-3)
  													edge node[right, inner sep=2pt] {$\pi_A$} (m-2-2)
  									(m-2-1) edge[barred] node[below] {$H$} (m-2-2)
  													edge node[left] {$s$} (m-3-1)
  									(m-2-2) edge[barred] node[below] {$J$} (m-2-3)
  													edge node[right] {$r$} (m-3-2)
  									(m-2-3) edge node[right] {$d$} (m-3-3);
 		 		\path				(m-1-1) edge[eq] (m-1-2)
  									(m-1-3) edge[eq] (m-2-3)
  									(m-3-1) edge[eq] (m-3-2)
  									(m-3-2) edge[eq] (m-3-3);
  			\path[transform canvas={shift={($(m-2-2)!0.5!(m-2-3)$)}, yshift=-2pt}]	(m-2-1) edge[cell] node[right] {$\phi''$} (m-3-1)
  									(m-2-2) edge[cell] node[right] {$\eps$} (m-3-2);
  			\draw				($(m-1-1)!0.5!(m-2-2)$) node {\tc}
  									($(m-1-2)!0.5!(m-2-3)$) node {\cc};
  		\end{tikzpicture}
		\end{displaymath}
		We write $\xi$ for the horizontal composite of the tabulation of $H$ and the cartesian cell defining $J(\pi_A, \id)$, that forms the top row in the composite on the left above. Assuming (c), $\phi \of \xi$ factors uniquely through $\eps \of \textup{cart}$ as a vertical cell $\phi'$ as in the first identity above. The cell $\phi'$ in turn factors through the tabulation, which is opcartesian, as a unique cell $\phi''$, as in the second identity above. The composite $\xi$ of the opcartesian tabulation and the cartesian cell is opcartesian by \lemref{composing cartesian and opcartesian cells}, so that the equality above implies $\phi = \phi'' \hc \eps$, since factorisations through opcartesian cells are unique.
		
		To show uniqueness of $\phi''$ consider a second cell $\cell\psi H{1_M}$ such that $\phi = \psi \hc \eps$. Composing with $\xi$ we obtain $\phi \of \xi = (\psi \of \textup{tab}) \hc (\eps \of \textup{cart})$, so that $\psi \of \textup{tab} = \phi'$ by uniqueness of $\phi'$. Hence $\psi = \phi''$ by uniqueness of factorsations through opcartesian cells, concluding the proof.
	\end{proof}
	
	Using the previous theorem we can show that, for any equipment $\K$ that has opcartesian tabulations, pointwise right Kan extensions along conjoints in $\K$ coincide with pointwise right Kan extensions in the $2$-category $V(\K)$, as follows.
	\begin{proposition} \label{pointwise right Kan extensions along conjoints as pointwise right Kan extensions in V(K)}
		In an equipment $\K$ that has opcartesian tabulations consider a vertical cell $\eps$, on the left below, as well as its factorisation $\eps'$ through the opcartesian cell defining the conjoint of $j$, as shown.
		\begin{displaymath}
			\begin{tikzpicture}[textbaseline]
    		\matrix(m)[math175em]{A & A \\ B & \phantom A \\ M & M \\};
    		\path[map]  (m-1-1) edge node[left] {$j$} (m-2-1)
       			        (m-1-2) edge node[right] {$d$} (m-3-2)
       			        (m-2-1) edge node[left] {$r$} (m-3-1);
      	\path				(m-1-1) edge[eq] (m-1-2)
      							(m-3-1) edge[eq] (m-3-2);
    		\path[transform canvas={shift={($(m-2-1)!0.5!(m-1-1)$)}}] (m-2-2) edge[cell] node[right] {$\eps$} (m-3-2);
  		\end{tikzpicture} = \begin{tikzpicture}[textbaseline]
				\matrix(m)[math175em]{A & A \\ B & A \\ M & M \\};
				\path[map]	(m-1-1) edge node[left] {$j$} (m-2-1)
										(m-2-1) edge[barred] node[below] {$j^*$} (m-2-2)
														edge node[left] {$r$} (m-3-1)
										(m-2-2) edge node[right] {$d$} (m-3-2);
				\path				(m-1-1) edge[eq] (m-1-2)
										(m-1-2) edge[eq] (m-2-2)
										(m-3-1) edge[eq] (m-3-2);
				\path[transform canvas={shift=(m-2-1), yshift=-2pt}]	(m-2-2) edge[cell] node[right] {$\eps'$} (m-3-2);
				\draw				($(m-1-1)!0.5!(m-2-2)$)	node {\oc};
			\end{tikzpicture}
		\end{displaymath}
		The vertical cell $\eps$ defines $r$ as the pointwise right Kan extension of $d$ along $j$ in $V(\K)$, in the sense of \defref{definition: pointwise left Kan extensions in 2-categories}, precisely if its factorisation $\eps'$ defines $r$ as the pointwise right Kan extension of $d$ along $j^*$ in $\K$, in the sense of \defref{definition: pointwise right Kan extension}.
	\end{proposition}
	\begin{proof}
		Given any vertical morphism $\map fCB$ we consider the composite on the left side of the equation below, where the first identity follows from the factorisation of $\eps$ above. The second identity is obtained by composing the top row in the middle composite, consisting of a cartesian and opcartesian cell, which results in the cartesian cell in the composite on the right, by \lemref{composing cartesian and opcartesian cells}.
		\begin{displaymath}
			\begin{tikzpicture}[textbaseline]
				\matrix(m)[math175em]{C & A & A \\ B & B & \\ M & M & M \\};
				\path[map]	(m-1-1) edge[barred] node[above] {$B(f, j)$} (m-1-2)
														edge node[left] {$f$} (m-2-1)
										(m-1-2) edge node[right] {$j$} (m-2-2)
										(m-1-3) edge node[right] {$d$} (m-3-3)
										(m-2-1) edge node[left] {$r$} (m-3-1)
										(m-2-2) edge node[left] {$r$} (m-3-2);
				\path				(m-1-2) edge[eq] (m-1-3)
										(m-2-1) edge[eq] (m-2-2)
										(m-3-1) edge[eq] (m-3-2)
										(m-3-2) edge[eq] (m-3-3);
				\path[transform canvas={shift={($(m-2-3)!0.5!(m-3-2)$)}}]	(m-1-2) edge[cell] node[right] {$\eps$} (m-2-2);
				\draw				($(m-1-1)!0.5!(m-2-2)$)	node {\cc};
			\end{tikzpicture} = \begin{tikzpicture}[textbaseline]
				\matrix(m)[math175em]{C & A & A \\ B & B & A \\ M & M & M \\};
				\path[map]	(m-1-1) edge[barred] node[above] {$B(f, j)$} (m-1-2)
														edge node[left] {$f$} (m-2-1)
										(m-1-2) edge node[left] {$j$} (m-2-2)
										(m-2-3) edge node[right] {$d$} (m-3-3)
										(m-2-1) edge node[left] {$r$} (m-3-1)
										(m-2-2) edge[barred] node[below] {$j^*$} (m-2-3)
														edge node[left] {$r$} (m-3-2);
				\path				(m-1-2) edge[eq] (m-1-3)
										(m-1-3) edge[eq] (m-2-3)
										(m-2-1) edge[eq] (m-2-2)
										(m-3-1) edge[eq] (m-3-2)
										(m-3-2) edge[eq] (m-3-3);
				\path[transform canvas={shift={($(m-2-2)!0.5!(m-2-3)$)}, yshift=-2pt}]	(m-2-2) edge[cell] node[right] {$\eps'$} (m-3-2);
				\draw				($(m-1-1)!0.5!(m-2-2)$)	node {\cc}
										($(m-1-2)!0.5!(m-2-3)$) node {\oc};
			\end{tikzpicture} = \begin{tikzpicture}[textbaseline]
				\matrix(m)[math175em]{C & A \\ B & A \\ M & M \\};
				\path[map]	(m-1-1) edge[barred] node[above] {$B(f, j)$} (m-1-2)
														edge node[left] {$f$} (m-2-1)
										(m-2-1) edge[barred] node[below] {$j^*$} (m-2-2)
														edge node[left] {$r$} (m-3-1)
										(m-2-2) edge node[right] {$d$} (m-3-2);
				\path				(m-1-2) edge[eq] (m-2-2)
										(m-3-1) edge[eq] (m-3-2);
				\path[transform canvas={shift=(m-2-2), yshift=-2pt}]	(m-2-1) edge[cell] node[right] {$\eps'$} (m-3-1);
				\draw				($(m-1-1)!0.5!(m-2-2)$)	node {\cc};
			\end{tikzpicture}
		\end{displaymath}
		Since $\K$ has opcartesian tabulations we have (a) $\Leftrightarrow$ (c) in the previous theorem, which means that $\eps'$ defines $r$ as the pointwise right Kan extension of $d$ along $j^*$ if and only if for each $\map fCB$ the composite on the right above, and thus each composite above, defines $r \of f$ as the right Kan extension of $d$ along $B(f, j)$.
		
		Now consider the composite on the left above precomposed with the tabulation of $B(f, j)$, as on the left below, as well as with the factorisation of this tabulation through the opcartesian cell defining $\pi_C^*$, as on the right. Since the two top cells of the first column in the composite on the left define the comma object of $f$ and $j$ in $V(\K)$, by \propref{comma objects in terms of tabulations}, we have denoted the object of the tabulation by $f \slash j$.
		\begin{displaymath}
			\begin{tikzpicture}[baseline]
				\matrix(m)[math175em]{f \slash j & f \slash j & f \slash j \\ C & A & A \\ B & B & \\ M & M & M \\};
				\path[map]	(m-1-1) edge node[left] {$\pi_C$} (m-2-1)
										(m-1-2) edge node[right] {$\pi_A$} (m-2-2)
										(m-1-3) edge node[right] {$\pi_A$} (m-2-3)
										(m-2-1) edge[barred] node[below] {$B(f, j)$} (m-2-2)
														edge node[left] {$f$} (m-3-1)
										(m-2-2) edge node[right] {$j$} (m-3-2)
										(m-2-3) edge node[right] {$d$} (m-4-3)
										(m-3-1) edge node[left] {$r$} (m-4-1)
										(m-3-2) edge node[left] {$r$} (m-4-2);
				\path				(m-1-1) edge[eq] (m-1-2)
										(m-1-2) edge[eq] (m-1-3)
										(m-2-2) edge[eq] (m-2-3)
										(m-3-1) edge[eq] (m-3-2)
										(m-4-1) edge[eq] (m-4-2)
										(m-4-2) edge[eq] (m-4-3);
				\path[transform canvas={shift={($(m-3-2)!0.5!(m-3-3)$)}}]	(m-2-2) edge[cell] node[right] {$\eps$} (m-3-2);
				\draw				($(m-1-1)!0.5!(m-2-2)$) node {\tc}
										($(m-2-1)!0.5!(m-3-2)$)	node[, yshift=-1pt] {\cc};
			\end{tikzpicture} \qquad\qquad\qquad \begin{tikzpicture}[baseline]
				\matrix(m)[math175em]{C &[0.2em] f \slash j & f \slash j \\ C & A & A \\ B & B & \\ M & M & M \\};
				\path[map]	(m-1-1) edge[barred] node[above] {$\pi_C^*$} (m-1-2)
										(m-1-2) edge node[right] {$\pi_A$} (m-2-2)
										(m-1-3) edge node[right] {$\pi_A$} (m-2-3)
										(m-2-1) edge[barred] node[below] {$B(f, j)$} (m-2-2)
														edge node[left] {$f$} (m-3-1)
										(m-2-2) edge node[right] {$j$} (m-3-2)
										(m-2-3) edge node[right] {$d$} (m-4-3)
										(m-3-1) edge node[left] {$r$} (m-4-1)
										(m-3-2) edge node[left] {$r$} (m-4-2);
				\path				(m-1-1) edge[eq] (m-2-1)
										(m-1-2) edge[eq] (m-1-3)
										(m-2-2) edge[eq] (m-2-3)
										(m-3-1) edge[eq] (m-3-2)
										(m-4-1) edge[eq] (m-4-2)
										(m-4-2) edge[eq] (m-4-3);
				\path[transform canvas={shift={($(m-3-2)!0.5!(m-3-3)$)}}]	(m-2-2) edge[cell] node[right] {$\eps$} (m-3-2);
				\draw				($(m-1-1)!0.5!(m-2-2)$) node {\tc$'$}
										($(m-2-1)!0.5!(m-3-2)$)	node[, yshift=-1pt] {\cc};
			\end{tikzpicture}
		\end{displaymath}
		The proof is completed as follows: we have seen that $\eps'$ defines $r$ as the pointwise right Kan extension of $d$ along $j^*$ in $\K$ if and only if the composite of the two bottom rows on the right above defines $r \of f$ as the right Kan extension of $d$ along $B(f, j)$, for each $\map fCA$. By \propref{right Kan extensions along conjoints} the latter is equivalent to the full composite on the right defining $r \of f$ as the right Kan extension of $d \of \pi_A$ along $\pi_C^*$, for each $\map fCA$. This in turn is equivalent to, for each $\map fCA$, the composite on the left defining $r \of f$ as the right Kan extension of $d \of \pi_A$ along $\pi_C$ in the $2$-category $V(\K)$, by \propref{right Kan extensions along conjoints as right Kan extensions in V(K)}. As a condition on the vertical cell $\eps$ in $V(\K)$, the latter is precisely the condition given by \defref{definition: pointwise left Kan extensions in 2-categories}.
	\end{proof}
	
	Combining the previous result with \propref{right Kan extensions along conjoints}, we see that in an equipment $\K$ with opcartesian tabulations all pointwise Kan extensions can be regarded as pointwise Kan extensions in $V(\K)$, as follows.
	\begin{corollary} \label{pointwise Kan extensions as pointwise vertical Kan extensions}
		Consider a cell $\eps$, as in the composite below, in an equipment $\K$ that has opcartesian tabulations. It defines $r$ as the pointwise right Kan extension of $d$ along $J$ in $\K$ precisely if the full composite defines $r$ as the pointwise right Kan extension of $d \of \pi_B$ along $\pi_A$ in $V(\K)$.
		\begin{displaymath}
			\begin{tikzpicture}
				\matrix(m)[tab]{\gen J & \gen J \\ A & B \\ M & M \\};
				\path[map]	(m-1-1) edge node[left] {$\pi_A$} (m-2-1)
										(m-1-2) edge node[right] {$\pi_B$} (m-2-2)
										(m-2-1) edge[barred] node[below] {$J$} (m-2-2)
														edge node[left] {$r$} (m-3-1)
										(m-2-2) edge node[right] {$d$} (m-3-2);
				\path				(m-1-1) edge[eq] (m-1-2)
										(m-3-1) edge[eq] (m-3-2);
				\path[transform canvas={shift=(m-2-2), yshift=-2pt}]	(m-2-1) edge[cell] node[right] {$\eps$} (m-3-1);
				\draw				($(m-1-1)!0.5!(m-2-2)$) node {\tc};
			\end{tikzpicture}
		\end{displaymath}
	\end{corollary}
	
	\subsection{Opcartesian tabulations for internal profunctors.}
	We close this paper by showing that internal profunctors admit opcartesian tabulations. A weaker version of this result is given by Betti in \cite{Betti96}, whose notion of tabulation is a variant of the $1$\ndash dimensional property of \defref{definition: tabulation}, which is weakened by only requiring that $\pi \of 1_\phi' = \phi$ holds up to invertible vertical cells $\phi_A \iso \pi_A \of \phi'$ and $\phi_B \iso \pi_B \of \phi'$. The proof given below shows that the tabulations in $\inProf\E$ can be chosen such that both universal properties of \defref{definition: tabulation} hold, in the strict sense.
	\begin{proposition}[Betti] \label{opcartesian tabulations for internal profunctors}
		Let $\E$ be a category with pullbacks, and coequalisers that are preserved by pullbacks. The equipment $\inProf\E$ of profunctors internal in $\E$ admits opcartesian tabulations.
	\end{proposition}
	\begin{proof}
		Consider an internal profunctor $\hmap{J = (J, l, r)}AB$ between internal categories $A = (A, m_A, e_A)$ and $B = (B, m_B, e_B)$. We will define a cell $\pi$, as on the left below, and show that it statisfies the universal properties and that it is opcartesian.
		\begin{equation} \label{internal tabulation pullback}
			\begin{tikzpicture}[textbaseline]
  			\matrix(m)[tab]{\tab J & \tab J \\ A & B \\};
  			\path[map]  (m-1-1) edge node[left] {$\pi_A$} (m-2-1)
										(m-1-2) edge node[right] {$\pi_B$} (m-2-2)
										(m-2-1) edge[barred] node[below] {$J$} (m-2-2);
				\path				(m-1-1) edge[eq] (m-1-2);
				\path[transform canvas={shift={($(m-1-2)!(0,0)!(m-2-2)$)}}] (m-1-1) edge[cell] node[right] {$\pi$} (m-2-1);
			\end{tikzpicture} \qquad\qquad\qquad\qquad \begin{tikzpicture}[textbaseline]
			\matrix(m)[math2em, column sep=0.75em]{\gen J & J \times_{B_0} B \\ A \times_{A_0} J & J \\};
				\path[map]	(m-1-1)	edge (m-1-2)
														edge (m-2-1)
										(m-1-2) edge node[right] {$r$} (m-2-2)
										(m-2-1) edge node[below] {$l$} (m-2-2);
				\coordinate (hook) at ($(m-1-1)+(0.4,-0.4)$);
				\draw (hook)+(0,0.17) -- (hook) -- +(-0.17,0);
			\end{tikzpicture}
		\end{equation}
		\emph{Definition of $\pi$.} We first define the internal category $\gen J$. As its underlying $\E$-span we take $\gen J_0 = J$ and let $\gen J$ be the pullback of the actions of $B$ and $A$ on $J$, as on the right above, while the source and target maps are given by the projections below. Notice that these are the internal versions of the sets of objects and morphisms that underlie the tabulation of an ordinary profunctor, as described in \exref{example: tabulations of profunctors}.
		\begin{displaymath}
			d_0 = \bigbrks{\gen J \to J \times_{B_0} B \to J} \qquad \qquad \qquad d_1 = \bigbrks{\gen J \to A \times_{A_0} J \to J}
		\end{displaymath}
		
		To give the multiplication $\map m{\gen J \times_J \gen J}{\gen J}$, notice that its source $W = \gen J \times_J \gen J$ forms the wide pullback of the diagram
		\begin{equation} \label{multiplication of internal tabulation}
			\begin{tikzpicture}[textbaseline]
				\matrix(m)[math175em, column sep=1em]
				{	J \times_{B_0} B & & A \times_{A_0} J & & J \times_{B_0} B & & A \times_{A_0} J, \\
					& J & & J & & J & \\}; 
				\path[map]	(m-1-1) edge node[below left] {$r$} (m-2-2)
										(m-1-3) edge node[below right] {$l$} (m-2-2)
														edge (m-2-4)
										(m-1-5) edge (m-2-4)
														edge node[below left] {$r$} (m-2-6)
										(m-1-7) edge node[below right] {$l$} (m-2-6);
				\draw[font=\scriptsize]	(m-1-1) node[above=3pt] {(1)}
																(m-1-3) node[above=3pt] {(2)}
																(m-1-5) node[above=3pt] {(3)}
																(m-1-7) node[above=3pt] {(4)};
			\end{tikzpicture}
		\end{equation}
		where the first factor of $W$ is the pullback of the objects labelled (1) and (2), the second factor is the pullback of those labelled (3) and (4), and where the unlabelled maps are projections. In the case of $\E = \Set$, $W$ consists of quadruples of pairs of maps, one in each of the sets on the top row above, such that the pairs in (1) and (2), as well as those in (3) and (4), form commutative squares \eqref{morphism in a double comma category}, and such that the bottom map of the first square coincides with the top map of the second. Let $m$ be induced by the compositions
		\begin{flalign*}
			&& W &\to \overbrace{J \times_{B_0} B}^{(1)} \times_{B_0} \overset{\raisebox{0.75ex}{\scriptsize{(3)}}}B \xrar{\id \times m_B} J \times_{B_0} B & \\
			\text{and} && W &\to \overset{\raisebox{0.75ex}{\scriptsize{(2)}}}A \times_{A_0} \overbrace{A \times_{A_0} J}^{(4)} \xrar{m_A \times \id} A \times_{A_0} J, &
		\end{flalign*}
		where the unlabelled maps are induced by projections from $W$ onto (factors of) the pullbacks in the top row of \eqref{multiplication of internal tabulation}, as indicated by the labels. When $\E = \Set$, this simply composes the two pairs of maps making up the sides of two adjacent commutative squares \eqref{morphism in a double comma category}. That the maps above indeed induce a map into the pullback $\gen J$ follows from
		\begin{align*}
			\bigbrks{&W \to \overbrace{J \times_{B_0} B}^{(1)} \times_{B_0} \overset{\raisebox{0.75ex}{\scriptsize{(3)}}}B \xrar{\id \times m_B} J \times_{B_0} B \xrar r J} \\
			=\bigbrks{&W \to \overbrace{A \times_{A_0} J}^{(2)} \times_{B_0} \overset{\raisebox{0.75ex}{\scriptsize{(3)}}}B \xrar{l \times \id} J \times_{B_0} B \xrar r J} \\
			=\bigbrks{&W \to \overset{\raisebox{0.75ex}{\scriptsize{(2)}}}A \times_{A_0} \overbrace{J \times_{B_0} B}^{(3)} \xrar{\id \times r} A \times_{A_0} J \xrar l J} \\ 
			=\bigbrks{&W \to \overset{\raisebox{0.75ex}{\scriptsize{(2)}}}A \times_{A_0} \overbrace{A \times_{A_0} J}^{(4)} \xrar{m_A \times \id} A \times_{A_0} J \xrar l J}.
		\end{align*}
		Here the first equality follows from the fact that we can replace $\id \times m_B$ by $r \times \id$ (by associativity of the action of $B$ on $J$) and that precomposing $r$ with the projection onto (1) coincides with precomposing $l$ with the projection onto (2), by definition of $W$. Likewise the projection from $W$ onto the factor $J$ of (2) equals the projection onto the same factor of (3) and, together with the fact that $r$ and $l$ commute, the second equality follows. The third equality is symmetric to the first.
	
		That the multiplication $\map m{\gen J \times_J \gen J}{\gen J}$ thus defined is compatible with the source and target maps is clear; that it is associative follows from the associativity of $m_A$ and $m_B$. Moreover one readily verifies that the unit $\map{e_{\gen J}} J{\gen J}$ can be taken to be induced by the $\E$-maps
		\begin{displaymath}
			J \xrar{(\id, e_B \of d_1)} J \times_{B_0} B \qquad \text{and} \qquad J \xrar{(e_A \of d_0, \id)} A \times_{A_0} J,
		\end{displaymath}
		which completes the definition of the structure that makes $\gen J$ into a category internal in $\E$. Next let the projection $\map{\pi_A}{\gen J}A$ be defined by the maps
		\begin{displaymath}
			\map{(\pi_A)_0 = d_0}J{A_0} \qquad \text{and} \qquad \pi_A = \bigbrks{\gen J \to A \times_{A_0} J \to A};
		\end{displaymath}
		the projection $\pi_B$ onto $B$ is given analogously. One readily checks that these maps are compatible with the structures on $\gen J$, $A$ and $B$. Finally we define the cell $\pi$, which is given by a map $\gen J \to J$ in $\E$, to be the diagonal of the square defining $\gen J$, on the right of \eqref{internal tabulation pullback}. That it is natural follows from the associativity of the actions of $A$ and $B$ on $J$. This completes the definition of the internal transformation $\pi$.
		
		\emph{Universal properties of $\pi$.} To see that $\pi$ satisfies the $1$-dimensional property of \defref{definition: tabulation} we consider an internal transformation $\phi$ as on the left below; we have to show that it factors uniquely through $\pi$ as an internal functor $\map{\phi'}X{\gen J}$.
		\begin{displaymath}
			\begin{tikzpicture}
  			\matrix(m)[math175em]{X & X \\ A & B \\};
  			\path[map]  (m-1-1) edge node[left] {$\phi_A$} (m-2-1)
										(m-1-2) edge node[right] {$\phi_B$} (m-2-2)
										(m-2-1) edge[barred] node[below] {$J$} (m-2-2);
				\path				(m-1-1) edge[eq] (m-1-2);
				\path[transform canvas={shift={($(m-1-2)!(0,0)!(m-2-2)$)}}] (m-1-1) edge[cell] node[right] {$\phi$} (m-2-1);
			\end{tikzpicture}
		\end{displaymath}
		Such an internal functor $\phi'$ is given by $\E$-maps $\map{\phi'_0}{X_0}J$ and $\map{\phi'}X{\gen J}$. We take $\phi'_0 = \bigbrks{X_0 \xrar{e_X} X \xrar{\phi} J}$ and let $\phi'$ be induced by
		\begin{equation} \label{1-dimensional universal property}
			X \xrar{(\phi'_0 \of d_0, \phi_B)} J \times_{B_0} B \qquad \text{and} \qquad X \xrar{(\phi_A, \phi'_0 \of d_1)} A \times_{A_0} J;
		\end{equation}
		the naturality of $\phi$ implies that both these maps, when composed with $r$ and $l$ respectively, equal $\phi$, so that they induce a map $\map{\phi'}X{\gen J}$ satisfying $\pi \of \phi' = \phi$. It is clear that $\phi'_0$ and $\phi$ are compatible with the source and target maps; that they are compatible with the multiplication and units of $A$ and $B$ follows from the fact that $\phi_A$ and $\phi_B$ are.
		
		We have already seen that $\pi \of \phi' = \phi$, and it is easily checked that $\pi_A \of \phi' = \phi_A$ and $\pi_B \of \phi' = \phi_B$ as well, so that $\phi$ factors through $\pi$ as $\phi'$. To see that $\phi'$ is unique notice that
		\begin{displaymath}
			\phi \of e_X = \pi \of \phi' \of e_X = \pi \of e_{\gen J} \of \phi'_0 = \phi'_0
		\end{displaymath}
		where the first follows from the factorisation $\phi = \pi \of \phi'$, the second from the unit axiom for $\phi'$ and the last from the definition of $e_{\gen J}$ and $\pi$. This shows that $\phi'_0$ is uniquely determined; that $\phi'$ must be determined by the maps \eqref{1-dimensional universal property} as well follows from $\pi_A \of \phi' = \phi_A$, $\pi_B \of \phi' = \phi_B$, and the compatibility of $\phi'$ and $\phi'_0$ with the source and target maps. This concludes the proof of the $1$-dimensional universal property of $\pi$.
		
		For the $2$-dimensional property of \defref{definition: tabulation} consider the identity of internal transformations as on the left below.
		\begin{displaymath}
			\begin{tikzpicture}[textbaseline]
				\matrix(m)[math175em]{X & Y & Y \\ A & A & B \\};
				\path[map]	(m-1-1) edge[barred] node[above] {$H$} (m-1-2)
														edge node[left] {$\phi_A$} (m-2-1)
										(m-1-2) edge node[right, inner sep=1pt] {$\psi_A$} (m-2-2)
										(m-1-3) edge node[right] {$\psi_B$} (m-2-3)
										(m-2-2) edge[barred] node[below] {$J$} (m-2-3);
				\path				(m-1-2) edge[eq] (m-1-3)
										(m-2-1) edge[eq] (m-2-2);
				\path[transform canvas={shift={($(m-1-2)!0.5!(m-2-3)$)}}] (m-1-1) edge[cell] node[right] {$\xi_A$} (m-2-1)
										(m-1-2) edge[cell] node[right] {$\psi$} (m-2-2);
			\end{tikzpicture} = \begin{tikzpicture}[textbaseline]
				\matrix(m)[math175em]{X & X & Y \\ A & B & B \\};
				\path[map]	(m-1-1) edge node[left] {$\phi_A$} (m-2-1)
										(m-1-2) edge[barred] node[above] {$H$} (m-1-3)
														edge node[right, inner sep=1pt] {$\phi_B$} (m-2-2)
										(m-1-3) edge node[right] {$\psi_B$} (m-2-3)
										(m-2-1) edge[barred] node[below] {$J$} (m-2-2);
				\path				(m-1-1) edge[eq] (m-1-2)
										(m-2-2) edge[eq] (m-2-3);
				\path[transform canvas={shift={($(m-1-2)!0.5!(m-2-3)$)}}] (m-1-1) edge[cell] node[right] {$\phi$} (m-2-1)
										(m-1-2) edge[cell] node[right] {$\xi_B$} (m-2-2);
			\end{tikzpicture} \qquad \qquad \begin{tikzpicture}[textbaseline]
  			\matrix(m)[math175em]{X & Y \\ \tab J & \tab J \\};
  			\path[map]  (m-1-1) edge[barred] node[above] {$H$} (m-1-2)
														edge node[left] {$\phi'$} (m-2-1)
										(m-1-2) edge node[right] {$\psi'$} (m-2-2);
				\path				(m-2-1) edge[eq] (m-2-2);
				\path[transform canvas={shift={($(m-1-2)!(0,0)!(m-2-2)$)}}] (m-1-1) edge[cell] node[right] {$\xi'$} (m-2-1);
			\end{tikzpicture}
		\end{displaymath}
		We have to construct an internal transformation $\xi'$ as on the right above, given by an $\E$-map $\map{\xi'}H{\gen J}$, such that $\pi_A \of \xi' = \xi_A$ and $\pi_B \of \xi' = \xi_B$. Together with the fact that $\xi'$ must be compatible with source and target maps, these identities completely determine $\xi'$ as being induced by the maps
		\begin{displaymath}
			H \xrar{(\phi'_0 \of d_0, \xi_B)} J \times_{B_0} B \qquad \text{and} \qquad H \xrar{(\xi_A, \psi'_0 \of d_1)} A \times_{A_0} J;
		\end{displaymath}
		that these maps coincide after composition with $r$ and $l$ respectively follows from the identity on the left above. Naturality of $\xi'$ follows from that of $\xi_A$ and $\xi_B$. This completes the proof of the universal properties of $\pi$.
		 
		\emph{The cell $\pi$ is opcartesian.} To show that $\pi$ is opcartesian we have to show that every internal transformation $\chi$ as on the left below factors uniquely as shown.
		\begin{displaymath}
		 	\begin{tikzpicture}[textbaseline]
    		\matrix(m)[math175em]{\gen J & \gen J \\ A & B \\ C & D \\};
    		\path[map]  (m-1-1) edge node[left] {$\pi_A$} (m-2-1)
        		        (m-1-2) edge node[right] {$\pi_B$} (m-2-2)
        		        (m-2-1) edge node[left] {$f$} (m-3-1)
        		        (m-2-2) edge node[right] {$g$} (m-3-2)
        		        (m-3-1) edge[barred] node[below] {$K$} (m-3-2);
       	\path				(m-1-1) edge[eq] (m-1-2);
    		\path[transform canvas={shift={($(m-2-1)!0.5!(m-1-1)$)}}] (m-2-2) edge[cell] node[right] {$\chi$} (m-3-2);
  		\end{tikzpicture} = \begin{tikzpicture}[textbaseline]
    		\matrix(m)[math175em]{\gen J & \gen J \\ A & B \\ C & D \\};
    		\path[map]  (m-1-1) edge node[left] {$\pi_A$} (m-2-1)
            		    (m-1-2) edge node[right] {$\pi_B$} (m-2-2)
            		    (m-2-1) edge[barred] node[below] {$J$} (m-2-2)
            		            edge node[left] {$f$} (m-3-1)
            		    (m-2-2) edge node[right] {$g$} (m-3-2)
            		    (m-3-1) edge[barred] node[below] {$K$} (m-3-2);
        \path				(m-1-1) edge[eq] (m-1-2);
    		\path[transform canvas={shift=(m-2-1))}]
        		        (m-1-2) edge[cell] node[right] {$\pi$} (m-2-2)
        		        (m-2-2) edge[transform canvas={yshift=-0.3em}, cell] node[right] {$\chi'$} (m-3-2);
  		\end{tikzpicture}
		\end{displaymath}
		The internal transformation $\chi'$ is given by an $\E$-map $\map{\chi'}JK$, and precomposing the identity above with the unit $\map{e_{\gen J}}J{\gen J}$, for which $\pi \of e_{\gen J} = \id_J$ holds, gives $\chi \of e_{\gen J} = \chi'$. Thus the identity above determines $\chi'$ as $\chi' = \chi \of e_{\gen J} = \chi_0$, where $\chi_0$ is as in \lemref{vertical internal transformations}. That $\chi' = \chi_0$ is natural with respect to the action of $A$ is shown by the following equation.
		\begin{align*}
			\bigbrks{&A \times_{A_0} J \xrar{f \times_{f_0} \chi_0} C \times_{C_0} K \xrar l K} \\
				= \bigbrks{&A \times_{A_0} J \xrar{((l, e_B \of d_1), \id)} \gen J \xrar{(f \of \pi_A, \chi_0 \of d_1)} C \times_{C_0} K \xrar l K} \\
				= \bigbrks{&A \times_{A_0} J \xrar{((l, e_B \of d_1), \id)} \gen J \xrar{(\chi_0 \of d_0, g \of \pi_B)} K \times_{D_0} D \xrar r K} \\
				= \bigbrks{&A \times_{A_0} J \xrar{(\chi_0 \of l, g \of e_B \of d_1)} K \times_{D_0} D \xrar r K} \\ 
				= \brks{&A \times_{A_0} J \xrar l J \xrar{\chi_0} K}
		\end{align*}
		The first and third identities here are straightforward to check, while the second and fourth follow from the naturality of $\chi_0$ (see \lemref{vertical internal transformations}) and the unit axiom for $r$ respectively. That $\chi'$ is natural with respect to the action of $B$ is shown similarly. This concludes the proof.
	\end{proof}


\begin{thebibliography}{Woo82}

\bibitem[Bet96]{Betti96}
R.~Betti.
\newblock Formal theory of internal categories.
\newblock {\em Le Matematiche}, 51(3):35--52, 1996.

\bibitem[Dub70]{Dubuc70}
E.~J. Dubuc.
\newblock {\em Kan extensions in enriched category theory}, volume 145 of {\em
  Lecture Notes in Mathematics}.
\newblock Springer-Verlag, Berlin, 1970.

\bibitem[GP99]{Grandis-Pare99}
M.~Grandis and R.~Par{\'e}.
\newblock Limits in double categories.
\newblock {\em Cahiers de Topologie et G\'eom\'etrie Diff\'erentielle
  Cat\'egoriques}, 40(3):162--220, 1999.

\bibitem[GP08]{Grandis-Pare08}
M.~Grandis and R.~Par{\'e}.
\newblock Kan extensions in double categories (on weak double categories, part
  {III}).
\newblock {\em Theory and Applications of Categories}, 20(8):152--185, 2008.

\bibitem[Gui80]{Guitart80}
R.~Guitart.
\newblock Relations et carr\'es exacts.
\newblock {\em Les Annales des Sciences Math\'ematiques du Qu\'ebec},
  4(2):103--125, 1980.

\bibitem[Kel82]{Kelly82}
G.~M. Kelly.
\newblock {\em Basic concepts of enriched category theory}, volume~64 of {\em
  London Mathematical Society Lecture Note Series}.
\newblock Cambridge University Press, Cambridge, 1982.

\bibitem[Kou13]{Koudenburg13}
S.~R. Koudenburg.
\newblock {\em Algebraic weighted colimits}.
\newblock PhD thesis, University of Sheffield, 2013.

\bibitem[Kou14]{Koudenburg14b}
S.~R. Koudenburg.
\newblock Algebraic {K}an extensions in double categories.
\newblock Preprint, available as
  \href{http://arxiv.org/abs/1406.6994}{\texttt{arXiv:1406.6994}}, 2014.

\bibitem[KS06]{Kashiwara-Schapira06}
M.~Kashiwara and P.~Schapira.
\newblock {\em Categories and sheaves}, volume 332 of {\em Grundlehren der
  Mathematischen Wissenschaften}.
\newblock Springer-Verlag, Berlin, 2006.

\bibitem[ML98]{MacLane98}
S.~Mac~Lane.
\newblock {\em Categories for the working mathematician}, volume~5 of {\em
  Graduate Texts in Mathematics}.
\newblock Springer-Verlag, New York, second edition, 1998.

\bibitem[Shu08]{Shulman08}
M.~A. Shulman.
\newblock Framed bicategories and monoidal fibrations.
\newblock {\em Theory and Applications of Categories}, 20(18):650--738, 2008.

\bibitem[Str74]{Street74}
R.~Street.
\newblock Fibrations and {Y}oneda's lemma in a {$2$}-category.
\newblock In {\em Proceedings of the {S}ydney {C}ategory {T}heory {S}eminar,
  1972/1973}, volume 420 of {\em Lecture Notes in Mathematics}, pages 104--133.
  Springer, Berlin, 1974.

\bibitem[Woo82]{Wood82}
R.~J. Wood.
\newblock Abstract proarrows {I}.
\newblock {\em Cahiers de Topologie et G\'eom\'etrie Diff\'erentielle},
  23(3):279--290, 1982.

\end{thebibliography}
\end{document}